\documentclass{amsart}

\usepackage{stmaryrd}
\usepackage{enumerate}
\usepackage[pdftex]{graphicx}
\usepackage{amsmath}
\usepackage{mathrsfs}
\usepackage{amsthm}

\newtheorem{thm}{Theorem}[section]
\newtheorem{prop}[thm]{Proposition}
\newtheorem{lem}[thm]{Lemma}
\newtheorem{cor}[thm]{Corollary}

\theoremstyle{definition}
\newtheorem{definition}[thm]{Definition}

\theoremstyle{remark}
\newtheorem{notation}[thm]{Notation}

\numberwithin{equation}{section}

\newcommand{\Span}{\mathrm{Span}} 
\newcommand{\suc}{\mathrm{suc}} 

\newcommand{\id}{\mathrm{id}}
\newcommand{\Ad}{\mathrm{Ad}}

\newcommand{\Hit}{\mathrm{Hit}}
\newcommand{\PSL}{\mathrm{PSL}}
\newcommand{\SO}{\mathrm{SO}}
\newcommand{\PGL}{\mathrm{PGL}}
\newcommand{\GL}{\mathrm{GL}}
\newcommand{\SL}{\mathrm{SL}}
\newcommand{\Gr}{\mathrm{Gr}}
\newcommand{\PSp}{\mathrm{PSp}}
\newcommand{\Sp}{\mathrm{Sp}}
\newcommand{\Stab}{\mathrm{Stab}} 

\newcommand{\Amc}{\mathcal{A}} 
\newcommand{\Bmc}{\mathcal{B}} 
\newcommand{\Cmc}{\mathcal{C}} 
\newcommand{\Dmc}{\mathcal{D}}

\newcommand{\Gmc}{\mathcal{G}} 
 
\newcommand{\Lmc}{\mathcal L}
\newcommand{\Mmc}{\mathcal{M}} 
\newcommand{\Nmc}{\mathcal{N}} 
\newcommand{\Pmc}{\mathcal{P}} 
\newcommand{\Qmc}{\mathcal{Q}} 
\newcommand{\Rmc}{\mathcal{R}} 
\newcommand{\Smc}{\mathcal{S}} 
\newcommand{\Tmc}{\mathcal{T}}
\newcommand{\Xmc}{\mathcal{X}} 
\newcommand{\Zmc}{\mathcal{Z}}

\newcommand{\Dbbb}{\mathbb{D}} 
 
\newcommand{\Nbbb}{\mathbb N} 
\newcommand{\Pbbb}{\mathbb{P}} 
\newcommand{\Rbbb}{\mathbb{R}} 
\newcommand{\Sbbb}{\mathbb{S}} 
 
\newcommand{\Zbbb}{\mathbb{Z}}

\newcommand{\Std}{\widetilde{S}}

\newcommand{\amf}{\mathfrak{a}}
\newcommand{\gmf}{\mathfrak{g}}
\newcommand{\kmf}{\mathfrak{k}}
\newcommand{\lmf}{\mathfrak{l}}
\newcommand{\omf}{\mathfrak{o}}
\newcommand{\pmf}{\mathfrak{p}}
\newcommand{\smf}{\mathfrak{s}}

\newcommand{\btd}{\widetilde{b}} 
\newcommand{\ctd}{\widetilde{c}} 
\newcommand{\etd}{\widetilde{e}} 
\newcommand{\ptd}{\widetilde{p}} 
\newcommand{\qtd}{\widetilde{q}}

\begin{document}

\title{Positively ratioed representations}

\author{Giuseppe Martone, Tengren Zhang}
\thanks{\\
G. Martone was partially supported by the NSF grant DMS-1406559.\\
T. Zhang was partially supported by the NSF grants DMS-1536017 and DMS-1566585.\\
The authors gratefully acknowledge support by NSF grants DMS-1107452, 1107263 and 1107367 ``RNMS: GEometric structures And Representation varieties'' (the GEAR Network).
}

\maketitle

\begin{abstract}
Let $S$ be a closed orientable surface of genus at least 2 and let $G$ be a semisimple real algebraic group of non-compact type. We consider a class of representations from the fundamental group of $S$ to $G$ called positively ratioed representations. These are Anosov representations with the additional condition that certain associated cross ratios satisfy a positivity property. Examples of such representations include Hitchin representations and maximal representations. Using geodesic currents, we show that the corresponding length functions for these positively ratioed representations are well-behaved. In particular, we prove a systolic inequality that holds for all such positively ratioed representations.
\end{abstract}

\section{Introduction}

Let $S$ be a closed, oriented, connected surface of genus at least $2$ with fundamental group $\Gamma$. The Teichm\"uller space of $S$, denoted $\Tmc(S)$, is the deformation space of hyperbolic structures on $S$. Via the holonomy, one can also think of $\Tmc(S)$ as a connected component of the space
\[\Xmc(\Gamma,\PSL(2,\Rbbb)):=\mathrm{Hom}(\Gamma,\PSL(2,\Rbbb))/\PGL(2,\Rbbb).\]
The representations in $\Tmc(S)$ can be characterized as the ones that are $[P]$-Anosov, where $P$ is the unique (up to conjugation) parabolic subgroup of $\PSL(2,\Rbbb)$. Let $\Cmc\Gmc(S)$ denote the set of free homotopy classes of closed curves (see also Definition \ref{closed geodesic}). Every hyperbolic structure $\rho\in\Tmc(S)$ induces a length function $\ell^\rho\colon\Cmc\Gmc(S)\to \Rbbb$ which associates to $c\in\Cmc\Gmc(S)$ the hyperbolic length, with respect to $\rho$, of the geodesic representative of $c$.

A \emph{geodesic current} on $S$ is a locally finite, $\Gamma$-invariant, Borel measure on the set of geodesics in the universal cover of $S$. Observe that the space of geodesic currents on $S$, denoted $\Cmc(S)$, is an open convex cone in an infinite dimensional vector space. Furthermore, $\Cmc\Gmc(S)$ can be identified with a subset of $\Cmc(S)$ (see Section \ref{Cross ratios and positively ratioed}). Bonahon \cite{Bon1} showed that $\Cmc(S)$ is naturally equipped with a continuous, bilinear \emph{intersection pairing} 
\[i:\Cmc(S)\times\Cmc(S)\to\Rbbb^+\cup \{0\}\]
which generalizes the geometric intersection number between free homotopy classes of closed curves in $\Cmc\Gmc(S)$. Also, he proved that for every hyperbolic structure $\rho\in\Tmc(S)$, there is a unique geodesic current $\mu^\rho\in\Cmc(S)$ with the property that for any $c\in\Cmc\Gmc(S)$, 
\[i(\mu^\rho,c)=\ell^\rho(c).\] 
The geodesic current $\mu^\rho$ is known as the \emph{intersection current} associated to $\rho$.

In this paper, we investigate the extent to which we can generalize this intersection current to the setting of $[P]$-Anosov representations $\rho:\Gamma\to G$, where $G$ is a non-compact semisimple, real algebraic group and $[P]$ is the conjugacy class of a parabolic subgroup $P\subset G$. Every conjugacy class of parabolic subgroups of $G$ determines a subset $\theta$ of the set of restricted simple roots $\Delta$ of $G$. We will assume, without loss of generality, that $\theta=\iota(\theta)$, where $\iota$ is the \emph{opposition involution} on $\Delta$ (see Sections  \ref{background} and \ref{Anosov representations}).

For each $\alpha\in\theta$, the corresponding restricted fundamental weight $\omega_\alpha$ allows us to define a length function 
\[\ell^\rho_\alpha:\Cmc\Gmc(S)\to\Rbbb^+\cup\{0\}\] 
for $\rho$, which generalizes the length function associated to a hyperbolic structure in $\Tmc(S)$. However, it is not true in general that there is a geodesic current $\nu$ so that $i(\nu,c)=\ell^\rho_\alpha(c)$ for every $c\in\Cmc\Gmc(S)$.

As such, we introduce the notion of a \emph{$[P]$-positively ratioed representation}. These are $[P]$-Anosov representations with the additional property that certain cross ratios associated to $\omega_\alpha$ for all $\alpha\in\theta$ are always positive (see Section \ref{Cross ratios and positively ratioed} for more details). Examples include $\PSL(n,\Rbbb)$-Hitchin representations and $\PSp(2n,\Rbbb)$-maximal representations. Combining the work of Hamenst\"adt \cite{Ham1},\cite{Ham2}, Otal \cite{Ota1} and Tits \cite{Tit1}, we have the following theorem.

\begin{thm}\label{positively ratioed to geodesic currents}
If $\rho:\Gamma\to G$ is a $[P]$-positively ratioed representation, then for any $\alpha\in\theta$, there is a unique geodesic current $\mu^\rho_\alpha$ so that $i(\mu^\rho_\alpha,c)=\ell^\rho_\alpha(c)$ for all $c\in\Cmc\Gmc(S)$.
\end{thm}

By Theorem \ref{positively ratioed to geodesic currents}, to prove statements about $\ell^\rho_\alpha$, one needs only to prove the analogous statements in the setting of geodesic currents. Using this strategy, we prove the remaining results in this paper. In fact, all the results in this introduction can be stated in the more technical language of \emph{period minimizing} geodesic currents with full support. These are geodesic currents with full support that satisfy the property that the number of closed geodesics $c\in\Cmc\Gmc(S)$ so that $\ell_\nu(c):=i(c,\nu)<T$ is finite for all $T\in\Rbbb^+$. However, to emphasize the application we are interested in, we will state most of our results for positively ratioed representations in the introduction, and indicate the numbering of the analogous statement about geodesic currents in parenthesis.

We will need the following notation. For an essential subsurface $S'\subset S$, i.e. an incompressible subsurface with negative Euler characteristic, denote by $\Cmc\Gmc(S')$ the set of free homotopy classes of unoriented closed curves in $S'$. Notice that $S'$ is an orientable surface of genus $g'$ with $n'$ boundary components so that $2g'-2+n'>0$.  

The main point of this paper is that Theorem \ref{positively ratioed to geodesic currents} can be exploited to study the length functions of positively ratioed representations. As a first example, we have the following corollary about the asymptotic behavior of length functions along a sequence of positively ratioed representations. This was motivated by the work of Burger-Pozzetti \cite{BurPoz1}.

\begin{cor}[Proposition \ref{asymptotic lengths}]\label{Burger-Pozzetti}
Let $\{\rho_j:\Gamma\to G_j\}_{j=1}^\infty$ be a sequence of $[P_j]$-positively ratioed representations, let $\theta_j$ be a subset of the restricted simple roots of $G_j$ determined by $P_j$, and let $\alpha_j\in\theta_j$. Fix an auxiliary hyperbolic structure on $S$. Then there is 
\begin{itemize}
\item a subsequence of $\{\rho_j\}_{j=1}^\infty$, also denoted $\{\rho_j\}_{j=1}^\infty$,
\item a (possibly disconnected, possibly empty) essential subsurface $S'\subset S$, 
\item a (possibly empty) collection of pairwise non-intersecting, non-peripheral simple closed curves $\{c_1,\dots,c_k\}$ in $\Cmc\Gmc(S\setminus S')$
\end{itemize}
so that $A:=S'\cup\bigcup_{i=1}^kc_i$ is non-empty, and the following holds. Let $c\in\Cmc\Gmc(S)$ be a closed curve so that $c\notin\Cmc\Gmc(S\setminus A)$ and $c$ is not a multiple of $c_i$ for $i=1,\dots,k$.
\begin{enumerate}
	\item If $d\in\Cmc\Gmc(S\setminus A)$ or $d$ is a multiple of $c_i$ for some $i=1,\dots,k$, then $\displaystyle\lim_{j\to\infty}\frac{\ell^{\rho_j}_{\alpha_j}(d)}{\ell^{\rho_j}_{\alpha_j}(c)}=0$.
	\item If $d\in\Cmc\Gmc(S)$ is a closed curve so that $d\notin\Cmc\Gmc(S\setminus A)$ and $d$ is not a multiple of $c_i$ for $i=1,\dots,k$, then $\displaystyle\lim_{j\to\infty}\frac{\ell^{\rho_j}_{\alpha_j}(d)}{\ell^{\rho_j}_{\alpha_j}(c)}\in\Rbbb^+$.
\end{enumerate}
\end{cor}

In the case when $G_j=\PSp(2n,\Rbbb)$ and $\rho_j$ is maximal for all $j$, Corollary \ref{Burger-Pozzetti} is a result of Burger-Pozzetti (Theorem 1.1 of \cite{BurPoz1}). More informally, this corollary states that the closed curves in $S$ whose lengths are growing at the fastest rate along a sequence of positively ratioed representations are exactly those that intersect a particular union of a subsurface of $S$ with a collection of pairwise non-intersecting simple closed curves in $S$.

A second important consequence of Theorem \ref{positively ratioed to geodesic currents} is that the length functions coming from positively ratioed representations behave as if they were the length functions of a negatively curved metric on $S$ when we perform surgery (see Section \ref{surgery and lengths}).

\begin{cor}[Proposition \ref{surgery}]\label{positively ratioed surgery}
Let $\rho:\Gamma\to G$ be $P_\theta$-positively ratioed for some $\theta\subset\Delta$. For $c\in\Cmc\Gmc(S)$ with $i(c,c)>0$, let $c_1,c_2,c_3\in\Cmc\Gmc(S)$ be obtained via surgery at a point of self-intersection of $c$ as in Proposition \ref{surgery}. Then for any $\alpha\in\theta$, we have
\[\ell_\alpha^\rho(c_1)< \ell_\alpha^\rho(c)\,\,\,\text{ and }\,\,\,\ell_\alpha^\rho(c_2)+\ell_\alpha^\rho(c_3)< \ell_\alpha^\rho(c).\]
\end{cor}

For any period minimizing geodesic current $\nu\in\Cmc(S)$, let $\ell_\nu:\Cmc\Gmc(S)\to\Rbbb^+\cup\{0\}$ be the function defined by $\ell_\nu(c):=i(\nu,c)$. Using this, we can define the following three quantities associated to connected essential subsurfaces $S'\subset S$. The first is the \emph{entropy} of $S'$, which is defined to be 
\[h_\nu(S')=h(\ell_\nu,S'):=\limsup_{T\to\infty}\frac{1}{T}\log\#\left\{c\in\Cmc\Gmc(S'):\ell_\nu(c)\leq T\right\},\]
and the second is the \emph{systole length}, which is defined as
\[L_\nu(S')=L(\ell_\nu,S'):=\min\{\ell_\nu(c):c\in\Cmc\Gmc(S')\}.\]
To define the third, one chooses a \emph{minimal pants decomposition} $\Pmc_{\nu,S'}$ of $S'$, i.e. a maximal collection in $\Cmc\Gmc(S')$ of pairwise non-intersecting simple closed geodesics $\{c_1,\dots,c_{3g'-3+2n'}\}$ so that $c_{3g'-2+n'},\dots, c_{3g-3+2n'}$ are the boundary components and for all $j=0,\dots,3g'-4+n'$, $c_{j+1}$ is a non-peripheral systole in $\Cmc\Gmc\left(S'\setminus\bigcup_{i=1}^{j}c_i\right)$. These exists because of Corollary \ref{positively ratioed surgery} (see Section \ref{pants decomp}). The \emph{panted systole length} is then the quantity
\[K_\nu(S')=K(\ell_\nu,S'):=\min\{\ell_\nu(c):c\in\Cmc\Gmc(S')\text{ is not a multiple of a curve in }\Pmc_{\nu,S'}\}.\]
It turns out that the panted systole length does not depend on the choice of a minimal pants decomposition (see Lemma \ref{psl independent}), and hence is an invariant of the geodesic current $\nu$.

In this setting, our main theorem is the following.

\begin{thm}\label{main theorem}
There is a constant $C\in\Rbbb^+$ depending only on the topology of $S'$, so that for any period minimizing $\nu\in\Cmc(S)$, we have the inequalities
\[\frac{1}{4}\log(2)\leq h_\nu(S')K_\nu(S')\leq C\cdot\left(\log(4)+1+\log\left(1+\frac{1}{x_0}\right)\right),\]
where $x_0$ is the unique positive solution to the equation $(1+x)^{\left\lceil\frac{K_\nu(S')}{L_\nu(S')}-1\right\rceil}x=1$.
\end{thm}

Together, Theorem \ref{positively ratioed to geodesic currents} and Theorem \ref{main theorem} give a systolic inequality that holds for all positively ratioed representations. 

Let $h^\rho_\alpha:=h(\ell^\rho_\alpha,S)$ and $L^\rho_\alpha:=L(\ell^\rho_\alpha,S)$. As a first corollary to Theorem \ref{main theorem}, we have the following.

\begin{cor}[Corollary \ref{non-closed}]\label{easy inequality}
There is a constant $C$ depending only on the topology of $S$, so that for any $[P_\theta]$-positively ratioed representation $\rho:\Gamma\to G$, and any $\alpha\in\theta$, we have the inequality
\[h^\rho_\alpha L^\rho_\alpha\leq C.\]
\end{cor}

We emphasize that $C$ does not depend on $G$. In particular, if $\{\rho_j:\Gamma\to G_j\}_{j\in J}$ is a collection of $[P_j]$-positively ratioed representations on $\Gamma$ so that $\{h_{\alpha_j}^{\rho_j}\}_{j\in J}$ is uniformly bounded below by a positive number, then Corollary \ref{easy inequality} implies that $\{L_{\alpha_j}^{\rho_j}\}_{j\in J}$ is uniformly bounded from above.

Similarly, given a negatively curved Riemannian metric $m$ on $S'$ with totally geodesic (possibly empty) boundary, one can also define a length function $\ell_m:\Cmc\Gmc(S')\to\Rbbb^+$ which assigns to each free homotopy class of closed curves the length of the geodesic representative in that free homotopy class. Theorem \ref{main theorem} also implies the following.

\begin{cor}\label{Riemannian}
There is a constant $C$ depending only on the topology of $S'$, so that for any negatively curved Riemannian metric $m$ on $S'$ with totally geodesic boundary,
\[h(\ell_m,S') L(\ell_m,S')\leq C.\]
\end{cor}
Corollary \ref{Riemannian} is a consequence of the work of Sabourau \cite{Sab1} in the case when $S'$ is a closed surface. The constants in the statements of Theorem \ref{main theorem}, Corollary \ref{easy inequality} and Corollary \ref{Riemannian} are explicit but not sharp, and depend exponentially on the Euler characteristic of the surface.

Another corollary of Theorem \ref{main theorem} is the following criterion for when the entropy along a sequence of ``thick'' positively ratioed representations converges to $0$. 

\begin{cor}[Corollary \ref{entropy degeneration corollary}]\label{mapping class}
Let $\{\rho_j:\Gamma\to G_j\}_{j=1}^\infty$ be a sequence of $[P_j]$-positively ratioed representations, let $\theta_j$ be a subset of the positive roots of $G_j$ determined by $[P_j]$. Also, let $\alpha_j\in\theta_j$ so that $\inf_jL^{\rho_j}_{\alpha_j}>0$. Then $\lim_{j\to\infty}h^{\rho_j}_{\alpha_j}=0$ if and only if for any subsequence of $\{\rho_j\}_{j=1}^\infty$, there is
\begin{itemize}
\item a further subsequence, which we also denote by $\{\rho_j\}_{j=1}^\infty$,
\item a sequence $\{f_j\}_{j=1}^\infty$ of elements in the mapping class group of $S$,
\item a (possibly empty) collection $\Dmc\subset\Cmc\Gmc(S)$ of pairwise non-intersecting simple closed curves,
\end{itemize}
so that 
\[\lim_{j\to\infty}\min\{\ell^{f_j\cdot\rho_j}_{\alpha_j}(c):c\in\Cmc\Gmc(S\setminus\Dmc) \text{ is non-peripheral}\}=\infty\] 
and 
\[\sup_j\max\{\ell^{f_j\cdot\rho_j}_{\alpha_j}(c):c\in\Dmc\}<\infty.\]
Here, $f_j\cdot\rho_j:=\rho_j\circ (f_j)_*$, where $(f_j)_*:\Gamma\to\Gamma$ is the group homomorphism induced by the mapping class $f_j:S\to S$.
\end{cor}

Nie \cite{Nie1}, \cite{Nie2} and the second author \cite{Zha1}, \cite{Zha2} previously studied sequences of Hitchin representations whose entropy goes to zero. Corollary \ref{mapping class} includes all such sequences.

The rest of this article is organized as follows. In Section 2, we define positively ratioed representations and prove Theorem \ref{positively ratioed to geodesic currents}. Then, we show that Hitchin and maximal representations are examples of positively ratioed representations in Section 3. In Section 4, we prove Corollary \ref{Burger-Pozzetti} and some facts regarding geodesic currents and the intersection pairing, and Section 5 and 6 are devoted to the proof of Theorem \ref{main theorem}. Finally, in Section 7, we prove Corollary \ref{easy inequality}, Corollary \ref{Riemannian} and Corollary \ref{mapping class}. 

{\bf Acknowledgements:}
This work has benefitted from conversations with Ursula Hamenst\"adt, Beatrice Pozzetti, Francis Bonahon, Marc Burger, Fanny Kassel and Richard Canary. The authors are grateful for their input. The authors also would especially like to thank Guillaume Dreyer, who inspired us to work on this project. Much of the work in this project was done during the program ``Dynamics on Moduli spaces of Geometric structures'' at MSRI in the Spring of 2015, and the GEAR-funded workshops ``Workshop on $\Sp(4,\Rbbb)$-Anosov representations'' and ``Workshop on surface group representations''. The authors thank the organizers of these programs for their hospitality. The authors also thank the referees of this paper for their valuable input.
\tableofcontents

\section{Positively ratioed representations}\label{PRR}

The goal of this section is to describe a class of surface group representations, which we call positively ratioed representations. The main property these representations have is that certain length functions associated to them ``arise from geodesic currents". This forces the length functions associated to these representations to satisfy some strong properties that are explained in Section \ref{Background on geodesic currents}.

\subsection{Topological geodesics} 
We begin by carefully specifying what we mean by a geodesic and a closed geodesic on a topological surface. The notation developed in this section will be used in the rest of the paper.

First, we will define closed geodesics. Let $[\Gamma]$ denote the set of conjugacy classes in $\Gamma$, and let $\sim$ be an equivalence relation on $[\Gamma]$ given by $[\gamma]\sim[\gamma^{-1}]$. 

\begin{definition}\label{closed geodesic}
A \emph{closed geodesic} in $S$ is a non-identity equivalence class in $[\Gamma]/\sim$. Also, we say that a closed geodesic is \emph{primitive} if it has a primitive representative in $\Gamma$ (equivalently, all of its representatives in $\Gamma$ are primitive). 
\end{definition}

We will denote the set of all closed geodesics in $S$ by $\Cmc\Gmc(S)$, and denote the equivalence class in $\Cmc\Gmc(S)$ containing $\gamma\in\Gamma\setminus\{\id\}$ by $[[\gamma]]$. Observe that $\Cmc\Gmc(S)$ is naturally in bijection with the free homotopy classes of closed curves on $S$. Hence, if we choose a hyperbolic structure $\Sigma$ on $S$, then the closed geodesics in $S$ are identified with the closed hyperbolic geodesics in $\Sigma$ since every free homotopy class of closed curves in $\Sigma$ contains a unique closed hyperbolic geodesic. 

Next, we will define geodesics. It is well-known that $\Gamma$ is a hyperbolic group, so it admits a Gromov boundary $\partial\Gamma$, which is topologically a circle. 

\begin{definition}\label{geodesic}
A (unoriented) \emph{geodesic in $\Std$} is an element of the topological space
\[\Gmc(\Std):=\{(x,y)\in\partial\Gamma\times\partial\Gamma:x\neq y\}/\sim,\]
where $\sim$ is the equivalence relation defined by $(x,y)\sim(y,x)$. Also, a \emph{geodesic in $S$} is an element in $\Gmc(S):=\Gmc(\Std)/\Gamma$. 
\end{definition}

Denote the equivalence classes in $\Gmc(\Std)$ and $\Gmc(S)$ containing $(x,y)$ by $\{x,y\}$ and $[x,y]$ respectively. Observe that if we choose a hyperbolic structure $\Sigma$ on $S$, then this induces a hyperbolic structure $\widetilde{\Sigma}$ on $\Std$. The natural identification of $\partial\Gamma$ with the visual boundary $\partial\widetilde{\Sigma}$ of $\widetilde{\Sigma}$ then realizes geodesics in $\Std$ (or $S$) as hyperbolic geodesics in $\widetilde{\Sigma}$ (or $\Sigma$). 

Of course, closed geodesics in $S$ and geodesics in $S$ can be explicitly related in the following way. Any $\gamma\in\Gamma\setminus\{\id\}$ has an attracting and a repelling fixed point in $\partial\Gamma$, which we denote by $\gamma^+$ and $\gamma^-$ respectively. This allows us to define the map $F:\Cmc\Gmc(S)\to\Gmc(S)$ by $F:[[\gamma]]\mapsto[\gamma^-,\gamma^+]$. More informally, this sends every closed geodesic to the bi-infinite geodesic that ``wraps around" it. Note that the map $F$ is not injective; if $\gamma\in\Gamma$ is primitive, then $F^{-1}(F[[\gamma]])=\{[[\gamma^n]]:n\in\Zbbb\setminus\{0\}\}$. 

Finally, we have a notion for when two geodesics intersect transversely. 

\begin{definition}
We say that $\{a,b\},\{c,d\}\in\Gmc(\Std)$ \emph{intersect transversely} if $a,c,b,d$ lie in $\partial\Gamma$ in that (strict) cyclic order. Similarly, two geodesics in $\Gmc(S)$ \emph{intersect transversely} if they have representatives in $\Gmc(\Std)$ that intersect transversely, and two closed geodesics in $\Cmc\Gmc(S)$ \emph{intersect transversely} if their images under the map $F$ described above intersect transversely in $\Gmc(S)$.
\end{definition}

If we equip $S$ with a hyperbolic structure $\Sigma$, then a pair of geodesics or closed geodesics in $S$ intersect transversely if and only if they intersect transversely as geodesics or closed geodesics for the metric on $\Sigma$.

\subsection{Cross ratios and geodesic currents}\label{Cross ratios and positively ratioed}
The reader should be cautioned that there are many non-equivalent definitions of cross ratios in the literature, even in the restricted setting of Anosov representations. The definition we use here is one given by Ledrappier (Definition 1.f of \cite{Led2}). Consider the set 
\[\partial \Gamma^{[4]}:=\{(a,b,c,d)\in\partial\Gamma^4\colon \{a,b\}\cap\{c,d\}=\emptyset \}.\] 

\begin{definition}\label{def:crossratio}\
\begin{itemize}
\item A \emph{cross ratio} is a continuous function $B:\partial \Gamma^{[4]}\to\Rbbb$ that is invariant under the diagonal action of $\Gamma$ and satisfies the following:
\begin{enumerate}
	\item \emph{(Symmetry)} $B(x,y,z,w)=B(z,w,x,y)$;
	\item \emph{(Additivity)} $B(x,y,z,w)+B(x,y,w,u)=B(x,y,z,u)$.
\end{enumerate}
for all $x,y,z,w,u\in \partial \Gamma$ such that $(x,y,z,w),(x,y,w,u)\in\partial \Gamma^{[4]}$.
\item For any $[[\gamma]]=c\in\Cmc\Gmc(S)$, the \emph{$B$-period} of $c$ is $\ell_B(c):=B(\gamma^-,\gamma^+,\gamma\cdot a,a)$ for some $a\in\partial\Gamma-\{\gamma^-,\gamma^+\}$. 
\end{itemize}
\end{definition}

One easily shows that the $B$-period of $c$ does not depend on the choice of $a$ or $\gamma$. The following theorem of Otal (Theorem 2.2 of \cite{Ota1}, see also Theorem 1.f of Ledrappier \cite{Led2}) states that any cross ratio $B$ is determined by the $B$-periods.

\begin{thm}[Otal]\label{otal}
If $B_1,B_2:\partial\Gamma^{[4]}\to\Rbbb$ are cross ratios so that $\ell_{B_1}(c)=\ell_{B_2}(c)$ for all $c\in\Cmc\Gmc(S)$, then $B_1=B_2$.
\end{thm}

Cross ratios are intimately related with geodesic currents, which we will now define. The notion of a geodesic current was first introduced by Bonahon \cite{Bon1}, who used them to study Teichm\"uller space.

\begin{definition}
A \emph{geodesic current on $S$} is a $\Gamma$-invariant, locally finite (non-signed) Borel measure on $\Gmc(\Std)$. 
\end{definition}

Denote the space of geodesic currents on $S$ by $\Cmc(S)$. It can be naturally realized as an open cone in an infinite dimensional vector space equipped with the weak$^*$ topology (see Section 1 of Bonahon \cite{Bon2}). The $\Gamma$-invariance in the above definition ensures that every geodesic current $\nu\in\Cmc(S)$ descends to a finite measure $\widehat{\nu}$ on the compact space $\Gmc(S)$. However, the $\Gamma$-action on $\Gmc(\Std)$ is not proper, so $\Gmc(S)$ is not Hausdorff. As such, it is often more convenient to work with $\nu$ instead of $\widehat{\nu}$. 

An important example of geodesic currents are the ones associated to closed geodesics. Given any closed geodesic $c=[[\gamma]]\in\Cmc\Gmc(S)$, let $\mu_c\in\Cmc(S)$ be the geodesic current defined by
\[
\mu_c=\sum_{\eta\cdot\langle\gamma\rangle\in\Gamma/\langle\gamma\rangle}\delta_{\{\eta\cdot\gamma^+,\eta\cdot\gamma^-\}},
\]
where $\delta_{\{x,y\}}$ is the Dirac measure supported at the point $\{x,y\}\in\Gmc(\Std)$. When $c$ is primitive, $\hat\mu_c$ is the Dirac measure on $F(c)$. This realizes $\Cmc\Gmc(S)$ as a subset of $\Cmc(S)$. \textbf{Henceforth, we will blur the distinction between $\Cmc\Gmc(S)$ and the subset of $\Cmc(S)$ it is identified with, by using $c$ to denote $\mu_c$.}

On the space of geodesic currents, Bonahon (Section 4.2 of \cite{Bon1}) defined an important function that we will now describe. 

Let $\Dmc\Gmc(\Std)\subset\Gmc(\Std)\times\Gmc(\Std)$ be the open subset defined by 
\[\Dmc\Gmc(\Std):=\{(l_1,l_2)\in\Gmc(\Std)\times\Gmc(\Std):l_1,l_2\text{ intersect transversely}\}.\]
Note that $\Dmc\Gmc(\Std)$ is stabilized by the diagonal $\Gamma$ action on $\Gmc(\Std)\times\Gmc(\Std)$, so we can define $\Dmc\Gmc(S):=\Dmc\Gmc(\Std)/\Gamma$. In this case, the $\Gamma$ action on $\Dmc\Gmc(\Std)$ is proper, so $\Dmc\Gmc(S)$ is a Hausdorff space. For any $\mu,\nu\in\Cmc(S)$, the $\Gamma$-invariant measure $\mu\times\nu$ on $\Dmc\Gmc(\Std)$ descends to a measure $\widehat{\mu\times\nu}$ on $\Dmc\Gmc(S)$. 

\begin{definition}\label{def:intersection form}
The \emph{intersection form} on $\Cmc(S)$ is the map $i:\Cmc(S)\times\Cmc(S)\to\Rbbb$ given by
$i(\mu,\nu)=\widehat{\mu\times\nu}(\Dmc\Gmc(S))$.
\end{definition}

Bonahon proved that the intersection form is well-defined and continuous, and it is easy to verify that it is symmetric and bilinear. He also proves that if $c,c'\in\Cmc\Gmc(S)$ then $i(c,c')$ gives the geometric intersection number between $c$ and $c'$. More properties of the intersection form are later explained in Section \ref{properties of the intersection form}.

There are several ways one can relate geodesic currents to cross ratios. One way to do so is to associate to every H\"older continuous cross ratio a Gibbs current (Hamenst\"adt \cite{Ham1}, also see Ledrappier \cite{Led2}). However, this is not useful for the purposes of this paper because it is not easy to read off the periods of the cross ratio from the Gibbs current. Instead, given a cross ratio $B$, we would like to find an intersection current, which we now define.

\begin{definition}
Let $B$ be a cross ratio. A geodesic current $\mu$ is an \emph{intersection current} for $B$ if $\ell_B(c)=i(\mu,c)$ for all $c\in\Cmc\Gmc(S)$. 
\end{definition}

Unfortunately, it turns out that there are cross ratios for which the intersection current does not exist. On the other hand, Hamenst\"adt observed (page 103 and 104 of \cite{Ham2}) that one can always find intersection currents for cross ratios that satisfy the following positivity condition.

\begin{definition} \label{positive}
A cross ratio $B$ is \emph{positive} if for all $x, y, z, w\in\partial\Gamma$ in this cyclic order, one has $B(x,y,z,w)\geq 0$.
\end{definition}

\begin{thm}[Hamenst\"adt]\label{cross ratio intersection} 
If $B:\partial\Gamma^{[4]}\to\Rbbb$ is a positive cross ratio, then it has a unique intersection current. 
\end{thm}

The proof of Theorem \ref{cross ratio intersection} is a standard argument from analysis. However, for lack of a good reference for the proof, we give the full proof in Appendix \ref{From cross ratios to geodesic currents}. 

\subsection{Background on Semisimple Lie groups} \label{background}

We would like to exploit the existence of intersection currents for positive cross ratios to study the length functions of a certain class of Anosov representations. To define Anosov representations, we need to recall some basic facts regarding non-compact, semisimple, real algebraic groups and their real representations. See Chapter 2 of Eberlein \cite{Ebe1}, Chapter VI.3 of Helgason \cite{Hel1}, Chapter I - III of Humphreys \cite{Hum1}, and Section 4 of Gu\'eritaud-Guichard-Kassel-Wienhard \cite{GGKW} for more details. 

Let $G$ be a non-compact, semisimple, Lie group with Lie algebra $\mathfrak g$. We will also assume that $G$ is a finite union of connected components (for the real topology) of the real points $\mathbf{G}(\mathbb{R})$ of some algebraic group $\mathbf{G}$, and that the adjoint action of $G$ on its Lie algebra is by inner automorphisms, i.e. $\Ad(G)\subset\mathrm{Aut}(\mathfrak{g})_0$. It is well-known that there is a unique non-positively curved Riemannian symmetric space $X$ on which $G$ acts transitively by isometries, so that for any point $p\in X$, the stabilizer in $G$ of $p$ is a maximal compact Lie subgroup $K\subset G$. The transitivity of the $G$-action on $X$ implies that $X\simeq G/K$ as $G$-spaces. 

Let $\kmf\subset\gmf$ be the Lie subalgebra of $K\subset G$. One can prove that $\kmf\subset\gmf$ is a maximal subspace on which the Killing form on $\gmf$ is negative definite. Since $\gmf$ is semisimple, the Killing form on $\gmf$ is non-degenerate. This gives an orthogonal decomposition $\gmf=\kmf+\pmf$. 

\begin{definition}
The \emph{Cartan involution} $\tau_K:\gmf\to\gmf$ is the involution so that $\tau_K|_{\kmf}=\id$ and $\tau_K|_{\pmf}=-\id$. 
\end{definition}

Geometrically, via the canonical identification $T_pX\simeq \gmf/\kmf\simeq\pmf$, the involution $\tau_K|_{\pmf}:T_pX\to T_pX$ is the derivative of the geodesic involution of $X$ at $p$.

A useful way to study $G$ is to consider its linear representations. To do so, we will consider its restricted weights, which we now define. Let $\mathfrak a$ be a maximal abelian subspace in $\pmf$, then $\amf\subset\pmf\simeq T_pX$ is the tangent space to a maximal flat $F$ in $X$ containing $p$, i.e. $\exp_p(\amf)=F$. Given an irreducible, real, finite dimensional, linear representation $r:G\to \GL(V)$ and $\alpha\in\mathfrak a^*$, define
\[V_\alpha:=\{w\in V:r(\exp(v))w=e^{\alpha(v)}w\text{ for all }v\in\mathfrak a\}.\]

\begin{definition}
We call $\alpha$ a \emph{restricted weight} of the representation $(r,V)$ if $\alpha\neq 0$ and $V_\alpha$ is non-empty. If $\alpha$ is a restricted weight, then $V_\alpha$ is a \emph{restricted weight space}.
\end{definition}

Let $\Phi(r,V)$ denote the set of restricted weights of $(r,V)$. Since $r(\exp(\mathfrak a))$ is simultaneously diagonalizable over $\Rbbb$, we can decompose
\[V=V_0+\sum_{\alpha\in\Phi(r,V)}V_\alpha\]
into its restricted weight spaces. If we specialize to the adjoint representation $(r,V)=(\Ad,\mathfrak g)$, then the restricted weights of this representation are called the \emph{restricted roots}, and the restricted weight spaces are called the \emph{restricted root spaces}. In this case, we use the notation $\Sigma:=\Phi(\Ad,\mathfrak g)$ and $\mathfrak g_\alpha:=V_\alpha$. 

Let $\Phi:=\bigcup_{(r,V)}\Phi(r,V)$, where the union is taken over all irreducible, real, finite dimensional, linear representations of $G$. It turns out that there is an easy description of $\Phi$ in terms of $\Sigma$:
\[\{0\}\cup\Phi=\left\{\alpha\in\mathfrak a^*:2\frac{(\alpha,\beta)}{(\beta,\beta)}\in\Zbbb\text{ for all }\beta\in\Sigma\right\}.\]
where $(\cdot,\cdot)$ is the Killing form on $\mathfrak a^*$. In particular, $\{0\}\cup\Phi\subset\mathfrak a^*$ is a lattice. One would then like to find a base for the lattice $\{0\}\cup\Phi$. 

To do so, choose any $v_0\in\mathfrak a$ so that $\alpha(v_0)\neq 0$ for all $\alpha\in\Sigma$, and let 
\[\Sigma^+:=\{\alpha\in\Sigma:\alpha(v_0)>0\}.\] 
It is a standard fact that $\alpha\in\Sigma$ if and only if $-\alpha\in\Sigma$, so $\Sigma=\Sigma^+\cup\{-\alpha:\alpha\in\Sigma^+\}$. 

\begin{definition}
A restricted root in $\Sigma^+$ is \emph{simple} if it cannot be written as a non-trivial linear combination of the roots in $\Sigma^+$ with positive coefficients. 
\end{definition}

It turns out that the set of simple restricted roots, denoted $\Delta$, is a basis for $\mathfrak a^*$. However, $\Delta$ is not a base for the lattice $\{0\}\cup\Phi$. To convert $\Delta$ into a base, we perform an additional ``orthogonalization procedure" to every simple restricted root $\alpha$. This gives the following definition.

\begin{definition}
For any $\alpha\in\Delta$, the \emph{restricted fundamental weight} associated to $\alpha$ is the linear functional $\omega_\alpha\in\mathfrak a^*$ defined by
\[2\frac{(\omega_\alpha,\beta)}{(\beta,\beta)}=\delta_{\alpha,\beta}\text{ for all }\beta\in\Delta,\]
where $\delta_{\cdot,\cdot}$ is the Kronecker symbol. 
\end{definition}

It is well-known that $\{\omega_\alpha:\alpha\in\Delta\}$ is a base for the lattice $\{0\}\cup\Phi$. The choice of $v_0$ induces a natural partial ordering $\leq$ on $\{0\}\cup\Phi$ defined as follows. For any $\omega_1,\omega_2\in\{0\}\cup \Phi$, $\omega_1\leq\omega_2$ if $\omega_2-\omega_1$ is a non-negative linear combination of the simple roots in $\Delta$. For any irreducible representation $(r,V)$ of $G$, the set of weights $\{0\}\cup\Phi(r,V)$ has a unique maximal element in the partial ordering $\leq$. This is called the \emph{highest restricted weight} of $(r,V)$, and is a non-negative linear combination of the restricted fundamental weights.

With this, we can state the following theorem of Tits \cite{Tit1} that will play an important role later. Also, see Proposition 3.2 of Quint \cite{Qui1} or Lemma 4.5 of Gu\'eritaud-Guichard-Kassel-Wienhard \cite{GGKW}.

\begin{thm}[Tits]\label{prop:tits}
For any $\alpha\in\Delta$, there is an irreducible linear representation $r_\alpha:G\to \SL(n,\Rbbb)$ so that highest restricted weight $\chi$ of $(r_\alpha,\Rbbb^n)$ is a positive integer multiple of the restricted fundamental weight $\omega_\alpha$, and the weight space $(\Rbbb^n)_\chi$ is one-dimensional. 
\end{thm}

We will refer to the representation $r_\alpha$ guaranteed by Theorem \ref{prop:tits} as an \emph{$\alpha$-Tits representation}.

The choice of $v_0$ also picks out a (closed) \emph{positive Weyl chamber} 
\[\overline{\mathfrak a^+}:=\{v\in\mathfrak a:\alpha(v)\geq 0\text{ for all }\alpha\in\Sigma^+\}.\]
One can show that $\overline{\mathfrak a^+}\subset T_pX\subset TX$ is a fundamental domain of the $G$-action on $TX$, i.e. for any $(p',v')\in TX$, there is a unique $v\in\overline{\mathfrak a^+}$ so that $g\cdot (p',v')=(p,v)$ for some $g\in G$. Since $X$ is complete, this means that for any pair of points $x,y\in X$, there is a unique vector $v_{x,y}\in\overline{\amf^+}$ so that $g\cdot x=p$ and $g\cdot y=\exp_p(v_{x,y})$ for some $g\in G$. This allows us to define the Weyl chamber valued distance as follows.

\begin{definition}
The \emph{Weyl chamber valued distance} is the function $d_{\overline{\mathfrak a^+}}:X\times X\to\overline{\amf^+}$ given by $d_{\overline{\mathfrak a^+}}(p,q):=v_{x,y}$.
\end{definition}

It is easy to see that the Weyl chamber valued distance descends to an injective map
$G\backslash(X\times X)\to\overline{\amf^+}$, and is thus a complete invariant of the $G$-orbits of pairs of points in $X$. Furthermore, $||d_{\overline{\mathfrak a^+}}(p,q)||=d(p,q)$, where $||\cdot||$ is the norm on $\overline{\amf^+}$ induced by the Riemannian metric on $X$ and $d$ is the distance on $X$. The classification of isometries on $G$ also implies that for any $g\in G$, either $\inf\{d(q,g\cdot q):q\in X\}=0$ or there is some $p\in X$ so that $\inf\{d(q,g\cdot q):q\in X\}=d(p,g\cdot p)$. Furthermore, if $\inf\{d(q,g\cdot q):q\in X\}=d(p,g\cdot p)=d(p',g\cdot p')$, then $d_{\overline{\mathfrak a^+}}(p,g\cdot p)=d_{\overline{\mathfrak a^+}}(p',g\cdot p')$. Using this, one can define the Jordan projection geometrically.

\begin{definition}
The \emph{Jordan projection} (sometimes also known as the \emph{Lyapunov projection}) is the map $\lambda_G=\lambda:G\to\overline{\mathfrak a^+}$ defined by 
\begin{itemize}
\item $\lambda:g\mapsto 0$ if $\inf\{d(q,g\cdot q):q\in X\}=0$, 
\item $\lambda:g\mapsto d_{\overline{\mathfrak a^+}}(p,g\cdot p)$ if there is some $p\in X$ so that $\inf\{d(q,g\cdot q):q\in X\}=d(p,g\cdot p)$.
\end{itemize}
\end{definition}

More algebraically, the Jordan projection can also be described as follows. The Jordan decomposition theorem (See Theorem 2.19.24 of Eberlein \cite{Ebe1}) ensures that any $g\in G$ can be written uniquely as a commuting product $g=g_hg_eg_u$, with $g_h$ hyperbolic, $g_e$ elliptic, and $g_u$ unipotent. Furthermore, the fact that $\overline{\amf^+}$ is a fundamental domain of the $G$ action on $TX$ implies that there is a unique vector $v_g\in\overline{\mathfrak a^+}$ so that $\exp(v_g)$ is conjugate to $g_h$. Then the Jordan projection is the map $\lambda:G\to\overline{\mathfrak a^+}$ that sends $g\in G$ to $v_g$. 

Since $X$ is non-positively curved, it has a visual boundary $\partial X$ that is topologically a sphere, and the $G$-action on $X$ extends to a $G$-action on $\partial X$. One can then consider the stabilizers in $G$ of points in $\partial X$.

\begin{definition}
A \emph{parabolic subgroup} of $G$ is the stabilizer of a point in $\partial X$. We say that two parabolic subgroups $P_1,P_2\subset G$ are \emph{opposite} if there is a geodesic in $X$ with endpoints $x_1,x_2\in\partial X$ so that $P_i=\Stab_G(x_i)$ for $i=1,2$.
\end{definition}

If $P_1$ and $P_1'$ are both parabolic subgroups that are opposite to $P_2$, then $P_1$ and $P_1'$ must be conjugate in $G$. As such, we can say that two conjugacy classes $[P_1]$ and $[P_2]$ are \emph{opposite} if for any representative $P_1$ of $[P_1]$ every opposite of $P_1$ lies in the conjugacy class $[P_2]$.

Using parabolic subgroups, we can define flag spaces.

\begin{definition}
Let $x\in\partial X$ and let $P:=\Stab_G(x)$. The $[P]$-flag space, denoted $\mathscr F_{[P]}$, is the $G$-orbit of $x$. 
\end{definition}

Observe that the $[P]$-flag space is a $G$-homogeneous space, so $\mathscr F_{[P]}\simeq G/P$ as $G$-spaces. In particular, as an abstract space, $\mathscr F_{[P]}$ does not depend on $x$. Furthermore, if $P'=gPg^{-1}$ for some $g\in G$, then there is canonical isomorphism between the $G$-spaces $G/P$ and $G/P'$. As such, $\mathscr F_{[P]}$ is well-defined, and depends only on the conjugacy class of $[P]$. 

The decomposition of $\gmf$ into its restricted root spaces can also be used to understand the parabolic subgroups of $G$. For any non-empty subset $\theta\subset\Delta$, the \emph{standard $\theta$-parabolic subgroup} is the parabolic subgroup $P_\theta\subset G$ with Lie algebra
\[
\mathfrak p_\theta:=\mathfrak g_0\oplus\bigoplus_{\alpha\in\Sigma^+}\mathfrak g_\alpha\oplus\bigoplus_{\alpha\in\Sigma^+\cap\Span_\Rbbb(\Delta-\theta)} \mathfrak g_{-\alpha}.
\]
Using this, we can define a map from non-empty subsets of $\Delta$ to conjugacy classes of parabolic subgroups in $G$ by $\theta\mapsto[P_\theta]$. It turns out that this map is in fact a bijection.

It is also well-known that there is some $g\in G$ so that $g\cdot\overline{\amf^+}=-\overline{\amf^+}$. In fact, for any $g'\in G$ so that $g'\cdot\overline{\amf^+}=-\overline{\amf^+}$, we know that $g'\cdot\amf=\amf$ and $g'|_{\amf}=g|_{\amf}$. Using this, we can define the \emph{opposition involution} $i:=-g\colon \overline{\mathfrak a^+}\to\overline{\mathfrak a^+}$. This gives an involution $\iota:\Delta\to\Delta$ defined by $\iota(\alpha)=\alpha\circ i$, which in turn induces an involution, also denoted by $\iota$, on conjugacy classes of parabolic subgroups defined by $\iota[P_\theta]\mapsto [P_{\iota(\theta)}]$. Geometrically, this involution sends the conjugacy class of any parabolic subgroup to the conjugacy class of its opposite. 

From the algebraic description of the Jordan projection $\lambda$, one can verify that for all $g\in G$ and for all $\alpha\in\Delta$, we have $\alpha\circ\lambda(g)=\iota(\alpha)\circ\lambda(g^{-1})$, which in turn implies that $\omega_{\alpha}\circ\lambda(g)=\omega_{\iota(\alpha)}\circ\lambda(g^{-1})$.

We finish this section by describing some of the objects defined above explicitly in the case when $G=\SL(n,\Rbbb)$, as this will be of particular importance to us. We can choose $K$ to be the maximal compact subgroup 
\[\SO(n):=\{k\in\SL(n,\Rbbb):k^Tk=\id\}\subset\SL(n,\Rbbb).\] 
In that case, $\kmf=\smf\omf(n)=\{A\in\smf\lmf(n,\Rbbb):A=-A^T\}$ and $\pmf=\{A\in\smf\lmf(n,\Rbbb):A=A^T\}$, so the Cartan involution $\tau_K:\smf\lmf(n,\Rbbb)\to\smf\lmf(n,\Rbbb)$ is given by $\tau:A\mapsto-A^T$. We can choose the maximal abelian subspace $\amf\subset\pmf$ to be the vector space of traceless diagonal matrices in $\smf\lmf(n,\Rbbb)$. This allows us to naturally identify 
\[\amf=\left\{(x_1,x_2,\dots,x_n)\in\Rbbb^n:\sum_{i=1}^n x_i=0\right\}.\]

With this identification, $\Sigma=\{\alpha_{i,j}\in\amf^*:i\neq j\}$, where $\alpha_{i,j}:\amf\to\Rbbb$ is given by $\alpha_{i,j}:(x_1,\dots,x_n)\mapsto x_i-x_j$. By making an appropriate choice of $v_0\in\amf$, we can also ensure that 
\begin{eqnarray*}
\Sigma^+&=&\{\alpha_{i,j}\in\amf^*:i<j\},\\
\Delta&=&\{\alpha_{i,i+1}\in\amf^*:i=1\dots,n-1\},\\
\overline{\amf^+}&=&\{(x_1,\dots,x_n)\in\amf:x_1\geq\dots\geq x_n\}.
\end{eqnarray*}
With these choices, the restricted fundamental weight $\omega_{\alpha_{i,i+1}}$ corresponding to $\alpha_{i,i+1}$ is given by the formula $\omega_{\alpha_{i,i+1}}(x_1,\dots,x_n)=x_1+\dots+x_i$, and the involution $\iota$ can also explicitly be given by $\iota(\alpha_{i,i+1})=\alpha_{n-i,n-i+1}$. Also, the Jordan projection evaluated on $g\in\SL(n,\Rbbb)$ is  
\[\lambda:g\mapsto \log \bar{g}\]
where $\bar{g}\in\SL(n,\Rbbb)$ is the diagonal matrix whose diagonal entries are the absolute values of the generalized eigenvalues of $g$, listed in descreasing order down the diagonal. The group element $\bar{g}$ is a conjugate of $g_h$.

Finally, if $\theta=\{\alpha_{1,2},\alpha_{n-1,n}\}$, then 
\[\mathfrak p_\theta=\left\{
\left(\begin{array}{ccccc}
*&*&\dots&*&*\\
0&*&\dots&*&*\\
\vdots&\vdots&\ddots&\vdots&\vdots\\
0&*&\dots&*&*\\
0&0&\dots&0&*
\end{array}\right)\in\smf\lmf(n,\Rbbb)\right\}
\]
and $\mathscr F_{[P_\theta]}=\{(L,H)\in\Rbbb\Pbbb^{n-1}\times(\Rbbb\Pbbb^{n-1})^*:L\in H\}$. As such, we refer to the conjugacy class $[P_\theta]$ as the \emph{line-hyperplane stabilizer} of $\SL(n,\Rbbb)$.

\subsection{Anosov representations and positively ratioed representations} \label{Anosov representations}
The notion of Anosov representations was first introduced by Labourie \cite{Lab1}, and later refined by Guichard-Wienhard \cite{GW}. Several other characterizations have been provided by Kapovich-Leeb-Porti \cite{KLP1} \cite{KLP2} and Gu\'eritaud-Guichard-Kassel-Wienhard \cite{GGKW}. In this article, we will only consider Anosov representations from the surface group $\Gamma$ to a non-compact, semisimple, real algebraic group, $G$.

Given a representation $\rho:\Gamma\to G$, a $\rho$-equivariant map $\xi\colon\partial\Gamma\to\mathscr F_{[P]}$ is \emph{dynamics-preserving} if for any $\gamma\in\Gamma\setminus\{\id\}$, $\xi(\gamma^+)$ is the attracting fixed point for the action of $\rho(\gamma)$ on $\mathscr F_{[P]}$. (In particular, $\rho(\gamma)$ has to have an attracting fixed point in $\mathscr F_{[P]}$.) A pair of maps $\xi:\partial\Gamma\to\mathscr F_{[P]}$ and $\eta:\partial\Gamma\to\mathscr F_{\iota[P]}$ is \emph{transverse} if for all $x\neq y$, $(\xi(x),\eta(y))$ lies in the unique open $G$-orbit of $\mathscr F_{[P]}\times\mathscr F_{\iota[P]}$. With this, we can define Anosov representations using a characterization by Gu\'eritaud-Guichard-Kassel-Wienhard (see Theorem 1.7 and Proposition 2.2 of \cite{GGKW}).

\begin{definition}\label{def:Anosov} A representation $\rho\colon\Gamma\to G$ is \emph{$[P]$-Anosov} if 
\begin{itemize}
\item there exist continuous, $\rho$-equivariant, dynamics-preserving and transverse maps $\xi\colon\partial\Gamma\to \mathscr F_{[P]}$ and $\eta\colon\partial\Gamma\to \mathscr F_{\iota[P]}$,
\item there exist $C,c>0$ such that $\alpha\circ\lambda\circ\rho(\gamma)>C\ell_\Gamma(\gamma)-c$ for all $\alpha\in\theta$ and $\gamma\in\Gamma$,
\end{itemize}
where $\ell_\Gamma(\gamma)$ is the translation distance of $\gamma\in\Gamma$ in the Cayley graph of $\Gamma$ with respect to some finite generating set. (Recall that $\lambda$ denotes the Jordan projection.) The maps $\xi$ and $\eta$ are called the \emph{limit curves} of $\rho$.
\end{definition}

Since the set of fixed points of group elements in $\Gamma$ is dense in $\partial\Gamma$, the maps $\xi$ and $\eta$ are unique. In particular, $\xi=\eta$ necessarily when $\theta=\iota(\theta)$. Also, it is a result of Guichard-Wienhard (Lemma 3.18 of \cite{GW}) that for any non-empty $\theta\subset\Delta$, $\rho:\Gamma\to G$ is $[P_\theta]$-Anosov if and only if it is $[P_{\theta\cap \iota(\theta)}]$-Anosov. Hence, we do not lose any generality by only considering parabolic subgroups of $P\subset G$ so that $[P]=\iota[P]$, i.e. non-empty subsets $\theta\subset\Delta$ so that $\theta=\iota(\theta)$. \begin{bf}We will do so for the rest of this article.\end{bf} Under this assumption, we can associate to any $[P]$-Anosov representation some natural length functions. 

\begin{definition}
Let $\rho:\Gamma\to G$ be a $[P_\theta]$-Anosov representation. 
\begin{itemize}
\item For any $\alpha\in\theta$, the \emph{$\alpha$-length function of $\rho$} is the function
\[\ell_\alpha^\rho:\Cmc\Gmc(S)\to\Rbbb\,\,\,\text{ given by }\,\,\,\ell^\rho_\alpha(c):=\big(\omega_\alpha+\omega_{\iota(\alpha)}\big)\circ\lambda\circ\rho(\gamma),\]
where $[[\gamma]]=c\in\Cmc\Gmc(S)$.
\item The \emph{entropy of $\ell_\alpha^\rho$} is the quantity
\[h(\ell_\alpha^\rho):=\limsup_{T\to\infty}\frac{1}{T}\log\#\{c\in\Cmc\Gmc(S)\colon \ell_\alpha^\rho(c)\leq T\}.\]
\end{itemize}
\end{definition}

One can verify that $\ell_\alpha^\rho$ is well-defined and $\ell^\rho_\alpha=\ell^\rho_{\iota(\alpha)}$. It is also a well-known consequence of the Anosovness of $\rho$ that $h(\ell^\rho_\alpha)\in\Rbbb^+$ (for example, see Theorem B of Sambarino \cite{Sam1}). When $G=\PSL(2,\Rbbb)$, one can choose $\amf$ to be the diagonal matrices in $\mathfrak s\mathfrak l(2,\Rbbb)$. If $\rho:\Gamma\to\PSL(2,\Rbbb)$ is a Fuchsian representation, then it is an easy exercise to verify that $\Delta=\{\alpha\}$, where $\alpha:\mathfrak a\to\Rbbb$ is defined by
\[\alpha:\left[\begin{array}{cc}
t&0\\
0&-t
\end{array}\right]\mapsto 2t,\]
and $\rho$ is $[P_\Delta]$-Anosov. In this case, for any $c\in\Cmc\Gmc(S)$, $\ell_\alpha^\rho(c)$ is the hyperbolic length of the geodesic $c$ measured in the hyperbolic structure $\Sigma$ on $S$ corresponding to $\rho$, and it is well-known that $h(\ell_\alpha^\rho)=1$.

For a general Anosov representation however, the length functions are so named purely by analogy as there is no natural metric on the surface that gives rise to these length functions.

As another example, we will consider projective Anosov representations. 

\begin{definition}
Let $[P]$ be the line-hyperplane stabilizer of $\SL(n,\Rbbb)$. A $[P]$-Anosov representation $\rho:\Gamma\to\SL(n,\Rbbb)$ is a \emph{projective Anosov representation}.
\end{definition} 

For any $g\in\SL(n,\Rbbb)$, let $\lambda(g)=(\lambda_1(g),\dots,\lambda_n(g))\in\overline{\amf^+}$ denote the Jordan projection of $g$. Recall that if $[P_\theta]$ is the line-hyperplane stabilizer of $\SL(n,\Rbbb)$, then $\theta=\{\alpha_{1,2},\alpha_{n-1,n}\}$. Hence, there is only one length function of $\rho$, which we will abbreviate by $\ell^\rho:\Cmc\Gmc(S)\to\Rbbb$. By the definition of $\ell^\rho$, we see that
\begin{eqnarray}\label{eqn:half}
\ell^\rho(c)&=&(\omega_{\alpha_{1,2}}+\omega_{\alpha_{n-1,n}})\circ\lambda\circ\rho(\gamma)\nonumber\\
&=&\lambda_1(\rho(\gamma))+\sum_{i=1}^{n-1}\lambda_i(\rho(\gamma))\\
&=&\lambda_1(\rho(\gamma))-\lambda_n(\rho(\gamma)).\nonumber
\end{eqnarray}

If $\rho:\Gamma\to\SL(n,\Rbbb)$ is projective Anosov, then the limit curve  
\[\xi:\partial\Gamma\to\mathscr F_{[P_\theta]}=\mathscr F_{\iota[P_\theta]}\]
corresponding to $\rho$ can be thought of as a pair of continuous, $\rho$-equivariant, maps $\xi^{(1)}:\partial\Gamma\to\Rbbb\Pbbb^{n-1}$ and $\xi^{(n-1)}:\partial\Gamma\to(\Rbbb\Pbbb^{n-1})^*\simeq\Gr(n-1,n)$ so that $\xi^{(1)}(x)\subset\xi^{(n-1)}(y)$ if and only if $x=y$. Since $\xi$ is dynamics preserving, we see that for any $\gamma\in\Gamma\setminus\{\id\}$, if $\gamma^+,\gamma^-\in\partial\Gamma$ denote the attractor and repeller of $\gamma$, then $\xi^{(1)}(\gamma^+)$ and $\xi^{(n-1)}(\gamma^-)$ are the attracting line and repelling hyperplane of $\rho(\gamma)$ respectively. 

In this case, Labourie \cite{Lab2} defined a function $L^\rho:\partial\Gamma^{[4]}\to\Rbbb$ given by 
\[L^\rho(x,y,z,w)=\frac{\langle\xi^{(n-1)}(x),\xi^{(1)}(y)\rangle\langle\xi^{(n-1)}(z),\xi^{(1)}(w)\rangle}{\langle\xi^{(n-1)}(x),\xi^{(1)}(w)\rangle\langle\xi^{(n-1)}(z),\xi^{(1)}(y)\rangle}.\]
Here, for any $p,q=x,y,z,w$, we choose a covector representative in $(\Rbbb^n)^*$ for $\xi^{(n-1)}(p)$ and a vector representative in $\Rbbb^n$ for $\xi^{(1)}(q)$ to evaluate $\langle\xi^{(n-1)}(p),\xi^{(1)}(q)\rangle$. One can verify that $L^\rho(x,y,z,w)$ does not depend on any of the choices made. 

We will refer to the function $L^\rho$ as the \emph{Labourie cross ratio}, even though it is not a cross ratio in the sense of Definition \ref{def:crossratio}. It is easy to see that 
\begin{equation}\label{Labourie additive}
L^\rho(x,y,z,w)\cdot L^\rho(x,w,z,t)=L^\rho(x,y,z,t).
\end{equation}
Furthermore, an easy computation proves that for all $\gamma\in\Gamma\setminus\{\id\}$ and for all $a\in\partial\Gamma\setminus\{\gamma^+,\gamma^-\}$, we have
\begin{equation}\label{Labourie period}
L^\rho(\gamma^-,\gamma\cdot a,\gamma^+,a)=\frac{e^{\lambda_1(\rho(\gamma))}}{e^{\lambda_n(\rho(\gamma))}}=L^\rho(\gamma\cdot a,\gamma^-,a,\gamma^+).
\end{equation}

Using (\ref{Labourie additive}), one can verify that the function
\[B^\rho(x,y,z,w):=\frac{1}{2}\log(L^\rho(x,z,y,w)\cdot L^\rho(z,x,w,y))\]
is indeed a cross ratio. Furthermore, for any $c=[[\gamma]]\in\Cmc\Gmc(S)$, (\ref{eqn:half}) and (\ref{Labourie period}) imply that 
\begin{eqnarray*}
\ell^\rho(c)&=&\frac{1}{2}\log(L^\rho(\gamma^-,\gamma\cdot a,\gamma^+,a)\cdot L^\rho(\gamma\cdot a, \gamma^-,a,\gamma^+))\\
&=&\ell_{B^\rho}(c).
\end{eqnarray*}
In particular, when $[P]$ is the line-hyperplane stabilizer of $\SL(n,\Rbbb)$, the length function of $[P]$-Anosov representations are the periods of a unique cross ratio (the uniqueness is a consequence of Theorem \ref{otal}). This is in fact true for any restricted simple root $\alpha\in\theta$ for any $[P_\theta]$-Anosov representation $\rho:\Gamma\to G$. To prove this, we need the following proposition, which is a special case of Proposition 4.3 of Guichard-Wienhard \cite{GW} (also see Proposition 4.6 of Gu\'eritaud-Guichard-Kassel-Wienhard \cite{GGKW}).

\begin{prop}[Guichard-Wienhard]\label{composition with Tits representation}
Let $\theta\subset\Delta$, let $\rho:\Gamma\to G$ be a $[P_\theta]$-Anosov representation, and let $\alpha\in\theta$. Also, let $r_\alpha:G\to\SL(n,\Rbbb)$ be an $\alpha$-Tits representation. Then $r_\alpha\circ\rho:\Gamma\to\SL(n,\Rbbb)$ is $[P]$-Anosov, where $[P]$ is the line-hyperplane stabilizer of $\SL(n,\Rbbb)$. Furthermore, the limit curve corresponding to $\rho$ is $f_\alpha\circ\xi:\partial\Gamma\to\mathscr F_{[P]}$, where $f_\alpha:\mathscr F_{[P_\theta]}\to\mathscr F_{[P]}$ is the unique $r_\alpha$-equivariant map.
\end{prop}

More explicitly, if $P\subset \SL(n,\Rbbb)$ is the representative in the conjugacy class $[P]$ so that $r_\alpha^{-1}(P)=P_\theta$, then $f_\alpha:\mathscr F_{[P_\theta]}=G/P_\theta\to \SL(n,\Rbbb)/P=\mathscr F_{[P]}$ is given by $f_\alpha:g\cdot P_\theta\mapsto r_\alpha(g)\cdot P$. Proposition \ref{composition with Tits representation} allows us to reduce the study of length functions of a general Anosov representation to the length function of a projective Anosov representations. In particular, we can prove the following.

\begin{prop}\label{cross ratio lengths}
Let $\theta\subset\Delta$ and let $\rho:\Gamma\to G$ be $[P_\theta]$-Anosov. For all $\alpha\in\theta$, there is a unique cross ratio $B^\rho_\alpha$ so that for all $c\in\Cmc\Gmc(S)$, 
\[\ell^\rho_\alpha(c)=\ell_{B^\rho_\alpha}(c).\]
\end{prop}

\begin{proof}
Let $r_\alpha:G\to\SL(n,\Rbbb)$ be an $\alpha$-Tits representation of $G$. Since $r_\alpha\circ\rho$ is Anosov with respect to the line-hyperplane stabilizer in $\SL(n,\Rbbb)$, $r_\alpha\circ\rho(\gamma)$ has a largest eigenvalue of multiplicity $1$ for all $\gamma\in\Gamma\setminus\{\id\}$. By Theorem \ref{prop:tits}, we have
\[\omega_{\alpha_{1,2}}\circ\lambda_{\SL(n,\Rbbb)}\circ r_\alpha\circ\rho(\gamma)=k\cdot\omega_\alpha\circ\lambda_G\circ\rho(\gamma)\]
for some $k\in\Zbbb^+$. Similarly, we have that
\begin{eqnarray*}
\omega_{\alpha_{n-1,n}}\circ\lambda_{\SL(n,\Rbbb)}\circ r_\alpha\circ\rho(\gamma)&=&\omega_{\alpha_{1,2}}\circ\lambda_{\SL(n,\Rbbb)}\circ r_\alpha\circ\rho(\gamma^{-1})\\
&=&k\cdot\omega_{\alpha}\circ\lambda_G\circ\rho(\gamma^{-1})\\
&=&k\cdot\omega_{\iota(\alpha)}\circ\lambda_G\circ\rho(\gamma).
\end{eqnarray*}
Together, these imply that for any $c=[[\gamma]]\in\Cmc\Gmc(S)$, 
\begin{eqnarray*}
\ell^{r_\alpha\circ\rho}(c)&=&(\omega_{\alpha_{1,2}}+\omega_{\alpha_{n-1,n}})\circ\lambda_{\SL(n,\Rbbb)}\circ r_\alpha\circ\rho(\gamma)\\
&=&(k\cdot\omega_\alpha+k\cdot\omega_{\iota(\alpha)})\circ\lambda_G\circ\rho(\gamma)\\
&=&k\cdot\ell^\rho_\alpha(c).
\end{eqnarray*}
Define $B^\rho_\alpha:=\frac{1}{k} B^{r_\alpha\circ\rho}$. Since $\ell^{r_\alpha\circ\rho}=\ell_{B^{r_\alpha\circ\rho}}$, it immediately follows that $\ell^\rho_\alpha=\ell_{B^\rho_\alpha}$. The uniqueness of $B^\rho_\alpha$ is a consequence of Theorem \ref{otal}.
\end{proof}

Using Proposition \ref{cross ratio lengths}, we can define positively ratioed representations.

\begin{definition}\label{positively ratioed}
A $[P_\theta]$-Anosov representation $\rho:\Gamma\to G$ is $[P_\theta]$-\emph{positively ratioed} if the cross ratio $B^\rho_\alpha$ is positive for all $\alpha\in\theta$ (see Proposition \ref{cross ratio lengths} and Definition \ref{positive}) .
\end{definition}

As a consequence of Theorem \ref{cross ratio intersection} and Proposition \ref{cross ratio lengths}, we see that for any $\theta\subset\Delta$, any $[P_\theta]$-positively ratioed representation $\rho:\Gamma\to G$, and any $\alpha\in\theta$, there is a unique geodesic current $\mu^\rho_\alpha\in\Cmc(S)$ so that $i(\mu^\rho_\alpha,c)=\ell^\rho_\alpha(c)$ for every $c\in\Cmc\Gmc(S)$. This is stated as Theorem \ref{positively ratioed to geodesic currents} in the introduction.

Let $\theta'\subset\theta$ be subsets of $\Delta$. Guichard-Wienhard (Lemma 3.18 of \cite{GW}) proved that if $\rho:\Gamma\to G$ is $[P_\theta]$-Anosov, then it is also $[P_{\theta'}]$-Anosov. It then follows from this definition that if $\rho$ is $[P_\theta]$-positively ratioed, then it is also $[P_{\theta'}]$-positively ratioed.

The intersection currents arising from positively ratioed representations satisfy some basic properties that we will now explain. 

Note that in the definition of a positive cross ratio, we used the weak inequality instead of the strict inequality. However, in the definition of positively ratioed representations, we can replace the weak inequality with a strict inequality without changing the definition. This is the content of the next proposition.

\begin{notation}\label{interval notation 2}
For any $x,y,z\in\partial\Gamma$, let $(x,y]_z$ denote the half-open subsegment of $\partial\Gamma$ that does not contain $z$ and has $x$ and $y$ as its open and closed endpoints respectively. We will also use $(x,y)_z$, $[x,y)_z$ and $[x,y]_z$ to denote the interval $(x,y]_z$, but with the appropriate closed and open endpoints.
\end{notation}

\begin{prop}\label{strictly positively ratioed}
Let $\theta\subset\Delta$, let $\rho:\Gamma\to G$ be a $[P_\theta]$-positively ratioed representation, and let $\alpha\in\theta$. Then $B^\rho_\alpha(x,y,z,w)>0$ for all $x,y,z,w\in\partial\Gamma$ in this cyclic order.
\end{prop}

\begin{proof}
Recall that $r_\alpha\circ\rho:\Gamma\to\SL(n,\Rbbb)$ is projective Anosov, and that $B^\rho_\alpha=\frac{1}{k}B^{r_\alpha\circ\rho}$ for some $k\in\Zbbb^+$. Therefore, we can assume that $\rho$ is projective Anosov. Let
\[\xi=\left(\xi^{(1)},\xi^{(n-1)}\right):\partial\Gamma\to\Rbbb\Pbbb^{n-1}\times(\Rbbb\Pbbb^{n-1})^*\] 
denote the limit curve of $\rho$. By assumption, $B^\rho(x,y,z,w)\geq 0$ for all $x,y,z,w\in\partial\Gamma$ in this cyclic order. Suppose for contradiction that there is some $x,y,z,w\in\partial\Gamma$ in this cyclic order so that $B^\rho(x,y,z,w)=0$. This means that for all $t\in[z,w]_x$, $B^\rho(x,y,t,w)=0$. It then follows from the definition of $B^\rho$ that $\xi^{(1)}\big([z,w]_x\big)$ lies in the proper subspace $\left(\xi^{(n-1)}(x)\cap\xi^{(n-1)}(y)\right)+\xi^{(1)}(w)\subset\Rbbb^n$. 

Let $V\subset\Rbbb^n$ be the minimal subspace containing $\xi^{(1)}\big([z,w]_x\big)$, and let $\gamma\in\Gamma\setminus\{\id\}$ so that its repellor $\gamma^-$ lies in $(z,w)_x$. Since $\xi^{(1)}\big([z,w]_x\big)\subset \rho(\gamma)\cdot \xi^{(1)}\big([z,w]_x\big)$, we see that $V\subset\rho(\gamma)\cdot V$, so $V$ is $\rho(\gamma)$-invariant. At the same time, observe that 
\[\bigcup_{n=0}^\infty\gamma^n\cdot[z,w]_x=\partial\Gamma\setminus\{\gamma^+\},\] 
so the continuity and $\rho$-equivariance of $\xi^{(1)}$ implies that $\xi^{(1)}(\partial\Gamma)\subset V$. However, $V\subset\left(\xi^{(n-1)}(x)\cap\xi^{(n-1)}(y)\right)+\xi^{(1)}(w)$, which means that $\xi^{(1)}(\partial\Gamma)\cap\xi^{(n-1)}(x)\subset\xi^{(n-1)}(x)\cap\xi^{(n-1)}(y)$. In particular, $\xi^{(1)}(x)\in\xi^{(n-1)}(y)$, but this violates the transversality of $\xi$.
\end{proof}

\begin{definition}\label{period minimizing}
Let $\nu\in\Cmc(S)$ be a geodesic current. We say that $\nu$ is \emph{period minimizing} if 
\[\left|\{c\in\Cmc\Gmc(S):i(\nu,c)<T\}\right|<\infty\]
for all $T\in\Rbbb^+$. Also, $\nu$ has \emph{full support} if $\nu(U)>0$ for any open set $U\subset\Cmc(S)$.
\end{definition}

It is well known that $|\{[\gamma]\in[\Gamma]:\ell_\Gamma(\gamma)<T\}|<\infty$ for all $T\in\Rbbb^+$. As such, an immediate consequence of the Anosovness of $\rho$ that $\mu^\rho_\alpha$ is period minimizing. In particular, measured laminations are not intersection currents coming from Anosov representations, because they are not period minimizing.

Let $x,y,z,w\in\partial\Gamma$ in this cyclic order, and let $\Gmc_{[x,y]_z,[w,z]_y}\subset\Gmc(\Std)$ be the set of geodesics in $\Std$ with one endpoint in $[x,y]_z$ and one endpoint in $[w,z]_y$. By the construction of $\mu^\rho_\alpha$ from $B^\rho_\alpha$ (see Appendix \ref{From cross ratios to geodesic currents}), we see that 
\begin{equation}\label{current ratio equation}\mu^\rho_\alpha\big(\Gmc_{[x,y]_z,[w,z]_y}\big)=B^\rho_\alpha(x,y,z,w).\end{equation}
In the degenerate case when $x=y\neq z=w$, this implies that $\mu^\rho_\alpha(\{x,z\})=0$, so the $\mu^\rho_\alpha$-measure of every point in $\Gmc(\Std)$ is zero. As a consequence, the intersection current arising from an Anosov representation has no atoms.

Finally, the intersection currents arising from Anosov representations have full support. To see this, observe that
\[\{\Gmc_{(x,y)_z,(w,z)_y}\subset\Gmc(\Std):x,y,z,w\in\partial\Gamma\text{ in this cyclic order}\}\] generates the topology on $\Gmc(\Std)$. Since $\mu^\rho_\alpha\big(\Gmc_{[x,y]_z,[w,z]_y}\big)=\mu^\rho_\alpha\big(\Gmc_{(x,y)_z,(w,z)_y}\big)$, (\ref{current ratio equation}) and Proposition \ref{strictly positively ratioed} immediately imply that $\mu^\rho_\alpha$ has full support.

\section{Examples of positively ratioed representations}\label{sec:examples}
In this section, we provide several important examples of positively ratioed representations to motivate the definition.

\subsection{Hitchin representations}
The \emph{Teichm\"uller space of $S$} can be defined to be
\[\Tmc(S):=\{\text{discrete, faithful }\rho:\Gamma\to\PSL(2,\Rbbb)\}/\PGL(2,\Rbbb).\]
This is the space of holonomy representations of hyperbolic structures on $S$. If we equip $\Tmc(S)$ with the compact-open topology, it is well-known that $\Tmc(S)$ is topologically a cell of dimension $6g-6$. Let 
\[
\iota_n:\PSL(2,\Rbbb)\to \PSL(n,\Rbbb)
\] 
be the projectivization of the unique (up to post-composing by an automorphism of $\SL(n,\Rbbb)$) $n$-dimensional irreducible representation of $\SL(2,\Rbbb)$ into $\SL(n,\Rbbb)$. If we equip 
\[\Xmc(S,\PSL(n,\Rbbb)):=\mathrm{Hom}(\Gamma,\PSL(n,\Rbbb))/\PGL(n,\Rbbb)\] 
with the compact-open topology, this gives us an embedding
\[i_n:\Tmc(S)\to\Xmc(S,\PSL(n,\Rbbb))\]
defined by $i_n[\rho]=[\iota_n\circ\rho]$. In particular, $i_n(\Tmc(S))\subset\Xmc(S,\PSL(n,\Rbbb))$ is connected.

\begin{definition}
The \emph{$\PSL(n,\Rbbb)$-Hitchin component} $\Hit_n(S)$ is the connected component of $\Xmc(S,\PSL(n,\Rbbb))$ that contains $i_n(\Tmc(S))$. The representations in $\Hit_n(S)$ are known as \emph{$\PSL(n,\Rbbb)$-Hitchin representations}.
\end{definition}

Often, we will simply use a representative $\rho$ in the conjugacy class $[\rho]$ to denote an element in $\Hit_n(S)$. It is classically known that $\Tmc(S)$ is a connected component of $\Xmc(S,\PSL(2,\Rbbb))$, so $\Hit_2(S)=\Tmc(S)$. As such, one can think of $\Hit_n(S)$ as a generalization of $\Tmc(S)$.

The Hitchin component was first studied by Hitchin \cite{Hit1}, who used Higgs bundle techniques to parameterize $\Hit_n(S)$ using certain holomorphic differentials on a Riemann surface homeomorphic to $S$. In particular, he showed that $\Hit_n(S)$ is topologically a cell of dimension $(n^2-1)(2g-2)$, where $g$ is the genus of $S$. With this, the global topology of $\Hit_n(S)$ is completely understood. However, the geometric properties of the representations in $\Hit_n(S)$ remained a mystery until a seminal theorem of Labourie. 

To explain this theorem, we first need the notion of a Frenet curve. Let $\mathscr F(\Rbbb^n)$ denote the space of complete flags in $\Rbbb^n$, i.e. $F\in\mathscr F(\Rbbb^n)$ is a properly nested sequence $F^{(1)}\subset\dots\subset F^{(n-1)}$ of linear subspaces in $\Rbbb^n$, where each $F^{(i)}$ has dimension $i$. When $G=\PSL(n,\Rbbb)$, it is easy to verify that $\mathscr F(\Rbbb^n)=\mathscr F_{[P_\Delta]}$.

\begin{definition}
A continuous map $\xi:S^1\to\mathscr F(\Rbbb^n)$ is \emph{Frenet} if the following hold:
\begin{itemize}
	\item For all $x_1,\dots,x_k\in S^1$ pairwise distinct and $m_1,\dots,m_k\in\Zbbb^+$ such that $k\leq n$ and $m_1+\dots+m_k=n$, we have that
\[\bigoplus_{i=1}^k\xi(x_i)^{(m_i)}=\Rbbb^n.\]
	\item Let $m_1,\dots,m_k\in\Zbbb^+$ such that $k\leq n$ and $m_1+\dots+m_k=m\leq n$, and let $\{(x_1^j,\dots,x_k^j)\}_{j=1}^\infty$ be a sequence a $k$-tuples of pairwise distinct points. If there is some $x\in S^1$ so that $\lim_{j\to\infty}x_i^j=x$ for all $i=1,\dots,k$, then
\[\lim_{j\to\infty}\bigoplus_{i=1}^k\xi(x_i^j)^{(m_i)}=\xi(x)^{(m)}.\]
\end{itemize}
\end{definition}

Labourie (Theorem 4.1 of \cite{Lab1}) proved that $\PSL(n,\Rbbb)$-Hitchin representations preserve a $\rho$-equivariant Frenet curve. Later, Guichard (Theorem 1 of \cite{Gui1}) proved the converse to this, thus giving us the following theorem.

\begin{thm}[Labourie, Guichard]\label{Labourie Guichard}
Let $\rho\in\Xmc(S,\PSL(n,\Rbbb))$. Then $\rho\in\Hit_n(S)$ if and only if there is a $\rho$-equivariant Frenet curve $\xi:\partial\Gamma\to\mathscr F_{[P_\Delta]}$.
\end{thm}

As a consequence of this, we know that every $\rho\in\Hit_n(S)$ is $[P_\Delta]$-Anosov. In particular, for all $\rho\in\Hit_n(S)$ and $\alpha\in\Delta$, we can define $\ell^\rho_\alpha$ and the corresponding cross ratios $B^\rho_\alpha$ as per Section \ref{Anosov representations} and Section \ref{Cross ratios and positively ratioed} respectively. In fact, we have the following theorem.

\begin{thm}\label{Hitchin positively ratioed}
If $\rho\in \mathrm{Hit}_n(S)$, then $\rho$ is $[P_\Delta]$-positively ratioed.
\end{thm}

To prove Theorem \ref{Hitchin positively ratioed}, we will use Theorem \ref{Labourie Guichard} to construct positive cross ratios $B_i^\rho:\partial\Gamma^{[4]}\to\Rbbb$ for $i=1,\dots,n-1$ in the following way.

\begin{notation} Given flags $(F_1,F_2,\dots, F_s)$ for every $l=1,\dots,s$ choose vectors $f_{1,l},\dots,f_{n-1,l}\in\Rbbb^n$ so that 
\[
\Span_\Rbbb\{f_{1,l},\dots,f_{i,l}\}=F_l^{(i)}
\]
for all $i=1,\dots,n-1$. Fix once and for all an identification $\bigwedge^n\Rbbb^n\cong \Rbbb$. For any integers $i_l\geq 0$ with $\sum_li_l=n$, denote by $F_1^{(i_1)}\wedge \dots \wedge F_s^{(i_s)}$ the real number 
\[
f_{1,1}\wedge\dots\wedge f_{i_1,1}\wedge \dots\wedge f_{1,s}\wedge \dots\wedge f_{i_s,s}.
\]
This notation involves some choices, but none of the quantities we define using this notation will depend on them.
\end{notation}

Let $\tilde{B}^\rho_{i}\colon \partial\Gamma^{[4]}\to\Rbbb$ be the function
\[
\tilde{B}^\rho_i(x,y,z,w):=\log\left\vert\frac{\xi(x)^{(n-i)}\wedge\xi(z)^{(i)}}{\xi(x)^{(n-i)}\wedge\xi(w)^{(i)}}\frac{\xi(y)^{(n-i)}\wedge\xi(w)^{(i)}}{\xi(y)^{(n-i)}\wedge\xi(z)^{(i)}}\right\vert
\]
and set $B^\rho_i:=\frac{1}{2}\left(\tilde{B}^\rho_i+\tilde{B}^\rho_{n-i}\right)$. 

\begin{lem}\label{lem:Hitchin non-negativity} 
For $i=1,\dots,n-1$, $B^\rho_i$ is a positive cross ratio. 
\end{lem}

\begin{proof} Additivity and symmetry of $B^\rho_i$ are easy to check thanks to the explicit formula above. Continuity of $B^\rho_i$ follows from the continuity of $\xi$. Hence $B^\rho_i$ is a cross ratio for every $i=1,\dots, n-1$. To show positivity of $B_i^{\rho}$, we will write it as a sum of functions on $\partial\Gamma^{[4]}$ that are positive when evaluated on points $x,y,z,w$ lying in this cyclic order along $\partial\Gamma$. 

Fix $i\in\{1,\dots,n-1\}$. For $(x,y,z,w)\in\partial\Gamma^{[4]}$, $k=1,\dots,n-i$ and $j=1,\dots,i$, define
\begin{gather*}
(x,y,z,w)_{i,j,k}:=\log\left\vert\frac{\xi(x)^{(n-i-k+1)}\wedge\xi(y)^{(k-1)}\wedge \xi(z)^{(i-j+1)}\wedge \xi(w)^{(j-1)}}{\xi(x)^{(n-i-k+1)}\wedge\xi(y)^{(k-1)}\wedge \xi(z)^{(i-j)}\wedge \xi(w)^{(j)}}\right\vert\\
\ \ \ \ \ \ \ \ \ \ \ \ \ \ \ \ \ \ \ \ \ \ \ \ \ \ \ \ \ +\log\left\vert\frac{\xi(x)^{(n-i-k)}\wedge\xi(y)^{(k)}\wedge \xi(z)^{(i-j)}\wedge \xi(w)^{(j)}}{\xi(x)^{(n-i-k)}\wedge\xi(y)^{(k)}\wedge \xi(z)^{(i-j+1)}\wedge \xi(w)^{(j-1)}}\right\vert,
\end{gather*}
and observe that for all $(x,y,z,w)\in \partial\Gamma^{[4]}$
\[
\tilde{B}_i^\rho(x,y,z,w)=\sum_{k=1}^{n-i}\sum_{j=1}^i(x,y,z,w)_{i,j,k}.
\]
The functions $(x,y,z,w)_{i,j,k}$ were studied by the second author, who proved (Proposition 2.12 of \cite{Zha1}), that each $(x,y,z,w)_{i,j,k}$ is positive on quadruples of points $x,y,z,w$ in this cyclic order along $\partial\Gamma$. This shows the positivity of $B_i^\rho$.
\end{proof}

\begin{proof}[Proof of Theorem \ref{Hitchin positively ratioed}]
By Lemma \ref{lem:Hitchin non-negativity}, it is sufficient to show that $B^\rho_{\alpha_{i,i+1}}=B^\rho_i$. For any element $g\in\PSL(n,\Rbbb)$, let $\lambda(g)=(\lambda_1(g), \dots, \lambda_n(g))\in\overline{\amf^+}$ be its Jordan projection. An easy computation, using the explicit formula for the restricted fundamental weights, shows that for all $\gamma\in\Gamma\setminus\{\id\}$,
\[
\ell^\rho_{\alpha_{i,i+1}}[[\gamma]]=(\omega_{\alpha_{i,i+1}}+\omega_{\alpha_{n-i,n-i+1}})\circ \lambda \circ \rho (\gamma)=\sum_{k=1}^i\lambda_k(\rho(\gamma))-\sum_{k=n-i+1}^n\lambda_k(\rho(\gamma)).
\]
By Theorem \ref{otal}, it is thus sufficient to show that $\ell_{B^\rho_i}=\ell^\rho_{\alpha_{i,i+1}}$.

Choose a basis $e_1,\dots,e_n\subset \Rbbb^n$ such that $e_i$ spans the line $\xi(\gamma^+)^{(i)}\cap \xi(\gamma^{-})^{(n-i+1)}$. Then for this basis we have
\begin{gather*}
\big|\xi(\gamma^-)^{(n-i)}\wedge\xi(\gamma\cdot x)^{(i)}\big|=\left|e^{\lambda_{1}(\rho(\gamma))+\dots+\lambda_{i}(\rho(\gamma))}\xi(\gamma^-)^{(n-i)}\wedge\xi(x)^{(i)}\right|;\\
\big|\xi(\gamma^+)^{(n-i)}\wedge\xi(\gamma\cdot x)^{(i)}\big|=\left|e^{\lambda_{n-i+1}(\rho(\gamma))+\dots+\lambda_{n}(\rho(\gamma))}\xi(\gamma^+)^{(n-i)}\wedge\xi(x)^{(i)}\right|.
\end{gather*}
Hence, 
\begin{align*}
2B_i^\rho(\gamma^-,\gamma^+,\gamma x,x)&= \log\left\vert\frac{\xi(\gamma^-)^{(n-i)}\wedge\xi(\gamma\cdot x)^{(i)}}{\xi(\gamma^-)^{(n-i)}\wedge\xi(x)^{(i)}}\frac{\xi(\gamma^+)^{(n-i)}\wedge\xi(x)^{(i)}}{\xi(\gamma^+)^{(n-i)}\wedge\xi(\gamma\cdot x)^{(i)}}\right\vert \\
&\qquad+\log\left\vert\frac{\xi(\gamma^-)^{(i)}\wedge\xi(\gamma\cdot x)^{(n-i)}}{\xi(\gamma^-)^{(i)}\wedge\xi(x)^{(n-i)}}\frac{\xi(\gamma^+)^{(i)}\wedge\xi(x)^{(n-i)}}{\xi(\gamma^+)^{(i)}\wedge\xi(\gamma\cdot x)^{(n-i)}}\right\vert\\
&=\sum_{k=1}^{i}\lambda_k(\rho(\gamma))-\sum_{k=n-i+1}^n\lambda_k(\rho(\gamma))\\
&\qquad +\sum_{k=1}^{n-i}\lambda_k(\rho(\gamma))-\sum_{k=i+1}^n\lambda_k(\rho(\gamma))\\
&=2\left(\sum_{k=1}^i\lambda_k(\rho(\gamma))-\sum_{k=n-i+1}^n\lambda_k(\rho(\gamma))\right)=2\ell^\rho_{\alpha_{i,i+1}}[[\gamma]].\qedhere
\end{align*}
\end{proof}

\subsection{Maximal representations}

Another important feature of $\PSL(2,\Rbbb)$ is that it is a Lie group of Hermitian type.

\begin{definition} A connected semisimple Lie group $G$ is of \emph{Hermitian type} if it has finite center, has no compact factors and the associated Riemannian symmetric space $X$ admits a $G$-invariant complex structure.
\end{definition}

For our purposes, the main example of Lie group of Hermitian type will be $G=\PSp(2n,\Rbbb)$. Let $g$ be the Riemannian metric on the symmetric space $X$ and $J$ the $G$-invariant complex structure. This allows us to define a non-degenerate two-form $\omega_X$ by
\[
\omega_X(V,W):=g(JV,W)
\]
for any two vector fields $V$, $W$ on $X$. One can show (Lemma 2.1 of Burger-Iozzi-Labourie-Wienhard \cite{BILW}) that $(X, \omega_X)$ is a K\"ahler manifold. For any representation $\rho\colon \Gamma\to G$, the symplectic form $\omega_X$ defines an important invariant for $\rho$ as follows. Consider the bundle $S \times_\Gamma X:=(\Std\times X)/\Gamma$ over $S$, where $\Gamma$ acts on $\Std$ via deck transformations and on $X$ via $\rho$. The fiber of this bundle is $X$, which is contractible, so $S \times_\Gamma X$ admits a smooth section. Equivalently, there exists a smooth $\Gamma$-equivariant map $f\colon\Std\to X$. The pull back $f^*(\omega_X)$ is a $\Gamma$-invariant two-form on $\Std$, which descends to the two-form $\widehat{f^*(\omega_X)}$  on the compact surface $S$. We can define the \emph{Toledo invariant} of $\rho$ as
\[
T(\rho):=\frac{1}{2\pi}\int_S\widehat{f^*(\omega_X)}.
\]
Since any two $\rho$-equivariant maps $f,f'\colon \Std\to X$ are homotopic, $T(\rho)$ is well-defined. 

If $\text{rank}_\Rbbb X$ is the real rank of the symmetric space $X$, the Toledo invariant satisfies the inequality
\[
\vert T(\rho)\vert\leq -\chi(S)\text{rank}_\Rbbb X
\]
(see Turaev \cite{Tur1}, Dominic-Toledo \cite{DT1}, Clerc-{\O}rsted \cite{CO1}). In the case $G=\PSL(2,\Rbbb)$, this is the classical Milnor-Wood inequality \cite{Mil1}. Goldman \cite{Gol1} showed that $\mathcal T(S)$ is the unique connected component of $\Xmc(\Gamma,\PSL(2,\Rbbb))$ with Toledo invariant $2g-2$ (the real rank in this case is 1). This motivated Burger-Iozzi-Wienhard \cite{BIW} to define the following class of representations.

\begin{definition} A representation $\rho\colon\Gamma\to G$, with $G$ a Lie group of Hermitian type is \emph{maximal} if $\vert T(\rho)\vert= -\chi(S)\text{rank}_\Rbbb X$. 
\end{definition}

\textbf{For the rest of this section, fix the target group to be $G=\PSp(2n,\Rbbb)$.} We will show that in this case, maximal representations are also positively ratioed with respect to a particular parabolic subgroup. Recall that the maximal compact subgroup of $\PSp(2n,\Rbbb)$ is isomorphic to $U(n)$. At the level of Lie algebras, we can write $\mathfrak{sp}(2n,\Rbbb)=\kmf + \pmf$ with
\begin{align*}
\mathfrak u(n)\cong\kmf&=\Big\{\begin{pmatrix} A & B \\ -B & A\end{pmatrix}\colon A,B\in M_n(\Rbbb), A^t=-A, \ B^t=B\Big\}\\
\pmf&=\Big\{\begin{pmatrix} A & B \\ B & -A\end{pmatrix}\colon A,B\in M_n(\Rbbb), A^t=A, \ B^t=B\Big\}
\end{align*}
where $M_n(\Rbbb)$ is the set of $n\times n$ matrices. The maximal abelian subspace $\mathfrak a\subset\pmf$ can therefore be identified with diagonal, traceless matrices in $\mathfrak{sp}(2n,\Rbbb)$. With this identification, the restricted simple roots can be chosen to be 
\[
\alpha_{i,i+1}(x_1,\dots,x_n,-x_n,\dots,-x_1)=x_i-x_{i+1}\text{ for }i=1,\dots,n.
\] 
Moreover, the opposition involution is the identity.

Burger-Iozzi-Labourie-Wienhard [Theorem 6.1 of \cite{BILW}] proved that maximal representations are Anosov.

\begin{thm} [Burger-Iozzi-Labourie-Wienhard] \label{maximal is Anosov} 
If $\rho\colon\Gamma\to G$ is a maximal representation, then $\rho$ is $[P_{\alpha_{n,n+1}}]$-Anosov.
\end{thm}

The quotient $\mathscr F_{[P_{\alpha_{n,n+1}}]}$ is the Grassmannian of Lagrangian subspaces in $\Rbbb^{2n}$. Consider four Lagrangian subspaces $L_1,L_2,L_3,L_4\in \mathscr F_{[P_{\alpha_{n,n+1}}]}$ so that $L_1,L_3$ and $L_2,L_4$ are transverse pairs of Lagrangians, and let $(e_j^1,\dots,e_j^n)$ be a basis for $L_j$. For any $i,j=1,\dots,4$, let $A_{i,j}$ be the matrix whose $(k,m)$-th entry is
\[
(A_{i,j})_{k,m}=\Omega(e_{i}^k,e_j^{m}),
\]
where $\Omega$ is the symplectic form on $\Rbbb^{2n}$ preserved by the $\Sp(2n,\Rbbb)$ action. Using this, define
\[
\mathbb B(L_1,L_2,L_3,L_4):=\frac{\det(A_{1,2})\cdot \det(A_{3,4})}{\det(A_{1,4})\cdot\det(A_{3,2})}.
\]
Labourie (Section 4.2 of \cite{Lab3}) proved the following.

\begin{thm}[Labourie]\label{Labourie maximal} 
If $\rho\colon \Gamma\to G$ is a maximal representation with flag curve $\xi\colon \partial\Gamma\to \mathscr F_{[P_{\alpha_{n,n+1}}]}$, then 
\[
B(x,y,z,w):=\log\big\vert \mathbb B(\xi(x),\xi(z),\xi(y),\xi(w))\big\vert
\]
is a cross ratio. Also, $\ell_B(c)=2(\lambda_1(\rho(\gamma))+\lambda_2(\rho(\gamma))+\dots+\lambda_n(\rho(\gamma)))$ for all $c=[[\gamma]]\in\Cmc\Gmc(S)$. Moreover, for any four distinct points $x,y,z,w$ in this cyclic order along $\partial\Gamma$, we have that $B(x,y,z,w)>0$.
\end{thm}

Combining Theorem \ref{maximal is Anosov} and Theorem \ref{Labourie maximal}, we obtain the following corollary.

\begin{cor} If $\rho\colon \Gamma\to G$ is a maximal representation, then $\rho$ is $[P_{\alpha_{n,n+1}}]$- positively ratioed. 
\end{cor}

\begin{proof}
The restricted fundamental weight $\omega_{\alpha_{n,n+1}}$ corresponding to $\alpha_{n,n+1}$ is given by $\omega_{\alpha_{n,n+1}}(x_1,\dots,x_n,-x_n,\dots, -x_1)=x_1+\dots+x_n$ and therefore
\[
\ell^\rho_{\alpha_{n,n+1}}[[\gamma]]=2(\lambda_1(\rho(\gamma))+\lambda_2(\rho(\gamma))+\dots+\lambda_n(\rho(\gamma))).
\]
Theorem \ref{otal} then implies that $B^\rho_{\alpha_{n,n+1}}$ is the cross ratio $B$ defined in Theorem \ref{Labourie maximal}, which is a positive cross ratio. 
\end{proof}

\section{Background on geodesic currents}\label{Background on geodesic currents}
In this section, we will introduce some notation, terminology and basic lemmas to study length functions on subsurfaces of $S$ induced by geodesic currents on $S$. 

\subsection{Extension to subsurfaces}
We begin by a definition for the kind of subsurfaces of $S$ that we consider.

\begin{definition}
Let $\Dmc$ be a (possibly empty) collection of pairwise non-intersecting, pairwise non-homotopic, non-contractible, simple, closed curves on $S$. An \emph{essential subsurface} $S'$ of $S$ is a union of connected components of $S\setminus\Dmc$.
\end{definition}

If $S'$ is connected, let $\Gamma'$ be the fundamental group of $S'$ and let $\Std'$ denote the universal cover of $S'$. By choosing appropriate base points in $\Std$ and $\Std'$, the inclusion $S'\subset S$ induces inclusions $\Gamma'\subset\Gamma$ and $\Std'\subset\Std$. Also, the inclusion $\Gamma'\subset\Gamma$ realizes the Gromov boundary $\partial\Gamma'$ of $\Gamma'$ as a subset of $\partial\Gamma$. 

If we choose a hyperbolic structure $\Sigma$ on $S$, then any connected essential subsurface $S'$ of $S$ is homotopic to a connected subsurface $\Sigma'\subset\Sigma$ with totally geodesic boundary. Also, denote the universal cover of $\Sigma'$ by $\widetilde{\Sigma}'$, then the inclusion $\Std'\subset\Std$ gives an inclusion $\widetilde{\Sigma}'\subset\widetilde{\Sigma}$ as the convex hull in $\widetilde{\Sigma}$ of $\partial\Gamma'\subset\partial\Gamma\simeq\partial\widetilde{\Sigma}$. 

Previously (see Definitions \ref{closed geodesic} and \ref{geodesic}), we defined a topological notion of geodesics in $\Std$ and $S$, as well as closed geodesics in $S$ using only $\Gamma$. Note that we can define oriented geodesics and geodesics in $\Std'$ and $S'$, as well as closed geodesics in $S'$ in the same way, using $\Gamma'$ in place of $\Gamma$. We will denote the set of geodesics in $\Std'$, the set of geodesics in $S'$, and the set of closed geodesics in $S'$ by $\Gmc(\Std')$, $\Gmc(S')$ and $\Cmc\Gmc(S')$ respectively. 

Since the closed geodesics in $S'$ are in a natural bijection with the free homotopy classes of closed curves on $S'$, we say that a closed geodesic in $S'$ is \emph{simple} if its corresponding free homotopy class contains a simple curve, and we say that it is \emph{peripheral} if its corresponding free homotopy class is peripheral. 

If $S'=\bigcup_{i=1}^k S_i$ is a disconnected union, then we define $\Cmc\Gmc(S'):=\bigcup_{i=1}^k\Cmc\Gmc(S_i)$ and $\Gmc(S'):=\bigcup_{i=1}^k\Gmc(S_i)$.

\textbf{For the rest of this paper, we will use the notations $S'\subset S$, $\Std'\subset\Std$, $\Gamma'\subset\Gamma$, $\partial\Gamma'\subset\partial\Gamma$, $\Cmc\Gmc(S')\subset\Cmc\Gmc(S)$, $\Gmc(\Std')\subset\Gmc(\Std)$ and $\Gmc(S')\subset\Gmc(S)$ as above. Also, whenever we choose a hyperbolic structure on $S$, we will identify  $\Sigma'$, $\Sigma$, $\widetilde{\Sigma}'$ and $\widetilde{\Sigma}$ with $S'$, $S$, $\Std'$ and $\Std$ respectively without any further comment.}

\subsection{Properties of the intersection form}\label{properties of the intersection form}
Although the intersection form (see Definition \ref{def:intersection form}) was defined purely topologically as the measure of the set $\Dmc\Gmc(S)$, it is often convenient to choose a hyperbolic structure $\Sigma$ on $S$. This choice allows us to use the following description of $\Dmc\Gmc(S)$, which will be useful for computing the intersection form. 

The tangent bundle of the Poincar\'e disc $T\Dbbb$ is a vector bundle over $\Dbbb$, so we can projectivize its fibers to obtain a fiber bundle over $\Dbbb$, which we denote by $\Pbbb(T\Dbbb)$. Let $\Pbbb(T\Dbbb)\oplus\Pbbb(T\Dbbb)$ be the fiber bundle over $\Dbbb$ obtained by taking the fiber-wise product of $\Pbbb(T\Dbbb)$ with itself. An element of $\Pbbb(T\Dbbb)\oplus\Pbbb(T\Dbbb)$ is thus a triple $(p,l_1,l_2)$, where $p\in\Dbbb$ and $l_1, l_2$ are lines through the origin in $T_p\Dbbb$. Clearly, the $\PGL(2,\Rbbb)=\mathrm{Isom}(\Dbbb)$ action on $\Pbbb(T\Dbbb)\oplus\Pbbb(T\Dbbb)$ leaves invariant the subset
\[\mathrm{Trans}\big(\Pbbb(T\Dbbb)\oplus\Pbbb(T\Dbbb)\big):=\{(p,l_1,l_2)\in\Pbbb(T\Dbbb)\oplus\Pbbb(T\Dbbb):l_1\neq l_2\}.\]

A choice of a hyperbolic metric $\Sigma$ on $S$ induces a unique (up to post composition by $\PGL(2,\Rbbb)$) isometry between $\Std$ and $\Dbbb$. The action of $\Gamma$ on $\Std$ by deck transformations then conjugates to a free and proper $\Gamma$ action on $\Dbbb$, which in turn induces a free and proper action of $\Gamma$ on $\Pbbb(T\Dbbb)\oplus\Pbbb(T\Dbbb)$ that stabilizes $\mathrm{Trans}\big(\Pbbb(T\Dbbb)\oplus\Pbbb(T\Dbbb)\big)$. This allows us to define the Hausdorff space 
\[\Qmc(\Sigma):=\mathrm{Trans}\big(\Pbbb(T\Dbbb)\oplus\Pbbb(T\Dbbb)\big)/\Gamma.\]

The isometry between $\Std$ and $\Dbbb$ also induces an obvious $\Gamma$-equivariant homeomorphism between $\Dmc\Gmc(\Std)$ and $\mathrm{Trans}\big(\Pbbb(T\Dbbb)\oplus\Pbbb(T\Dbbb)\big)$, which descends to a homeomorphism between $\Dmc\Gmc(S)$ and $\Qmc(\Sigma)$. This identification allows us to prove Lemma \ref{compute length} below. However, to state Lemma \ref{compute length}, we first need to develop some notation.

\begin{notation}\label{interval notation} 
For any $q,p\in\Dbbb$, let $(q,p]$ denote the half-open geodesic in $\Dbbb$ with open endpoint $q$ and closed endpoint $p$. Similarly, $(q,p)$, $[q,p)$ and $[q,p]$ will denote the interval $(q,p]$, but with the appropriate open and closed endpoints. 
\end{notation}

\begin{notation}\label{geodesic notation}
For any $q,p\in\Dbbb$, let $I$ be one of the four geodesic segments in $\Dbbb$ described in Notation \ref{interval notation} with endpoints $p$ and $q$. Then let $G(I)$ denote the set of geodesics in $\Dbbb$ that intersect $I$ transversely. Similarly, for any $x,y,z,w\in\partial\Dbbb$ in that cyclic order, let $L$ be the geodesic in $\Dbbb$ with endpoints $x,w$ and let $J$ be either of the following four subsegments of $\partial\Gamma$:
\[(y,z)_x=(y,z)_w,\,\,\,\,\, [y,z)_x=[y,z)_w,\,\,\,\,\,  (y,z]_x=(y,z]_w\,\,\,\,\, \text{ or }\,\,\,\,\, [y,z]_x=[y,z]_w.\] Then let $G_{\{x,w\}}(J)$ denote the set of geodesics in $\Dbbb$ that intersect $L$ and have one endpoint in $J$. 
\end{notation}

\begin{lem}\label{compute length}
Let $\nu\in\Cmc(S)$ and let $c=[[\gamma]]\in\Cmc\Gmc(S)$. Let $\{x,y\}$ be the set of fixed points of $\gamma$. Also, choose any hyperbolic structure on $S$ and let $L$ be the axis of $\gamma$ in $\Std=\Dbbb$. Finally, let $q\in \Dbbb$ and let $z\in\partial\Dbbb\setminus\{x,y\}$. Then the following hold:
\begin{enumerate}
\item If $q\in L$, then
\[i(c,\nu)=\nu\big(G(q,\gamma\cdot q]\big)=\nu\big(G_{\{x,y\}}(z,\gamma\cdot z]\big).\]
\item If $q\notin L$, then 
\[i(c,\nu)\leq\nu\big(G(q,\gamma\cdot q]\big)\]
and the inequality holds strictly if $\nu$ has full support.
\end{enumerate}
\end{lem}

\begin{proof} Note that $\{x,y\}$ is the set of endpoints for $L$ in $\partial\Dbbb$. 

Proof of (1). Let $\eta\in\Gamma$ be the primitive element so that $\gamma=\eta^k$ for some positive integer $k$, and let $c':=[[\eta]]\in\Cmc\Gmc(S)$. By definition, $i(c,\nu)$ is the mass of a fundamental domain of the $\Gamma$-action on $\Dmc\Gmc(\Std)=\mathrm{Trans}\big(\Pbbb(T\Dbbb)\oplus\Pbbb(T\Dbbb)\big)$ in the measure $c\times\nu$. Since the support of $c\times\nu$ lies in the set 
\[\{(p,l_1,l_2)\in\mathrm{Trans}\big(\Pbbb(T\Dbbb)\oplus\Pbbb(T\Dbbb)\big)\colon\exp_p(l_1)\text{ is a lift of }c\text{ to }\Dbbb\},\]
this means that 
\begin{eqnarray*}
i(c,\nu)&=&(c\times\nu)\big(\{(p,l_1,l_2)\colon p\in (q,\eta\cdot q]\text{ and }\exp_p(l_1)=L\}\big)\\
&=&k(c'\times\nu)\big(\{(p,l_1,l_2)\colon p\in (q,\eta\cdot q]\text{ and }\exp_p(l_1)=L\}\big)\\
&=&(c'\times\nu)\big(\{(p,l_1,l_2)\colon p\in (q,\gamma\cdot q]\text{ and }\exp_p(l_1)=L\}\big)\\
&=&\nu\big(G(q,\gamma\cdot q]\big).
\end{eqnarray*}

Next, we will prove that $\nu\big(G(q,\gamma\cdot q]\big)=\nu\big(G_{\{x,y\}}(z,\gamma\cdot z]\big)$. For any $k\in\Zbbb$, let $G^k(q,\gamma\cdot q]\subset G(q,\gamma\cdot q]$ be the subset of geodesics with one endpoint in $\gamma^k\cdot (z,\gamma\cdot z]_x$. It is clear that $G(q,\gamma\cdot q]$ can be written as the disjoint union
\[G(q,\gamma\cdot q]=\bigcup_{k\in\Zbbb}G^k(q,\gamma\cdot q].\]
Also, for any $k\in\Zbbb$, let $G^k_{\{x,y\}}(z,\gamma\cdot z]\subset G_{\{x,y\}}(z,\gamma\cdot z]$ be the subset of geodesics that intersect $\gamma^k\cdot (q,\gamma\cdot q]$. As before, $G_{\{x,y\}}(z,\gamma\cdot z]$ can be written as the disjoint union 
\[G_{\{x,y\}}(z,\gamma\cdot z]=\bigcup_{k\in\Zbbb}G^k_{\{x,y\}}(z,\gamma\cdot z].\]

Finally, observe that $G^k(q,\gamma\cdot q]=\gamma^k\cdot G_{\{x,y\}}^{-k}(z,\gamma\cdot z]$. Hence,
\begin{eqnarray*}
\nu\big(G(q,\gamma\cdot q]\big)&=&\sum_{k\in\Zbbb}\nu\big(G^k(q,\gamma\cdot q]\big)\\
&=&\sum_{k\in\Zbbb}\nu\big(\gamma^k\cdot G_{\{x,y\}}^{-k}(z,\gamma\cdot z]\big)\\
&=&\sum_{k\in\Zbbb}\nu\big(G_{\{x,y\}}^{-k}(z,\gamma\cdot z]\big)\\
&=&\nu\big(G_{\{x,y\}}(z,\gamma\cdot z]\big).
\end{eqnarray*}

Proof of (2). Let $p$ be the foot of the perpendicular from $q$ to $L$ and let $L'$ be the bi-infinite geodesic through $q$ and $\gamma\cdot q$. Observe that $\gamma\cdot p$ is also the foot of the perpendicular from $\gamma\cdot q$ to $L$, and $L\cap L'$ is empty. Let $z,w\in\partial\Dbbb$ be the endpoints of $L'$ so that $z,x,y,w\in\partial\Dbbb$ in that order (see Figure \ref{figure1}). By (1), we know that $i(c,\nu)=\nu\big(G(p,\gamma\cdot p]\big)$, so it is sufficient to show that $\nu\big(G(q,\gamma\cdot q]\big)\geq\nu\big(G(p,\gamma\cdot p]\big)$, and that this inequality is strict when $\nu(U)>0$ for all open sets $U\subset\Gmc(\Std)$.

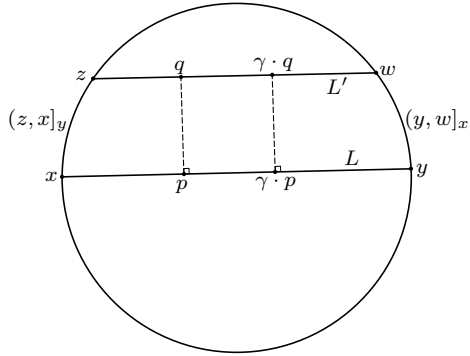
\begin{figure}
\includegraphics[scale=0.5]{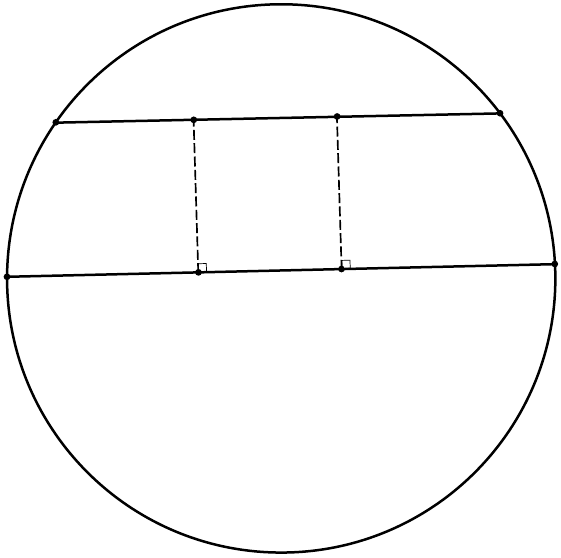}
\put (-244, 69){\makebox[0.7\textwidth][r]{\footnotesize$y$ }}
\put (-384, 66){\makebox[0.7\textwidth][r]{\footnotesize$x$ }}
\put (-373, 104){\makebox[0.7\textwidth][r]{\footnotesize$z$ }}
\put (-256, 106){\makebox[0.7\textwidth][r]{\footnotesize$w$ }}
\put (-336, 108){\makebox[0.7\textwidth][r]{\footnotesize$q$ }}
\put (-295, 109){\makebox[0.7\textwidth][r]{\footnotesize$\gamma\cdot q$ }}
\put (-335, 63){\makebox[0.7\textwidth][r]{\footnotesize$p$ }}
\put (-294, 64){\makebox[0.7\textwidth][r]{\footnotesize$\gamma\cdot p$ }}
\put (-270, 72){\makebox[0.7\textwidth][r]{\footnotesize$L$ }}
\put (-274, 99){\makebox[0.7\textwidth][r]{\footnotesize$L'$ }}
\put (-380, 87){\makebox[0.7\textwidth][r]{\footnotesize$(z,x]_y$ }}
\put (-228, 87){\makebox[0.7\textwidth][r]{\footnotesize$(y,w]_x$ }}
\caption{Proof of (2) of Lemma \ref{compute length}}
\label{figure1}
\end{figure}

For any $k\in\Zbbb$, let $G_k(p,\gamma\cdot p]\subset G(p,\gamma\cdot p]$ be the subset of geodesics that intersect $\gamma^k\cdot (q,\gamma\cdot q]$. Then $G(p,\gamma\cdot p]$ can be written as the disjoint union
\[G(p,\gamma\cdot p]=\bigcup_{k\in\Zbbb}G_k(p,\gamma\cdot p].\]
Similarly, for any $k\in\Zbbb$, let $G_k(q,\gamma\cdot q]\subset G(q,\gamma\cdot q]$ be the subset of geodesics that intersect $\gamma^k\cdot (p,\gamma\cdot p]$. Also, let $A\subset G(q,\gamma\cdot q]$ be the subset of geodesics with one endpoint in $(z,x]_y$ and let $B\subset G(q,\gamma\cdot q]$ be the subset of geodesics with one endpoint in $[y,w)_x$. Observe that $G(q,\gamma\cdot q]$ can again be written as the disjoint union
\[G(q,\gamma\cdot q]=A\cup B\cup\bigcup_{k\in\Zbbb}G_k(q,\gamma\cdot q].\]

Since $G_k(p,\gamma\cdot p]=\gamma^k\cdot G_{-k}(q,\gamma\cdot q]$ for all $k\in\Zbbb$, we have 
\begin{eqnarray*}
\nu\big(G(q,\gamma\cdot q]\big)&=&\nu(A)+\nu(B)+\sum_{k\in\Zbbb}\nu\big(G_k(q,\gamma\cdot q]\big)\\
&=&\nu(A)+\nu(B)+\sum_{k\in\Zbbb}\nu\big(\gamma^k\cdot G_{-k}(q,\gamma\cdot q]\big)\\
&=&\nu(A)+\nu(B)+\sum_{k\in\Zbbb}\nu\big(G_k(p,\gamma\cdot p]\big)\\
&=&\nu(A)+\nu(B)+\nu\big(G(p,\gamma\cdot p]\big)\\
&\geq&\nu\big(G(p,\gamma\cdot p]\big).
\end{eqnarray*}
It is clear that $A$ and $B$ contain open subsets of geodesics in $\Gmc(S)$, so the strictness statement holds.
\end{proof}

\subsection{Surgery and lengths}\label{surgery and lengths}
Let $c\in\Cmc\Gmc(S')$ be a primitive closed geodesic of $S'$ with positive geometric self-intersection number. Choose a representative $\bar{c}$ in the free homotopy class of closed curves corresponding to $c$, so that $\bar{c}$ has only transverse self-intersections and minimal self-intersection number. We can also assume that $\bar{c}$ only has simple self-intersection points, i.e. if we choose a parameterization of $c$ by $\Sbbb^1$, then $c(t_1)=c(t_2)=c(t_3)$ implies that $t_1\in\{t_2,t_3\}$. Let $p$ be a point of self-intersection for $\bar{c}$. 

There is a well-known procedure one can apply to $\bar{c}$ known as \emph{surgery at $p$} to obtain new closed curves in $S'$. To do so, choose a small topological disc in $U\subset S'$ so that $\bar{c}\cap\partial U$ is four points $x,y,z,w$ that lie along $\partial U$ in that order, and $\bar{c}\cap U$ is the union of two simple paths that intersect at $p$, one with endpoints $x$ and $z$, and the other with endpoints $y$ and $w$. We can then modify the curve $\bar{c}$ by replacing the two simple paths $\bar{c}\cap U$ that intersect at $p$ with two simple paths in $U$ that do not intersect. There are two ways to do so; we can either replace $\bar{c}\cap U$ with two simple, non-intersecting paths in $U$ with endpoints $x,y$ and $z,w$, or we can replace $\bar{c}\cap U$ with two simple, non-intersecting paths in $U$ with endpoints $y,z$ and $x,w$. 

These two different ways of performing surgery to $\bar{c}$ at $p$ yield either one closed curve $\bar{c}_1$ in $S'$ or two closed curves $\bar{c}_2$ and $\bar{c}_3$ in $S'$. For $i=1,2,3$, let $c_i\in\Cmc\Gmc(S')$ correspond to the free homotopy class of closed curves in $S'$ that contains $\bar{c}_i$ (see Figure \ref{figure2}). It is easy to see that $c_1$, $c_2$ and $c_3$ do not depend on the choice of $U$. Moreover, an easy homotopy argument shows that $c_1$, $c_2$ and $c_3$ also do not depend on $\bar{c}$ in the following sense. If $\bar{c}'$ is another closed curve in the free homotopy class corresponding to $c$ with minimal geometric self-intersection number and only simple, transverse self-intersection points, then the homotopy between $\bar{c}$ and $\bar{c}'$ gives a bijection 
\[h_{\bar{c},\bar{c}'}:\{\text{self-intersection points of }\bar{c}\}\to\{\text{self-intersection points of }\bar{c}'\}.\] 
If we perform surgery to $\bar{c}'$ at the self-intersection point $h_{\bar{c},\bar{c}'}(p)$ in both ways, then the closed geodesics corresponding to the free homotopy classes of closed curves we obtain are exactly $c_1$, $c_2$ and $c_3$. 

\begin{figure}
\includegraphics[scale=0.5]{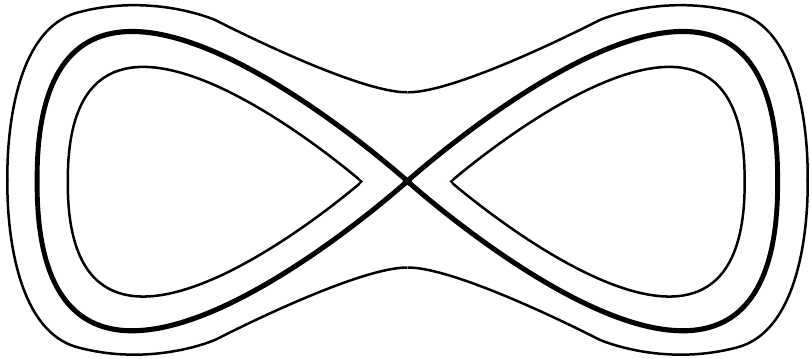}
\put (-345, 35){\makebox[0.7\textwidth][r]{\footnotesize$p$ }}
\put (-332, 57){\makebox[0.7\textwidth][r]{\footnotesize$c$ }}
\put (-418, 50){\makebox[0.7\textwidth][r]{\footnotesize$c_2$ }}
\put (-267, 50){\makebox[0.7\textwidth][r]{\footnotesize$c_3$ }}
\put (-332, 70){\makebox[0.7\textwidth][r]{\footnotesize$c_1$ }}
\caption{Surgery}
\label{figure2}
\end{figure}

The following proposition gives useful inequalities involving $i(c_1,\nu)$, $i(c_2,\nu)$, $ i(c_3,\nu)$ and $i(c,\nu)$.

\begin{prop}\label{surgery}
Let $\nu\in\Cmc(S)$ and let $c\in\Cmc\Gmc(S)$ be a primitive closed geodesic so that $i(c,c)>0$. By performing surgery to $c$ at a point of self-intersection in two different ways (see discussion above), we obtain either a single geodesic $c_1$ or a pair of geodesics $c_2$, $c_3$. Then 
\[i(c_1,\nu)\leq i(c,\nu)\,\,\,\text{ and }\,\,\,i(c_2,\nu)+i(c_3,\nu)\leq i(c,\nu).\]
Furthermore, these inequalities are strict when $\nu$ has full support.
\end{prop}

\begin{proof}
Let $\gamma\in\Gamma$ be a group element so that $[[\gamma]]=c\in\Cmc\Gmc(S)$. Choose a hyperbolic structure on $S$, and let $\bar{c}$ be a closed curve homotopic to $c$ with transverse, minimal self-intersection and only simple self-intersection points. The homotopy between $\bar{c}$ and $c$ gives a surjection
\[h_{\bar{c},c}:\{\text{self-intersection points of }\bar{c}\}\to\{\text{self-intersection points of }c\}.\] 
Let $q$ be the self-intersection point of $\bar{c}$ where the surgeries to obtain $c_1$, $c_2$ and $c_3$ are performed, and let $p=h_{\bar{c},c}(q)$.

Let $L\subset\Std=\Dbbb$ be the axis of $\gamma$, and observe that $L$ is a lift of the geodesic $c$. Let $\ptd$ be a point in $L$ whose image under the covering map $\Pi:\Std\to S$ is $p$. Then $\gamma\cdot\ptd$ also lies in $L$ and $\Pi(\gamma\cdot\ptd)=p$ as well. Let $\gamma_2,\gamma_3\in\Gamma$ be the group elements so that $[[\gamma_2]]=c_2$, $[[\gamma_3]]=c_3$, $\gamma=\gamma_3\gamma_2$ and $\gamma_2\cdot\ptd\in (\ptd,\gamma\cdot\ptd]$ (see Notation \ref{interval notation}). It is clear that $\Pi(\gamma_2\cdot\ptd)=p$ and $[[\gamma_1:=\gamma_3^{-1}\gamma_2]]=c_1$.

We will first prove the inequality $i(c_2,\nu)+i(c_3,\nu)\leq i(c,\nu)$. By (2) of Lemma \ref{compute length}, we have
\begin{eqnarray*}
i(c_2,\nu)+i(c_3,\nu)&\leq&\nu\big(G(\ptd,\gamma_2\cdot\ptd]\big)+\nu\big(G(\gamma_2\cdot\ptd,\gamma_3\gamma_2\cdot\ptd]\big)\\
&=&\nu\big(G(\ptd,\gamma\cdot \ptd]\big)\\
&=&i(c,\nu).
\end{eqnarray*}
To prove the inequality $i(c_1,\nu)\leq i(c,\nu)$, first make the following observation. For any $\qtd\in\Dbbb$, let $L'$ be the geodesic in $\widetilde{S}$ containing $[\qtd,\gamma_3\cdot\qtd]$. If $\nu(\{L'\})=0$, then $\nu(\gamma_3\cdot \{L'\})=0$ as well, so the $\nu$-measure of the set of geodesics through $\gamma_3\cdot\widetilde{q}$ that are transverse to $L'$, is equal to the $\nu$-measure of the set of geodesics through $\gamma_3\cdot\widetilde{q}$, which is again equal to the $\nu$-measure of the set of geodesics through $\gamma_3\cdot\widetilde{q}$ transverse to $\gamma_3\cdot L'$. This is in turn equal to the $\nu$-measure of the set of geodesics through $\widetilde{q}$ transverse to $L'$. Hence, we may conclude that if $\nu(\{L'\})=0$, then
\[\nu\big(G(\qtd,\gamma_3^{-1}\cdot\qtd]\big)=\nu\big(G[\qtd,\gamma_3\cdot\qtd)\big)=\nu\big(G(\qtd,\gamma_3\cdot\qtd]\big).\]

Now, observe that the geodesic containing $[\gamma_2\cdot\ptd,\gamma_3\gamma_2\cdot\ptd]$ is $L$. Hence, the previous observation, together with (2) Lemma \ref{compute length}, allows us to conclude that if $\nu(\{L\})=0$, then
\begin{eqnarray*}
i(c_1,\nu)&\leq&\nu\big(G(\ptd,\gamma_1\cdot\ptd]\big)\\
&\leq&\nu\big(G(\ptd,\gamma_2\cdot\ptd]\big)+\nu\big(G(\gamma_2\cdot\ptd,\gamma_3^{-1}\gamma_2\cdot\ptd]\big)\\
&=&\nu\big(G(\ptd,\gamma_2\cdot\ptd]\big)+\nu\big(G(\gamma_2\cdot\ptd,\gamma_3\gamma_2\cdot\ptd]\big)\\
&=&\nu\big(G(\ptd,\gamma\cdot \ptd]\big)\\
&=&i(c,\nu).
\end{eqnarray*}

With this, we can prove $i(c_1,\nu)\leq i(c,\nu)$ for general $\nu$. Since $c_1$ is obtained from $c$ by performing surgery, it is clear that $i(c_1,c)\leq i(c,c)$. Also, since $L$ is a lift of $c$, we can write $\nu=\mu+kc$ for some $\mu\in\Cmc(S)$ so that $\mu(\{L\})=0$. This means that 
\[i(c_1,\nu)=i(c_1,\mu)+ki(c_1,c)\leq i(c,\mu)+ki(c,c)=i(c,\nu).\]

Finally, we argue that these inequalities are strict when $\nu$ has full support. By the strictness statement in (2) of Lemma \ref{compute length}, it is sufficient to show that $\ptd$ does not lie along the axes of $\gamma_1$, $\gamma_2$ and $\gamma_3$. This is obvious, since the axes of $\gamma$, $\gamma_1$, $\gamma_2$ and $\gamma_3$ are pairwise disjoint.
\end{proof}

As a consequence of Theorem \ref{positively ratioed to geodesic currents}, Proposition \ref{strictly positively ratioed} and Proposition \ref{surgery}, we have Corollary \ref{positively ratioed surgery}.

\subsection{Asymptotics of lengths}\label{sec:BurgPoz} 
In \cite{BurPoz1}, Burger and Pozzetti consider a metric compactification of the space of $\mathrm{Sp}(2n,\mathbb R)$-maximal representations. The limit points correspond to actions via isometries of $\Gamma$ on certain asymptotic cones. They prove (Theorem 1.1 of \cite{BurPoz1}) that a boundary point determines a decomposition of $S$ into essential subsurfaces. This decomposition is defined in terms of the asymptotic behavior of the length function. 

In this section, we obtain an analogous result for sequences of positively ratioed representations as a consequence of Theorem \ref{positively ratioed to geodesic currents}. Here, we use the compactness of the space of \emph{projectivized geodesic currents} $\mathcal{PC}(S):=\Cmc(S)/\Rbbb^+$ (Corollary 5 of \cite{Bon2}) to describe the limit points.
First, we need the following lemma.

\begin{lem}\label{panted systole}
Let $\nu\in\Cmc(S)$ and $\bar{e}\in\Cmc\Gmc(S')$ be a primitive non-simple curve. Then there is some geodesic pair of pants $P\subset S'$ and $e\in\Cmc\Gmc(P)$ so that
\begin{itemize}
\item $i(e,\nu)\leq i(\bar{e},\nu)$
\item $e$ is primitive and has a unique self-intersection point $p$
\item the three closed geodesics $e_1$, $e_2$ and $e_3$ obtained by performing surgery to $e$ at $p$ in the two different ways specified in Section \ref{surgery and lengths} are the three boundary components of $P$.
\item if a curve $c\in \Cmc\Gmc(S')$ intersects $\bar{e}$ transversely, then $c$ intersects $e$ transversely.
\end{itemize}
\end{lem}

\begin{proof}
Let $q$ be any self intersection point of $\bar{e}$ and let $\bar{e}_1$, $\bar{e}_2$ and $\bar{e}_3$ be the three closed geodesics obtained by performing surgery to $\bar{e}$ at $q$. It is clear that the self-intersection numbers of $\bar{e}_1$, $\bar{e}_2$ and $\bar{e}_3$ are less than that of $\bar{e}$. Also,  Proposition \ref{surgery} implies that $i(\bar{e}_j,\nu)\leq i(\bar{e},\nu)$ for all $j=1,2,3$. Suppose that there is some $j_0=1,2,3$ so that $\bar{e}_{j_0}$ is not a multiple of a simple curve. Then let $\bar{\bar{e}}$ be the closed geodesic so that $\bar{\bar{e}}=[[\gamma]]$ for some primitive $\gamma\in\Gamma$ with the property that $[[\gamma^k]]=\bar{e}_{j_0}$ for some $k\in\Zbbb$. Then $\bar{\bar{e}}$ is primitive, non-simple, has fewer self-intersection points than $\bar{e}$, and $i(\bar{\bar{e}},\nu)\leq i(\bar{e}_{j_0},\nu)\leq i(\bar{e},\nu)$. Replace $\bar{e}$ with $\bar{\bar{e}}$.

Iterate the replacement procedure above. This iteration will terminate to give a non-simple $e\in\Cmc\Gmc(S')$ so that $i(e,\nu)\leq i(\bar{e},\nu)$, and for any self-intersection point $p$ of $e$, the three closed geodesics $e_1$, $e_2$ and $e_3$ obtained by performing surgery to $e$ at $p$ are multiples of simple curves in $S'$. This then implies that $e$ must have a unique self-intersection point. In particular, $e_1$, $e_2$ and $e_3$ are simple and are pairwise non-intersecting. The homotopy from $e_1$, $e_2$ and $e_3$ to $e$ is a pair of pants $P$ that contains $e$, and has $e_1$, $e_2$ and $e_3$ as its boundary components.
\end{proof}

\begin{prop}\label{asymptotic lengths}
Fix an auxiliary hyperbolic structure on $S$. Let $\{\nu_j\}_{j=1}^\infty\subset \Cmc(S)$ be a sequence of non-zero geodesic currents. There is 
\begin{itemize}
	\item a subsequence of $\{\nu_j\}_{j=1}^\infty$, also denoted $\{\nu_j\}_{j=1}^\infty$, 
	\item a (possibly disconnected, possibly empty) essential subsurface $S'\subset S$,
	\item a (possibly empty) collection of pairwise non-intersecting, simple closed geodesics $\{c_1,\dots, c_k\}\subset\Cmc\Gmc(S\setminus S')$
\end{itemize} 
so that $A:=S'\cup \bigcup_{i=1}^kc_i\subset S$ is non-empty, and the following property holds. Let $c\in\Cmc\Gmc(S)$ be a closed curve so that $c\notin\Cmc\Gmc(S\setminus A)$ and $c$ is not a multiple of $c_i$ for $i=1,\dots,k$.
\begin{enumerate}
	\item If $d$ is a multiple of $c_i$ for some $i=1,\dots,k$ or $d\in\Cmc\Gmc(S\setminus A)$, then $\displaystyle\lim_{j\to\infty}\frac{i(\nu_j,d)}{i(\nu_j,c)}=0$.
	\item If $d\in\Cmc\Gmc(S)$ is a closed curve so that $d\notin\Cmc\Gmc(S\setminus A)$ and $d$ is not a multiple of $c_i$ for $i=1,\dots,k$, then $\displaystyle \lim_{j\to\infty}\frac{i(\nu_j,d)}{i(\nu_j,c)}\in\Rbbb^+$.
\end{enumerate}
\end{prop}

\begin{proof} Since the weak$^*$ topology on $\Cmc(S)$ is metrizable and $\mathcal{PC}(S)$ is compact, there exist
\begin{itemize}
\item a subsequence of $\{\nu_j\}_{j=1}^\infty$, also denoted $\{\nu_j\}_{j=1}^\infty$,
\item a sequence of positive real numbers $\{\lambda_j\}_{j=1}^\infty$,
\item a non-zero geodesic current $\nu$,
\end{itemize}
such that $\lim_{j\to\infty}\lambda_j\nu_j=\nu$. Define
\[
\text{supp }\nu:=\{g\in \Gmc(\Std)\colon \nu(U_g)>0 \text{ for all neighborhoods $U_g$ of $g$}\},
\]
and consider $B:=(\text{supp }\nu)/\Gamma\subset \Gmc(S)$. 

For our choice of an auxiliary hyperbolic metric on $S$, let $\{c'_1,c'_2,\dots, c'_l\}$ be a maximal (possibly empty) collection of pairwise non-intersecting simple closed geodesics that do not have transverse intersections with any geodesic in $B$. Then let $\{S_1,\dots, S_m\}$ be the list of connected components of $S\setminus \{c'_1,c'_2,\dots, c'_l\}$ and define
\[
S':=\bigcup_{\{t\in\{1,\dots,m\}:S_t\cap B\neq\emptyset\}} S_t.
\]
If $B\cap\{c'_1,c'_2,\dots,c'_l\}=\emptyset$, set $k$ to be $0$. Otherwise, let $c_1,\dots,c_k$ be the closed geodesics in $B\cap\{c'_1,\dots,c'_l\}$. 

Notice that $A$ is non-empty because $\text{supp }\nu$ is non-empty. Since the intersection pairing is continuous, for any $a\in\Cmc\Gmc(S)$ and for any $b$ intersecting transversely a geodesic in $B$, up to passing to a further subsequence, we have 
\[\lim_{j\to\infty}\frac{i(\nu_j,a)}{i(\nu_j,b)}=\lim_{j\to\infty}\frac{i(\lambda_j\nu_j,a)}{i(\lambda_j\nu_j,b)}=\frac{i(\nu,a)}{i(\nu,b)}.\]
Also, $i(\nu,a)>0$ if and only if $a$ intersects some geodesic in $B$ transversely. It is thus sufficient to show that for any $a\in\Cmc\Gmc(S)$, $a$ intersects a geodesic in $B$ transversely if and only if $a$ is not a multiple of $c_i$ for $i=1,\dots,k$ and $a\notin\Cmc\Gmc(S\setminus A)$. 

Clearly, if $a$ is a multiple of $c_i$ for some $i=1,\dots,k$ or $a\in\Cmc\Gmc(S\setminus A)$, then $a$ does not intersect any geodesic in $\text{supp }\nu$ transversely. To prove the converse, suppose that $a$ is not a multiple of $c_i$ for $i=1,\dots,k$, and $a\notin\Cmc\Gmc(S\setminus A)$. If $a$ intersects $c_i$ transversely for some $i=1,\dots,k$, we are done. Hence, for the rest of the proof, we will assume that $a$ intersects the interior of $S'$. The proof proceeds in two cases.

\textbf{Case 1:} Suppose $a\in\Cmc\Gmc(S')$. By the way $S'$ is constructed, if $a$ is simple, then it must intersect a geodesic in $B$ (otherwise the maximality of $\{c_1',\dots,c_l'\}$ is contradicted). Hence, we may assume that $a$ is non-simple. By Lemma \ref{panted systole}, we may also assume that $a$ is contained in a geodesic pair of pants $P\subset S'$. Since $a$ is a non-peripheral geodesic in $P$, it has transverse intersections with every non-peripheral geodesic in $\Gmc(P)$ and every geodesic segment in $P$ with endpoints in $\partial P$. Also, because $P\subset S'$, there is some geodesic in $B$ that intersects the interior of $P$. Hence, $a$ intersects some geodesic in $B$ transversely.

\textbf{Case 2:} Suppose $a$ is not entirely contained in $S'$. Let $S_0'$ be a connected component of $S'$ so that $a$ intersects the interior of $S_0'$, and let $B_0'$ be the set of geodesics in $B$ that lie in $S_0'$. Let $x_1$ and  $x_2$ be a pair of points where $a$ intersects the boundary of $S'_0$ in $S$, so that there is a subsegment $e$ of $a$ with endpoints $x_1$ and $x_2$ that is entirely contained in $S'_0$. Let $b_1$ and $b_2$ be the boundary components of $\overline{S'_0}$ containing $x_1$ and $x_2$, respectively. 

For $s=1,2$, choose a parameterization $f_s:[0,1]\to b_s$ so that $f_s(0)=f_s(1)=x_s$ and choose a parameterization $g:[0,1]\to e$ so that $g(0)=x_1$ and $g(1)=x_2$. Consider, the closed curve $\gamma=g^{-1}*f_2^2*g*f_1^2$, where $*$ is the symbol for concatenation. Observe that $\gamma$ is freely homotopic to a non-peripheral geodesic $d$ in $S'_0$. By the previous case, we know that $d$ intersects a geodesic in $B_0'$, so $\gamma$ also intersects a geodesic in $B_0'$. Since $b_1$ and $b_2$ are boundary geodesics, they do not intersect any geodesics in $B_0'$. This means that $e$ intersects a geodesic in $B_0'$. Moreover, since $e$ is a geodesic segment, this intersection is transverse. Hence, $a$ intersects a geodesic in $B$ transversely in this case as well.
\end{proof}

As a consequence of Theorem \ref{positively ratioed to geodesic currents} and Proposition \ref{asymptotic lengths}, we have Corollary \ref{Burger-Pozzetti}.

\subsection{Systoles and minimal pants decompositions} \label{pants decomp}
We will now explore the consequences of Proposition \ref{surgery} on systole lengths of any essential subsurface $S'\subset S$. If $\nu\in\Cmc(S)$ is period minimizing, then the function $\Cmc\Gmc(S)\to\Rbbb$ given by $c\mapsto i(c,\nu)$ is minimized at some $c\in\Cmc\Gmc(S)$. The same idea gives us a notion of systoles for essential subsurfaces, which we will now define.

\begin{definition}
Let $S'\subset S$ be an essential subsurface, and let $\nu\in\Cmc(S)$ be period minimizing. The \emph{$\nu$-systole length} of $S'$ is
\[L_\nu(S'):=\min\{i(c,\nu):c\in\Cmc\Gmc(S')\},\]
and a \emph{$\nu$-systole} of $S'$ is a closed geodesic $c\in\Cmc\Gmc(S')$ so that $i(c,\nu)=L_\nu(S')$.
Also, define the \emph{$\nu$-interior systole length }of $S'$ to be 
\[L_\nu^{int}(S'):=\min\{i(c,\nu):c\in\Cmc\Gmc(S')\text{ is non-peripheral}\},\]
and a \emph{$\nu$-interior systole} of $S'$ is a non-peripheral closed geodesic $c\in\Cmc\Gmc(S')$ so that $i(c,\nu)=L_\nu^{int}(S')$. In the case when $S=S'$, we will denote $L_\nu:=L_\nu(S)=L_\nu^{int}(S)$.
\end{definition}

Using Proposition \ref{surgery}, we can prove the following corollary.

\begin{cor}\label{shortest is simple}
Let $S'\subset S$ be a connected essential subsurface and let $\nu\in\Cmc(S)$ be period minimizing. Suppose that $S'$ is not a pair of pants. Then the following hold. 
\begin{enumerate}
\item There is a $\nu$-interior systole of $S'$ that is simple. 
\item If $\nu$ has full support, then every $\nu$-interior systole of $S'$ is simple.
\end{enumerate}
\end{cor}

\begin{proof}
Let $c$ be a $\nu$-interior systole of $S'$. We may assume without loss of generality that $c$ is primitive. Suppose that $c=[[\gamma]]$ has $k\geq1$ self-intersections. Then we can perform surgery to $c$ at some point of self-intersection to obtain $c_1=[[\gamma_1]]$, $c_2=[[\gamma_2]]$ and $c_3=[[\gamma_3]]$ with $\gamma=\gamma_3\gamma_2$ and $\gamma_1=\gamma_3^{-1}\gamma_2$. If $c_1$, $c_2$ and $c_3$ are all peripheral, then the relation $\gamma_1=\gamma_3^{-1}\gamma_2$ implies that $S'$ is a pair of pants, which contradicts the hypothesis of the corollary. Hence, for some $j_0=1,2,3$, $c_{j_0}$ is a non-peripheral closed geodesic whose self-intersection number is strictly less than the self-intersection number of $c$. 

Proof of (1). By Proposition \ref{surgery}, we know that $c_{j_0}$ is also a $\nu$-interior systole of $S'$, so we can iterate the above procedure with $c_{j_0}$ in place of $c$. This will eventually terminate after at most $k$ steps to give a $\nu$-interior systole that is simple.

Proof of (2). In the case when $\nu$ has full support, Proposition \ref{surgery} tells us that $i(c_{j_0},\nu)< i(c,\nu)$. This contradicts the fact that $c$ is a $\nu$-interior systole. 
\end{proof}

In particular, if we have a period minimizing $\nu\in\Cmc(S)$, we can build a \emph{$\nu$-minimal pants decomposition}, denoted $\Pmc_\nu(S')$, on any essential subsurface $S'\subset S$. Let $c_1,\dots, c_n$ be the $n$ boundary components of $S'$. If $S'$ is a disjoint union of pairs of pants, then $n$ is three times the number of components of $S'$ and $\Pmc_\nu(S')=\{c_1,\cdots,c_n\}$. Otherwise, Corollary \ref{shortest is simple} implies that there is a $\nu$-interior systole of $S'$ that is simple. Let $c_{n+1}$ be such a $\nu$-interior systole of $S'$, then $S'\setminus c_{n+1}$ is again an essential subsurface of $S$. Hence, we can iterate this procedure until we have a pants decomposition $\Pmc_\nu(S')$. Denote $\Pmc_\nu(S)$ simply by $\Pmc_\nu$.

\section{Combinatorial description of $\Cmc\Gmc(S')$}\label{combinatorial}

\textbf{In this section, fix some $\nu\in\Cmc(S)$ that is period minimizing.} An important ingredient in the proof of Theorem \ref{main theorem} is a finite combinatorial description, defined below, for each conjugacy class in $\Gamma'$ that is adapted to $\nu$. The methods in this section and the following one are inspired by work of the second author \cite{Zha2}.

\subsection{Minimal pants decompositions and related structures}\label{Minimal pants decompositions and related structures}
First, we need to equip $S$ with an ideal triangulation which depends on $S'$ and $\nu$.
 
\begin{definition}
An \emph{ideal triangulation} of $\Std$ is a maximal $\Gamma$-invariant subset $\widetilde{\Tmc}\subset\Gmc(\Std)$ such that the following hold:
\begin{enumerate}
\item Any two pairs of geodesics $\{x, y\}, \{z, w\}\in\widetilde{\Tmc}$ do not intersect transversely.
\item For any geodesic $\{x,y\}\in\widetilde{\Tmc}$, one of the following must hold:
\begin{itemize}
\item There is some $z$ in $\partial\Gamma$ such that $\{x,z\},\{y,z\}\in\widetilde{\Tmc}$.
\item There is some $\gamma\in\Gamma$ such that $\{x,y\}$ is the set of fixed points of $\gamma$.
\end{itemize}
\end{enumerate}
An \emph{ideal triangulation} of $S$ is then the quotient $\Tmc:=\widetilde{\Tmc}/\Gamma$ for some ideal triangulation of $\widetilde{\Tmc}$ of $\Std$. A \emph{triangle} is an unordered triple of geodesics in $\widetilde{\Tmc}$ of the form $\big\{\{x,y\}, \{y,z\}, \{z,x\}\big\}$.
\end{definition}

If we choose a hyperbolic structure on $S$, then every ideal triangulation $\widetilde{\Tmc}$ of $\Std$ can be realized as an ideal triangulation of $\Std=\Dbbb$ (in the classical sense) by assigning to each pair $\{x,y\}\in\widetilde{\Tmc}$ the unique hyperbolic geodesic in $\Dbbb$ with endpoints $x,y\in\partial\Dbbb$. Moreover, this ideal triangulation is $\Gamma$-invariant, so $\Tmc$ can be thought of as an ideal triangulation (in the classical sense) of $S$. 

For our purposes, we will consider a particular ideal triangulation $\Tmc_{\nu,S'}$ of $S$, defined as follows. \textbf{Choose an orientation on $S$.} Recall that we previously constructed a $\nu$-minimal pants decomposition $\Pmc_\nu(S')$ of $S'$ as a consequence of Corollary \ref{shortest is simple}. Extend this to a pants decomposition $\Pmc_{\nu,S'}$ of $S$, and let $P_1,\dots,P_{2g-2}$ be the pairs of pants given by $\Pmc_{\nu,S'}$ where $g$ is the genus of $S$.  For each $j=1,\dots,2g-2$, orient each component of $\partial P_j$ so that $P_j$ lies on the left of the boundary component. Let $\gamma_{1,j}, \gamma_{2,j},\gamma_{3,j}\in\Gamma$ be primitive group elements corresponding to the three boundary components of $P_j$ equipped with their orientations, so that $\gamma_{3,j}\gamma_{2,j}\gamma_{1,j}=\id$. For each $i=1,2,3$ and $j=1,\dots,2g-2$, let $\gamma_{i,j}^+,\gamma_{i,j}^-\in\partial\Gamma$ denote the attracting and repelling fixed points of $\gamma_{i,j}$ respectively. 

Let $\widetilde{\Qmc}_j$ and $\widetilde{\Pmc}_j$ be the subsets of $\Gmc(\Std)$ defined by
\[\widetilde{\Qmc}_j:=\bigcup_{\gamma\in\Gamma}\big\{\gamma\cdot\{\gamma_{1,j}^-,\gamma_{2,j}^-\},\gamma\cdot\{\gamma_{2,j}^-,\gamma_{3,j}^-\},\gamma\cdot\{\gamma_{3,j}^-,\gamma_{1,j}^-\}\big\},\]
\[\widetilde{\Pmc}_j:=\bigcup_{\gamma\in\Gamma}\big\{\gamma\cdot\{\gamma_{1,j}^-,\gamma_{1,j}^+\},\gamma\cdot\{\gamma_{2,j}^-,\gamma_{2,j}^+\},\gamma\cdot\{\gamma_{3,j}^-,\gamma_{3,j}^+\}\big\},\]
and note that both $\widetilde{\Qmc}_j$ and $\widetilde{\Pmc}_j$ do not depend on the choice of $\gamma_{1,j}$, $\gamma_{2,j}$ and $\gamma_{3,j}$. They are also $\Gamma$-invariant, so we can define $\Qmc_j:=\widetilde{\Qmc}_j/\Gamma$ and $\Pmc_j:=\widetilde{\Pmc}_j/\Gamma$. With this, define 
\[\widetilde{\Qmc}:=\bigcup_{l=1}^{2g-2}\widetilde{\Qmc}_j,\,\,\,\,\,\widetilde{\Pmc}:=\bigcup_{l=1}^{2g-2}\widetilde{\Pmc}_j,\,\,\,\,\,\Pmc:=\bigcup_{l=1}^{2g-2}\Pmc_j,\,\,\,\,\,\Qmc:=\bigcup_{l=1}^{2g-2}\Qmc_j,\] 
and observe that $\widetilde{\Tmc}_{\nu,S'}:=\widetilde{\Qmc}\cup\widetilde{\Pmc}$ and $\Tmc_{\nu,S'}:=\Qmc\cup\Pmc$ are ideal triangulations of $\Std$ and $S$ respectively.  

\begin{figure}
\includegraphics[scale=0.5]{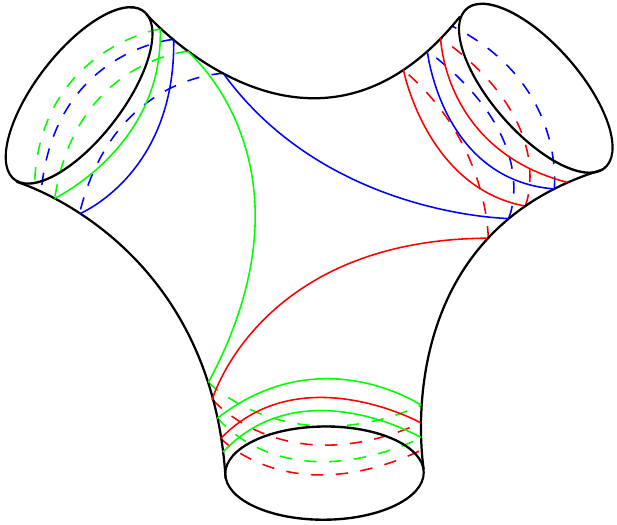}
\caption{Curves in $\Qmc_j$}
\label{figure3}
\end{figure}

It is clear that $\Pmc=F\big(\Pmc_{\nu,S'}\big)$ (recall that $F\colon \Cmc\Gmc(S)\to\Gmc(S)$ sends $[[\gamma]]$ to $[\gamma^-,\gamma^+]$). Also, if we choose a hyperbolic structure on $S$, then for all $j=1,\dots,2g-2$, the three geodesics in $\Qmc_j$ correspond to three simple, pairwise non-intersecting geodesics in the hyperbolic pair of pants $P_j$ that each ``spiral" towards two different boundary components of $P_j$ (see Figure \ref{figure3}).

The ideal triangulation by itself is insufficient to give a finite combinatorial description for the geodesics in $\Cmc\Gmc(S')$. We need to make some additional choices, which we will now specify. 

\textbf{Choose
\begin{itemize}
\item an orientation on each simple closed geodesic in $\Pmc_{\nu,S'}$.
\item a hyperbolic structure $\Sigma$ on $S$.
\end{itemize}}

Since we have chosen orientations on every $c\in\Pmc_{\nu,S'}$, $c$ can be viewed as a conjugacy class in $[\Gamma]$. For any such $c$, let $\gamma_c\in\Gamma$ be a primitive group element so that $[\gamma_c]=c\in[\Gamma]$. Then let 
\[V(\gamma_c^\pm):=\big\{x\in\partial\Gamma\setminus\{\gamma_c^-,\gamma_c^+\}:\{x,\gamma_c^\pm\}\in\widetilde{\Tmc}_{\nu,S'}\big\}\]
and define 
\[\widetilde{\Nmc}(\gamma_c^\pm):=\big\{\{x,y\}\in\widetilde{\Tmc}_{\nu,S'}:x,y\in V(\gamma_c^\pm)\big\}.\]

Observe that $V(\gamma_c^\pm)$ and $\widetilde{\Nmc}(\gamma_c^\pm)$ are both invariant under the cyclic subgroup $\langle\gamma_c\rangle\subset\Gamma$. Also, the geodesics in $\widetilde{\Nmc}(\gamma_c^-)\cup\widetilde{\Nmc}(\gamma_c^+)$ are realized as hyperbolic geodesics in $\Std\simeq\Dbbb$, and their union bounds a simply connected, convex domain $\Omega_{\gamma_c}\subset\Std$ that contains the axis of $\gamma_c$. Let $P_1$ and $P_2$ be the two pairs of pants given by $\Pmc_{\nu,S'}$ that have $c$ as a common boundary component, so that $P_1$ and $P_2$ lie on the left and right of $c$ respectively. (It is possible that $P_1=P_2$).

\textbf{Choose a point $r^\pm$ on a hyperbolic geodesic in $\widetilde{\Nmc}(\gamma_c^\pm)$, and let $p_{\gamma_c}^\pm\in\langle\gamma_c\rangle\cdot r^\pm$ be a point so that 
\[\nu\big(G[p_{\gamma_c}^+,p_{\gamma_c}^-]\big)=\min\big\{\nu\big(G[p^+,p^-]\big):p^\pm\in\langle\gamma_c\rangle\cdot r^\pm\big\}.\]}
Observe that this minimum exists because 
\[\lim_{n-m\to\pm\infty}\nu\big(G[\gamma_c^n\cdot r^-,\gamma_c^m\cdot r^+]\big)=\infty.\]
Also, let $x_{\gamma_c}^\pm,y_{\gamma_c}^\pm\in\partial\Gamma$ be the points so that $\{x_{\gamma_c}^\pm,y_{\gamma_c}^\pm\}\in\Gmc(\Std)$ correspond to the hyperbolic geodesics in $\widetilde{\Nmc}(\gamma_c^\pm)$ that contain $p_{\gamma_c}^\pm$ (see Figure \ref{figure4}). 

\begin{figure}
\includegraphics[scale=0.8]{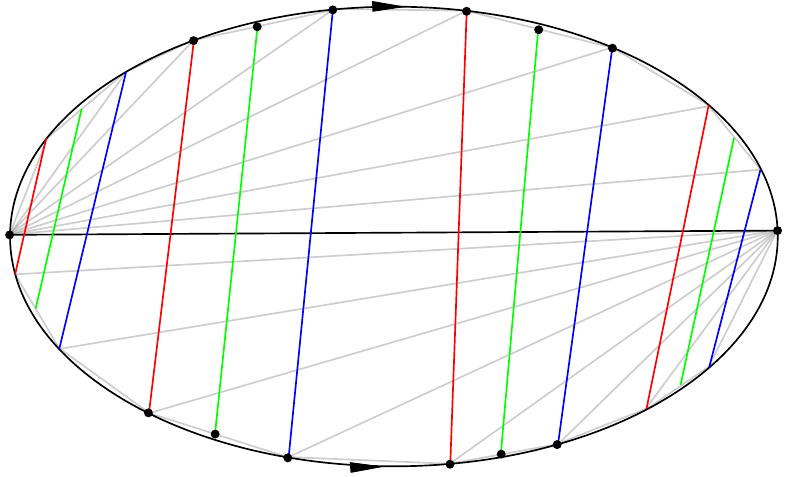}
\put (-237, 94){\makebox[0.7\textwidth][r]{\footnotesize$\gamma_c^-$ }}
\put (-549, 93){\makebox[0.7\textwidth][r]{\footnotesize$\gamma_c^+$ }}
\put (-453, 164){\makebox[0.7\textwidth][r]{\footnotesize$p_{\gamma_c}^+$ }}
\put (-453, 18){\makebox[0.7\textwidth][r]{\footnotesize$p_{\gamma_c}^-$ }}
\put (-472, 173){\makebox[0.7\textwidth][r]{\footnotesize$x_{\gamma_c}^+$ }}
\put (-490, 17){\makebox[0.7\textwidth][r]{\footnotesize$x_{\gamma_c}^-$ }}
\put (-415, 185){\makebox[0.7\textwidth][r]{\footnotesize$y_{\gamma_c}^+$ }}
\put (-433, -1){\makebox[0.7\textwidth][r]{\footnotesize$y_{\gamma_c}^-$ }}
\put (-340, 163){\makebox[0.7\textwidth][r]{\footnotesize$\gamma_c\cdot p_{\gamma_c}^+$ }}
\put (-345, 14){\makebox[0.7\textwidth][r]{\footnotesize$\gamma_c\cdot  p_{\gamma_c}^-$ }}
\put (-360, 184){\makebox[0.7\textwidth][r]{\footnotesize$\gamma_c\cdot  x_{\gamma_c}^+$ }}
\put (-368, -2){\makebox[0.7\textwidth][r]{\footnotesize$\gamma_c\cdot  x_{\gamma_c}^-$ }}
\put (-295, 170){\makebox[0.7\textwidth][r]{\footnotesize$\gamma_c\cdot  y_{\gamma_c}^+$ }}
\put (-310, 6){\makebox[0.7\textwidth][r]{\footnotesize$\gamma_c\cdot  y_{\gamma_c}^-$ }}
\caption{$\{x_{\gamma_c}^-,x_{\gamma_c}^+\}$, $\{y_{\gamma_c}^-,y_{\gamma_c}^+\}$ and $\{p_{\gamma_c}^-,p_{\gamma_c}^+\}$}
\label{figure4}
\end{figure}
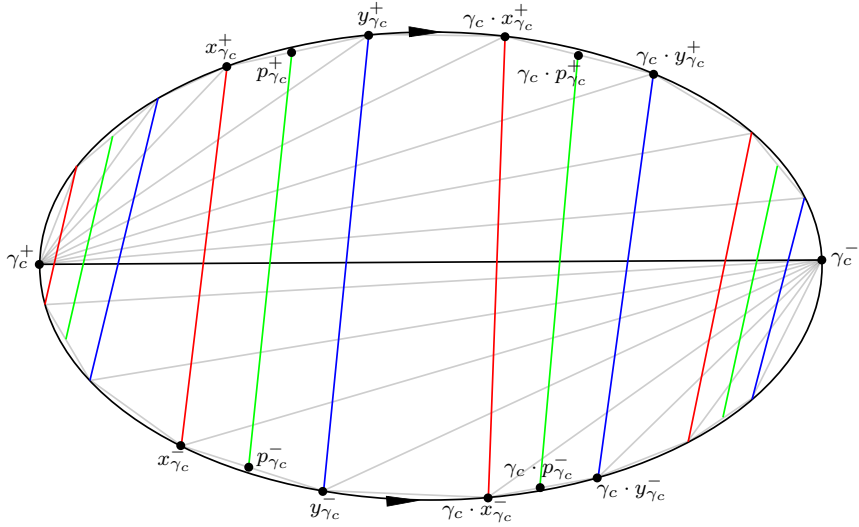

By reversing the labeling of $x_{\gamma_c}^+$ and $y_{\gamma_c}^+$ if necessary, we can assume without loss of generality that the hyperbolic geodesics corresponding to $\{x_{\gamma_c}^+,x_{\gamma_c}^-\}$ and $\{y_{\gamma_c}^+,y_{\gamma_c}^-\}$ do not intersect. Then define
\[\widetilde{\Rmc}_1(\gamma_c):=\bigcup_{k\in\Zbbb}\big\{{\gamma_c}^k\cdot\{x_{\gamma_c}^+,x_{\gamma_c}^-\},\gamma_c^k\cdot\{y_{\gamma_c}^+,y_{\gamma_c}^-\}\big\}\subset\Gmc(S),\,\,\,\,\,\,\,\,\widetilde{\Rmc}_1(c):=\bigcup_{\eta\in\Gamma}\eta\cdot\widetilde{\Rmc}_1(\gamma_c),\]
and
\[\widetilde{\Rmc}_2(\gamma_c):=\big\{\gamma_c^k\cdot[p_{\gamma_c}^+,p_{\gamma_c}^-]:k\in\Zbbb\big\},\,\,\,\,\,\,\,\,\widetilde{\Rmc}_2(c):=\bigcup_{\eta\in\Gamma}\eta\cdot\widetilde{\Rmc}_2(\gamma_c).\]

Note that $\gamma_c$ induces orderings on $\widetilde{\Rmc}_1(\gamma_c)$ and $\widetilde{\Rmc}_2(\gamma_c)$. Also, for $i=1,2$, $\widetilde{\Rmc}_i(\gamma_c)/\langle\gamma_c\rangle=\widetilde{\Rmc}_i(c)/\Gamma$, which consists of two geodesics in $\Gmc(S)$ when $i=1$ and one geodesic in $\Gmc(S)$ when $i=2$.

\subsection{Binodal edges and winding}

Let $[\gamma]\in[\Gamma']$ be the conjugacy class of any non-identity element. We can now define (given all the choices made above) a finite combinatorial description for each conjugacy class $[\gamma]\in[\Gamma']$, which is adapted to $\nu$.

Recall that we have already chosen a hyperbolic structure on $S$.

\begin{definition}\label{binodal definition}
Let $I\subset\widetilde{S}$ be either a geodesic or geodesic subsegment. Also, for any $\gamma\in\Gamma\setminus\{\id\}$, let $L_\gamma\subset\Dbbb$ be the axis of $\gamma$.
\begin{itemize}
\item Let $\widetilde{\Amc}(I)$ be the set of geodesics in $\widetilde{\Qmc}$ that intersect $I$ transversely. A point in $\partial\Gamma$ is a \emph{node} of $I$ if it is the common endpoint of two distinct geodesics in $\widetilde{\Amc}(I)$. We call a geodesic in $\widetilde{\Amc}(I)$ \emph{binodal} if both of its endpoints in $\partial\Dbbb$ are nodes. Denote the set of binodal edges in $\widetilde{\Amc}(I)$ by $\widetilde{\Bmc}(I)$.
\item In the case when $I=L_\gamma$, observe that $\widetilde{\Amc}(\gamma):=\widetilde{\Amc}(L_\gamma)$ and $\widetilde{\Bmc}(\gamma):=\widetilde{\Bmc}(L_\gamma)$ are both $\langle \gamma\rangle$-invariant, so we can define $\Amc[\gamma]:=\widetilde{\Amc}(\gamma)/\langle \gamma\rangle$ and $\Bmc[\gamma]:=\widetilde{\Bmc}(\gamma)/\langle \gamma\rangle$.
\end{itemize}
\end{definition}

Observe that we can think of $\Amc[\gamma]$ and $\Bmc[\gamma]$ as cyclic sequences of geodesics in $S$. In that case, they depend only on the conjugacy class of $\gamma$, and not on $\gamma$ itself. Also, $\Bmc[\gamma]$ is finite, and is empty if and only if $L_\gamma\in\widetilde{\Pmc}$. \textbf{For the rest of this section, we will assume that $\Bmc[\gamma]$ is non-empty unless stated otherwise.} 

The orientation on $L_\gamma$ induces a natural ordering $\prec$ on $\widetilde{\Amc}(\gamma)$. Also, since $\widetilde{\Amc}(\gamma)$ does not contain any of its accumulation points, we can define a bijective successor map $\suc:\widetilde{\Amc}(\gamma)\to\widetilde{\Amc}(\gamma)$. Moreover, the ordering $\prec$ induces a cyclic order on $\Amc[\gamma]$, and the successor map $\suc:\widetilde{\Amc}(\gamma)\to\widetilde{\Amc}(\gamma)$ descends to a successor map $\suc:\Amc[\gamma]\to\Amc[\gamma]$.

The orientation on $S$ induces an orientation on $\partial\Dbbb=\partial\Gamma$. Let $s_0(\gamma)$ and $s_1(\gamma)$ be the two connected components of $\partial\Gamma\setminus\{\gamma^-,\gamma^+\}$, oriented from $\gamma^-$ to $\gamma^+$, so that the orientation on $s_0(\gamma)$ agrees with the orientation on $\partial\Gamma$.

\begin{figure}
\includegraphics[scale=0.8]{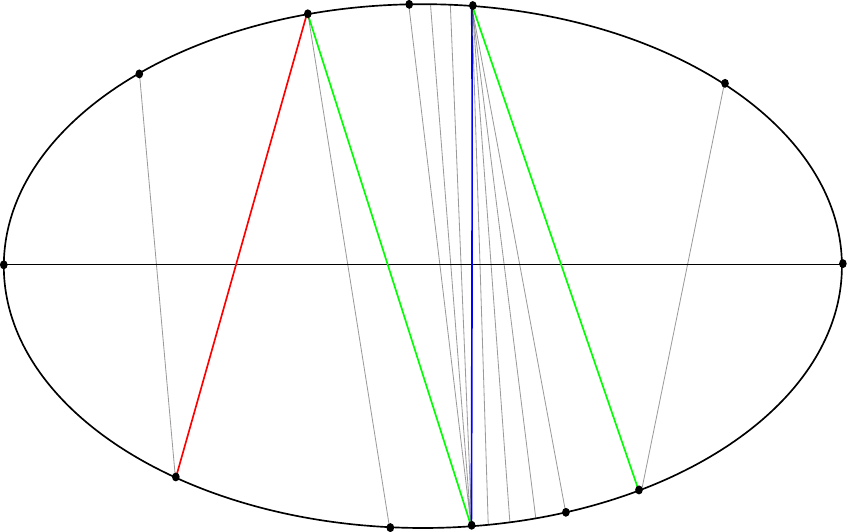}
\put (-388, -7){\makebox[0.7\textwidth][r]{\footnotesize$\eta^-$ }}
\put (-388, 207){\makebox[0.7\textwidth][r]{\footnotesize$\eta^+$ }}
\put (-575, 100){\makebox[0.7\textwidth][r]{\footnotesize$\gamma^-$ }}
\put (-236, 101){\makebox[0.7\textwidth][r]{\footnotesize$\gamma^+$ }}
\put (-413, 60){\makebox[0.7\textwidth][r]{\footnotesize$e_{i+1}$ }}
\put (-433, 50){\makebox[0.7\textwidth][r]{\footnotesize$\suc(e_i)$ }}
\put (-341, 49){\makebox[0.7\textwidth][r]{\footnotesize$e_{i+2}$ }}
\put (-304, 130){\makebox[0.7\textwidth][r]{\footnotesize$\suc(e_{i+2})$ }}
\put (-475, 140){\makebox[0.7\textwidth][r]{\footnotesize$e_i$ }}
\put (-517, 123){\makebox[0.7\textwidth][r]{\footnotesize$\suc^{-1}(e_i)$ }}
\put (-253,155){\makebox[0.7\textwidth][r]{\footnotesize$s_0(\gamma)$ }}
\put (-253, 45){\makebox[0.7\textwidth][r]{\footnotesize$s_1(\gamma)$ }}
\caption{$e_i$ is of S-type. $e_{i+1}$ and $e_{i+2}$ are of Z-type. Notice $\suc(e_i)=\suc^{-1}(e_{i+1})$.}
\label{figure6}
\end{figure}
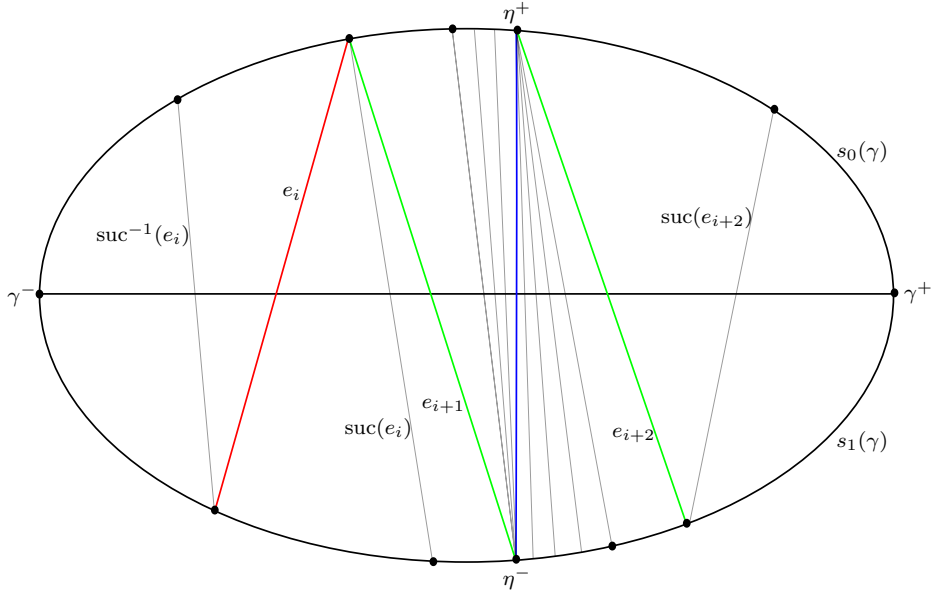

\begin{definition}
(See Figure \ref{figure6}.) Let $\{x,y\}$ be an edge in $\widetilde{\Bmc}(\gamma)$ and assume without loss of generality that $x$ lies in $s_0(\gamma)$ and $y$ lies in $s_1(\gamma)$. We say $\{x,y\}$ is 
\begin{itemize}
\item \emph{Z-type} if $\big(\suc\{x,y\}\big)\cap\{x,y\}=\{y\}$ and $\big(\suc^{-1}\{x,y\}\big)\cap\{x,y\}=\{x\}$, 
\item \emph{S-type} if $\big(\suc\{x,y\}\big)\cap\{x,y\}=\{x\}$ and $\big(\suc^{-1}\{x,y\}\big)\cap\{x,y\}=\{y\}$.
\end{itemize}
Let $\widetilde{\Zmc}(\gamma)$ be the edges in $\widetilde{\Bmc}(\gamma)$ that are Z-type and $\widetilde{\Smc}(\gamma)$ be the edges in $\widetilde{\Bmc}(\gamma)$ that are S-type. Since $\widetilde{\Zmc}(\gamma)$ and $\widetilde{\Smc}(\gamma)$ are $\langle \gamma\rangle$-invariant, we can define $\Zmc[\gamma]:=\widetilde{\Zmc}(\gamma)/\langle \gamma\rangle$ and $\Smc[\gamma]:=\widetilde{\Smc}(\gamma)/\langle \gamma\rangle$.
\end{definition}

Again, $\Smc[\gamma]$ and $\Zmc[\gamma]$, when viewed as a sequence of geodesics in $S$, depend only on the conjugacy class of $\gamma$. Also, note that $\Zmc[\gamma]\cup\Smc[\gamma]=\Bmc[\gamma]$, and the cyclic order on $\Amc[\gamma]$ induces cyclic orders on $\Zmc[\gamma]$, $\Smc[\gamma]$ and $\Bmc[\gamma]$. Let $e$ and $e'$ be consecutive geodesics in $\Bmc[\gamma]$ so that $e$ precedes $e'$. Then the following must hold:
\begin{enumerate}
\item If $e$ and $e'$ are not of the same type, then there are representatives $\etd, \etd'\in\widetilde{\Bmc}(\gamma)$ of $e$, $e'$ respectively so that $\etd$ and $\etd'$ share a common endpoint in $\partial\Dbbb$, and $\etd\prec\etd'$.
\item If $e$ and $e'$ are of the same type, then there are representatives $\etd, \etd'\in\widetilde{\Bmc}(\gamma)$ of $e$, $e'$ respectively so that there is a geodesic in $\widetilde{\Pmc}$ that has a common endpoint with each of $\etd$ and $\etd'$, and $\etd\prec\etd'$.
\end{enumerate}

If (1) holds, let $\gamma(\etd,\etd')\in\Gamma$ be the primitive group element that has the common vertex of $\etd$ and $\etd'$ as a fixed point, and so that the conjugacy class $[\gamma(\etd,\etd')]$ corresponds to an oriented closed geodesic in $\Pmc_{\nu,S'}$. On the other hand, if (2) holds, let $\gamma(\etd,\etd')$ be the element in $\Gamma$ whose axis is the geodesic in $\widetilde{\Pmc}$ that has common endpoints with $\etd$ and $\etd'$, and so that the conjugacy class $[\gamma(\etd,\etd')]$ corresponds to an oriented closed geodesic in $\Pmc_{\nu,S'}$. If $\gamma\in\Gamma'$, the closed geodesic in $\Pmc_{\nu,S'}$ corresponding to $\gamma(\etd,\etd')$ is in $\Pmc_\nu(S')$.

\begin{notation}\label{t notation}
For $i=1,2$, let $t_i(e,e')=t_{i,\gamma}(e,e')$ be the signed number of edges in $\widetilde{\Rmc}_i\big(\gamma(\etd,\etd')\big)$ that intersect $L_\gamma$. Here, the sign is positive if the orderings on these edges induced by $\gamma(\etd,\etd')$ and by $\gamma$ agree, and is negative otherwise.
\end{notation}

The quantities $t_i(e,e')$ for $i=1,2$ do not depend on the choice of $\etd$ and $\etd'$. Also, they do not depend on the choice of $\gamma\in[\gamma]$ in the following sense: if $\bar{\gamma}=\eta\gamma\eta^{-1}$ for some $\eta\in\Gamma'$, then $\eta\cdot e$ and $\eta\cdot e'$ are consecutive elements in $\Bmc[\bar{\gamma}]$, and $t_{i,\gamma}(e,e')=t_{i,\bar{\gamma}}(\eta\cdot e,\eta\cdot e')$. 

\begin{notation} \label{w notation} Let $[p,q]\subset \overline{\Omega}_{\gamma_c}$ be a geodesic segment that intersects the geodesics in $\widetilde{\Rmc}_1(\gamma_c)\cup\widetilde{\Rmc}_2(\gamma_c)$ transversely. For $i=1,2$, let $w_i[p,q]$ denote the number of edges in $\widetilde{\Rmc}_i(\gamma_c)$ that intersect $[p,q]$ respectively. 
\end{notation}
It is clear that $[\suc^{-1}(\etd)\cap L_\gamma,\suc(\etd')\cap L_\gamma]\subset \overline{\Omega}_{\gamma(\etd,\etd')}$, and that $|t_i(e,e')|=w_i[\suc^{-1}(\etd)\cap L_\gamma,\suc(\etd')\cap L_\gamma]$.

Cyclically enumerate $\Bmc[\gamma]=\{e_{m+1}=e_1,e_2\dots,e_m\}$, and for each $i=1,\dots,m$, let $T_i$ be the type (Z or S) of $e_i$. Then define the cyclic sequence of tuples
\[\psi_{\Pmc_{\nu,S'}}[\gamma]=\psi[\gamma]:=\{(\suc^{-1}(e_i),e_i,\suc(e_i),T_i,t_1(e_i,e_{i+1}))\}_{i=1}^m.\]
This is the combinatorial description of $[\gamma]\in[\Gamma]$ mentioned at the start of the section.\

\begin{figure}
\includegraphics[scale=0.8]{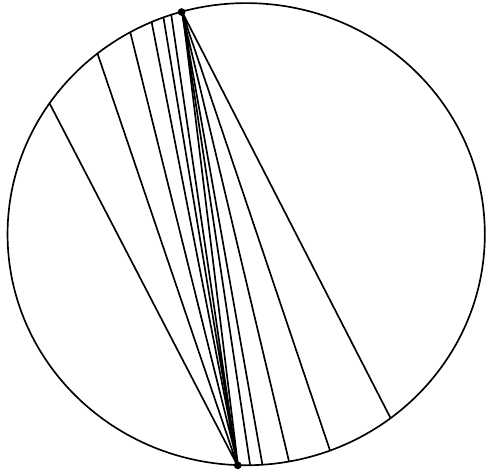}
\put (-338, -4){\makebox[0.7\textwidth][r]{\footnotesize$\eta^-$ }}
\put (-360, 182){\makebox[0.7\textwidth][r]{\footnotesize$\eta^+$ }}
\put (-400, 100){\makebox[0.7\textwidth][r]{\footnotesize$\btd_i$ }}
\put (-388, 120){\makebox[0.7\textwidth][r]{\footnotesize$\ctd_i$ }}
\put (-300, 62){\makebox[0.7\textwidth][r]{\footnotesize$\btd_i^*$ }}
\put (-315, 49){\makebox[0.7\textwidth][r]{\footnotesize$\ctd_i^*$ }}
\caption{$\btd_i,\ctd_i$, $\btd^*_i$ and $\ctd^*_i$}
\label{figure5}
\end{figure}

Let $\Psi$ be the collection of cyclic sequences of the form $\{(a_i,b_i,c_i,T_i,t_i)\}_{i=1}^m$, where $a_i,b_i,c_i$ are the three distinct edges in $\Qmc_j$ for some $j$, $T_i$ is the symbol S or Z, and $t_i\in\Zbbb$. For any term $\{a_i,b_i,c_i,T_i,t_i\}$ of the sequence $\{(a_i,b_i,c_i,T_i,t_i)\}_{i=1}^m\in\Psi$, let 
\[T_i^*:=\left\{\begin{array}{ll}
\mathrm{S}&\text{if }T_i=\mathrm{Z},\\
\mathrm{Z}&\text{if }T_i=\mathrm{S}.
\end{array}\right.\]
Also, let $\btd_i$ and $\ctd_i$ be lifts of $b_i$ and $c_i$ respectively that share a common endpoint in $\partial\Gamma$, and let $\eta\in\Gamma$ be the group element whose repelling fixed point is this common endpoint. Then there are exactly two geodesics $b_i^*$ and $c_i^*$ in $\Qmc$ with lifts $\btd_i^*$ and $\ctd_i^*$ in $\Dbbb$ respectively that have $\eta^+$ as a common endpoint. Let $a_i^*$ be the edge in $\Qmc$ so that $\{a_i^*, b_i^*,c_i^*\}=\Qmc_j$ for some $j$ (see Figure \ref{figure5}). 

\begin{definition}\label{admissible sequence}
We say a sequence $\{(a_i,b_i,c_i,T_i,t_i)\}_{i=1}^m$ in $\Psi$ is \emph{admissible} if for all $i=1,\dots,m$, $(a_{i+1},b_{i+1},c_{i+1},T_{i+1})$ is one of the following:
\[(b_i,c_i,a_i,T_i^*),(c_i,b_i,a_i,T_i^*),(b_i^*,c_i^*,a_i^*,T_i),(c_i^*,b_i^*,a_i^*,T_i).\]
(Notice that the last two cases correspond to $\gamma$ crossing the pants curve $[[\eta]]$.) Let $\Psi'$ denote the set of admissible sequences in $\Psi$.
\end{definition}

Observe that $\psi$ can be viewed as a map from $[\Gamma]$ to $\Psi'$. The most important property of $\psi$ is its injectivity, which was previously proven by the second author (Proposition 4.5 of \cite{Zha2}).

\begin{prop}\label{combinatorial prop}
Let $\gamma_0,\gamma_1$ be elements in $\Gamma'$. Then $\psi[\gamma_0]=\psi[\gamma_1]$ if and only if $[\gamma_0]=[\gamma_1]$.
\end{prop}

\begin{notation}\label{combinatorial notation}\
\begin{itemize}
	\item For any cyclic sequence $\sigma=\{(a_i,b_i,c_i,T_i,t_i)\}_{i=1}^m\in\Psi$, let $B(\sigma):= m$ and let $W_1(\sigma):=\sum_{i=1}^m |t_i|$.
	\item If $c=[[\gamma]]\in\Cmc\Gmc(S')$, let
\[p(c):=\sum_{d\in\Pmc_\nu(S')}i(c,d),\,\,\,\, b(c):=|\Bmc[\gamma]|,\]
and for $i=1,2$, let
\[w_i(c):=\sum_{j=1}^m |t_i(e_j,e_{j+1})|=\sum_{j=1}^m w_i[\suc^{-1}(e_j)\cap L_\gamma,\suc(e_{j+1})\cap L_\gamma].\]
\end{itemize}
\end{notation}

Note that $p(c)$, $b(c)$, $w_1(c)$ and $w_2(c)$ are well-defined as they do not depend on the orientation on $c$ induced by $[\gamma]$. Also, note that $b(c)=B(\psi[\gamma])$ and $w_1(c)=W_1(\psi[\gamma])$. Informally, $p(c)$ is the number of times $c$ cuts across pants curves, $b(c)$ is the number of times $c$ crosses a binodal edge in $\Qmc$, and $w_1(c)$ and $w_2(c)$ are two different ways of measuring how many times $c$ ``winds around" collar neighborhoods of the curves in $\Pmc_\nu(S')$. 

The advantage of $w_1(c)$ over $w_2(c)$ is that $w_1(c)$ can be read off the combinatorial description $\psi(\gamma)$. On the other hand, we will later obtain a lower bound for $i(c,\nu)$ in terms of $w_2(c)$. In the following lemma, we make the relationship between $w_1(c)$ and $w_2(c)$ explicit.

\begin{lem}\label{w1 and w2}
Let $\gamma\in\Gamma'$ and let $c=[[\gamma]]\in\Cmc\Gmc(S')$. Then
\[\frac{1}{2}w_1(c)-b(c)\leq w_2(c)\leq \frac{1}{2}w_1(c)+b(c).\]
\end{lem}

\begin{proof}
First, observe that for any consecutive pair $e,e'\in\Bmc[\gamma]$ with $e$ preceding $e'$, we have 
\[\frac{1}{2}|t_1(e,e')|-1\leq |t_2(e,e')|\leq \frac{1}{2}|t_1(e,e')|+1.\]
Summing the above inequality over all consecutive pairs in $\Bmc[\gamma]$ yields the required inequality.
\end{proof}

\section{Lengths and geodesic currents}\label{lengths}

In this section, we will prove some inequalities about lengths of closed geodesics which depend on their intersections with a $\nu$-minimal pants decomposition $\Pmc_\nu(S')$ and the corresponding ideal triangulation $\Tmc_{\nu,S'}$ as defined in Section \ref{combinatorial}. \textbf{For the rest of this section, fix a period minimizing geodesic current $\nu\in\Cmc\Gmc(S)$, a hyperbolic structure $\Sigma$ on $S$, and an essential subsurface $S'$ of $S$.} The goal of this section is to prove Theorem \ref{length bound theorem}. For any $c\in\Cmc\Gmc(S')$, this theorem gives a lower bound of $i(\nu,c)$ in terms of the $\nu$-panted systole length, the $\nu$-systole length, and the combinatorial description $b(c)$ and $w_2(c)$ defined in Section \ref{combinatorial}.

\subsection{Length lower bounds: intersection with pants curves}

We begin by first finding a lower bound for $i(c,\nu)$ in terms of the number of times $c\in\Cmc\Gmc(S')$ intersects $\Pmc_\nu(S')$. To do so, we define the following quantity.

\begin{definition}
Let $\Pmc_\nu(S')$ be a $\nu$-minimal pants decomposition of $S'$. Define the \emph{$\nu$-panted systole length} to be
\[K_\nu(S'):=\min\{i(c,\nu):c\in\Cmc\Gmc(S')\text{ is not a multiple of a geodesic in }\Pmc_\nu(S')\}.\]
\end{definition}

\begin{lem}\label{psl independent} The $\nu$-panted systole length does not depend on the choice of a minimal pants decomposition $\Pmc_\nu(S')$. Namely, 
\[
K_\nu(S')=\min\{i(c,\nu):c\in\Cmc\Gmc(S')\text{ is not a multiple of a geodesic in }\overline{\Pmc_\nu(S')}\}.
\]
for any minimal pants decomposition $\overline{\Pmc_\nu(S')}$.
\end{lem}
\begin{proof} 
Assume there exists a minimal pants decomposition $\overline{\Pmc_\nu(S')}$ such that 
\[
\overline{K_\nu(S')}:=\min\{i(c,\nu):c\in\Cmc\Gmc(S')\text{ is not a multiple of a geodesic in }\overline{\Pmc_\nu(S')}\}
\]
is greater or equal to $K_\nu(S')$. We claim that this implies $K_\nu(S')=\overline{K_\nu(S')}$. Let $c_0=[[\gamma_0]]\in\Cmc\Gmc(S')$ so that $c_0$ is not a multiple of a geodesic in $\Pmc_\nu(S')$, and $i(c_0,\nu)=K_\nu(S')$. The minimality of $K_\nu(S')$ implies that $c_0$ is primitive. 

If $c_0$ is not a simple closed geodesic in $\overline{\Pmc_\nu(S')}$, then we are done because 
\[\overline{K_\nu(S')}\leq i(c_0,\nu)=K_\nu(S')\leq\overline{K_\nu(S')}.\] 
On the other hand, if $c_0\in\overline{\Pmc_\nu(S')}$, then the fact that $c_0$ is simple implies that there are closed geodesics $c\in\Pmc_\nu(S')$ so that $i(c,c_0)\neq 0$. Let $c_1\in \Pmc_\nu(S')$ be such a closed geodesic so that $i(c_1,\nu)$ is minimal, and observe that $i(c_1,\nu)\leq i(c_0,\nu)$ by the definition of a $\nu$-minimal pants decomposition. Also, since $i(c_0,c_1)\neq 0$, we have $c_1\not\in\overline{\Pmc_\nu(S')}$, so $\overline{K_\nu(S')}\leq i(c_1,\nu)$. Therefore, 
\[
\overline{K_\nu(S')} \leq  i(c_1,\nu)\leq i(c_0,\nu)= K_\nu(S') \leq \overline{K_\nu(S')}.
\]
One finishes the proof by reversing the roles of $\Pmc_\nu(S')$ and $\overline{\Pmc_\nu(S')}$ if $\overline{K_\nu(S')}\leq K_\nu(S')$.
\end{proof}

With the notion of a panted systole length, we have the following lemma.

\begin{lem}\label{intersection}
Let $c$ be a simple $\nu$-interior systole of $S'$, and let $p,q\in\Std'\subset\Dbbb$ be points such that the interval $(p,q]$ intersects $\Pi^{-1}(c)$ transversely. Then
\[\nu\big(G(p,q]\big)\geq \Big(\big|[p,q]\cap\Pi^{-1}(c)\big|-2\Big)\cdot\frac{K_\nu(S')}{10},\]
where $G(p,q]\subset\Gmc(S)$ is the set of geodesics defined in Notation \ref{geodesic notation}. 
\end{lem}

\begin{proof}
First, observe that since $[p,q]\subset \Std'\subset\Dbbb$ is compact, $[p,q]\cap\Pi^{-1}(c)$ is finite. Also, if $[p,q]\cap\Pi^{-1}(c)=0,1$ or $2$, then the desired inequality clearly holds. Thus, we will assume for the rest of this proof that $\big|[p,q]\cap\Pi^{-1}(c)\big|\geq 3$. Let $p_1,p_2,\dots,p_k$ be the points in $[p,q]\cap\Pi^{-1}(c)$ in that order along $[p,q]$, where $k=\big|[p,q]\cap\Pi^{-1}(c)\big|$. For any $j=1,\dots,k$, let $\gamma_j\in\Gamma'$ denote a group element so that 
\begin{itemize}
\item $[[\gamma_j]]=c\in\Cmc\Gmc(S')$,
\item the axis $L_j$ of $\gamma_j$ contains $p_j$.
\end{itemize} 
The proof will proceed in two cases from here. 

\begin{figure}
\includegraphics[scale=0.6]{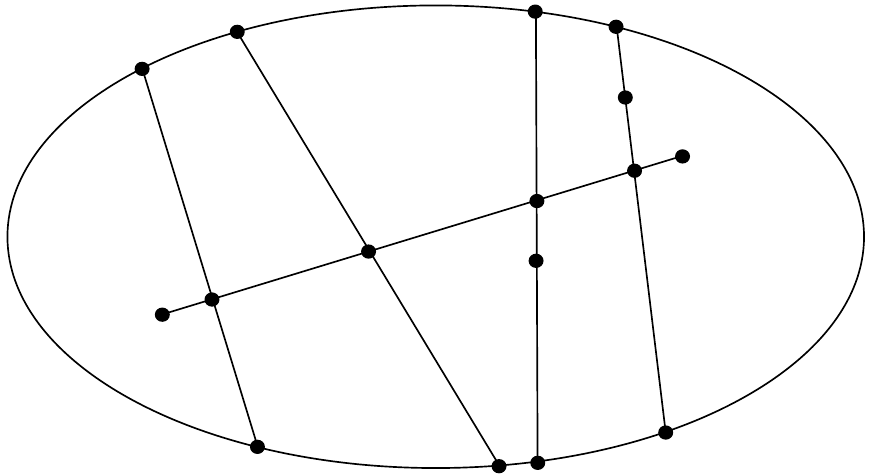}
\put (-452, 40){\makebox[0.7\textwidth][r]{\footnotesize$p$ }}
\put (-295, 90){\makebox[0.7\textwidth][r]{\footnotesize$q$ }}
\put (-427, 47){\makebox[0.7\textwidth][r]{\footnotesize$p_1$ }}
\put (-392, 58){\makebox[0.7\textwidth][r]{\footnotesize$p_2$ }}
\put (-347, 82){\makebox[0.7\textwidth][r]{\footnotesize$p_3$ }}
\put (-320, 90){\makebox[0.7\textwidth][r]{\footnotesize$p_4$ }}
\put (-346, 56){\makebox[0.7\textwidth][r]{\footnotesize$\gamma_{1,3}\cdot p_1$ }}
\put (-288, 105){\makebox[0.7\textwidth][r]{\footnotesize$\gamma_{2,4}\cdot p_2$ }}
\put (-440, 90){\makebox[0.7\textwidth][r]{\footnotesize$L_1$ }}
\put (-405, 102){\makebox[0.7\textwidth][r]{\footnotesize$L_2$ }}
\put (-347, 110){\makebox[0.7\textwidth][r]{\footnotesize$L_3$ }}
\put (-300, 45){\makebox[0.7\textwidth][r]{\footnotesize$L_4$ }}
\caption{Case 1 of proof of Lemma \ref{intersection}}
\label{figure7}
\end{figure}

\textbf{Case 1: $i(c,\nu)\geq \frac{2K_\nu(S')}{5}$.} Then for $j=1,\dots,k-2$, let $\gamma_{j,j+2}\in\Gamma'$ be a group element so that 
\begin{itemize}
\item $\gamma_{j,j+2}\cdot L_j=L_{j+2}$,
\item $\nu\big(G(\gamma_{j,j+2}\cdot p_j,p_{j+2}]\big)=\min\left\{\nu\big(G(\gamma\cdot p_j,p_{j+2}]\big):\gamma\in\Gamma,\gamma\cdot L_j=L_{j+2}\right\}$,
\end{itemize}
and let $c_{j,j+2}\in\Cmc\Gmc(S)$ be the closed geodesic such that $[[\gamma_{j,j+2}]]=c_{j,j+2}$ (see Figure \ref{figure7}). Note that $c_{j,j+2}$ is not a multiple of a curve in $\Pmc_\nu(S')$ because it has positive geometric intersection number with $c$, so $i(c_{j,j+2},\nu)\geq i(c,\nu)$. By (1) of Lemma \ref{compute length}, we have $\nu\big(G(\gamma_{j+2}\cdot r,r]\big)=i(c,\nu)$ for all $r\in L_{j+2}$. The definition of $\gamma_{j,j+2}$ implies that 
\[\nu\big(G[\gamma_{j,j+2}\cdot p_j,p_{j+2})\big)\leq\frac{1}{2}i(c,\nu)\]
for all $j=1,\dots,k-2$. Then by (2) of Lemma \ref{compute length}, we have that 
\begin{eqnarray*}
\nu\big(G(p,q]\big)&\geq&\frac{1}{2}\cdot\sum_{j=1}^{k-2}\nu\big(G(p_j,p_{j+2}]\big)\\
&\geq&\frac{1}{2}\cdot\sum_{j=1}^{k-2}\Big(\nu\big(G(p_j,\gamma_{j,j+2}\cdot p_j]\big)-\nu\big(G[\gamma_{j,j+2}\cdot p_j,p_{j+2})\big)\Big)\\
&\geq&\frac{1}{2}\cdot\sum_{j=1}^{k-2}\Big(i(c_{j,j+2},\nu)-\frac{1}{2}i(c,\nu)\Big)\\
&\geq&(k-2)\cdot\frac{i(c,\nu)}{4}\geq(k-2)\cdot\frac{K_\nu(S')}{10}.
\end{eqnarray*}

\textbf{Case 2: $i(c,\nu)< \frac{2K_\nu(S')}{5}$.} Let $\Pmc_{\nu}(S')$ be a $\nu$-minimal pants decomposition of $S'$ that contains $c$. For any $j=1,\dots,k-1$, let $\gamma_{j,j+1}:=\gamma_j\cdot \gamma_{j+1}\in\Gamma'$ and let $c_{j,j+1}\in\Cmc\Gmc(S')$ be the closed geodesic such that $[[\gamma_{j,j+1}]]=c_{j,j+1}$. Note that $c_{j,j+1}$ is not a multiple of a curve in $\Pmc_\nu(S')$ because $\gamma_j\neq\gamma_{j+1}$ and $c_{j,j+1}$ has positive geometric self-intersection number, so $i(c_{j,j+1},\nu)\geq K_\nu(S')$. Thus, by Lemma \ref{compute length},
\begin{eqnarray*}
\nu\big(G(p,q]\big)&\geq&\frac{1}{2}\cdot\Big(\nu\big(G(p_1,p_k]\big)+\nu\big(G(p_k,p_1]\big)\Big)\\
&=&\frac{1}{2}\cdot\sum_{j=1}^{k-1}\Big(\nu\big(G(\gamma_j^{-1}\cdot p_j,p_j]\big)+\nu\big(G(p_j,p_{j+1}]\big)-\nu\big(G(\gamma_j^{-1}\cdot p_j,p_j]\big)\Big)\\
&&+\frac{1}{2}\cdot\sum_{j=1}^{k-1}\Big(\nu\big(G(p_{j+1},\gamma_{j+1}\cdot p_{j+1}]\big)+\nu\big(G(\gamma_{j+1}\cdot p_{j+1},\gamma_{j+1}\cdot p_j]\big)\\
&&-\nu\big(G(p_{j+1},\gamma_{j+1}\cdot p_{j+1}]\big)\Big)\\
&\geq&\frac{1}{2}\cdot\sum_{j=1}^{k-1}\Big(\nu\big(G(\gamma_j^{-1}\cdot p_j,\gamma_{j+1}\cdot p_j]\big)-2i(c,\nu)\Big)\\
&=&\frac{1}{2}\cdot\sum_{j=1}^{k-1}\Big(\nu\big(G(p_j,\gamma_j\gamma_{j+1}\cdot p_j]\big)-2i(c,\nu)\Big)\\
&\geq&\frac{1}{2}\cdot\sum_{j=1}^{k-1}\Big(i(c_{j,j+1},\nu)-2i(c,\nu)\Big)\\
&\geq&(k-2)\cdot\frac{K_\nu(S')}{10}.\hspace{6.8cm}\qedhere
\end{eqnarray*}
\end{proof}

As a consequence of the above lemma, we obtain the following corollary.

\begin{cor}
Let $c$ be a simple $\nu$-interior systole of $S'$. For any $d\in\Cmc\Gmc(S')$,
\[i(d,\nu)\geq \big(i(d,c)-1\big)\cdot\frac{K_\nu(S')}{10}.\]
\end{cor}

\begin{proof}
If $i(d,c)=0$ or $1$, the corollary clearly holds. For the rest of this proof, we will assume that $i(d,c)\geq 2$. Choose a hyperbolic structure $\Sigma$ on $S$. Then $c$ and $d$ are realized as closed geodesics in $S'$. Choose a point $p\in c\cap d$ and a point $\ptd\in\Std$ so that $\Pi(\ptd)=p$. Let $\gamma_d\in\Gamma'$ be a group element so that $[[\gamma_d]]=d$ and $\ptd$ lies in the axis $L_d$ of $\gamma_d$. Then $\big|[\ptd,\gamma_d\cdot\ptd]\cap\Pi^{-1}(c)\big|=i(d,c)+1$. Hence, by (1) of Lemma \ref{compute length} and Lemma \ref{intersection}, we have
\[i(d,\nu)=\nu\big(G(\ptd, \gamma_d\cdot\ptd]\big)\geq \big(i(d,c)-1\big)\cdot\frac{K_\nu(S')}{10}.\qedhere\]
\end{proof}

By applying Lemma \ref{intersection} to all the curves in a $\nu$-minimal pants decomposition on $S'$, we can also obtain the following lower bound on $i(c,\nu)$ in terms of the number of times $c$ intersects the curves in a $\nu$-minimal pants decomposition $\Pmc_\nu(S')$.

\begin{lem}\label{pants curve bound}
Suppose that $S'\subset S$ is a connected essential subsurface of genus $g$ with $n$ boundary components. Let $\Pmc_\nu(S')=\{c_1,\dots,c_{3g-3+2n}\}$ so that the boundary components of $S'$ are $c_{3g-3+n+1},\dots,c_{3g-3+2n}$, and let $c\in\Cmc\Gmc(S')$. Then
\[i(c,\nu)\geq\left(\sum_{j=1}^{3g-3+n}i(c,c_j)\right)\frac{K_\nu(S')}{10\cdot 3^{3g-3+n}}.\]
\end{lem}

\begin{proof}
Assume without loss of generality that $i(\nu,c_j)\leq i(\nu,c_{j+1})$ for all $j=1,\dots,3g-3+n-1$. If 
\[\sum_{j=1}^{3g-3+n}i(c,c_j)=0\] 
(this has to happen when $S'$ is a pair of pants), the desired inequality holds, so we assume that $\sum_{j=1}^{3g-3+n}i(c,c_j)>0$ in the rest of this proof.

Let $\gamma\in\Gamma'$ so that $[[\gamma]]=c\in\Cmc\Gmc(S')$ and let $\ptd\in\Std'\subset\Dbbb$ so that 
\[\ptd\in \Pi^{-1}\left(c\cap\left(\bigcup_{j=1}^{3g-3+n}c_j\right)\right).\] 
Then let $m=3^{3g-3+n}$ and let $\ptd=\ptd_0,\ptd_1,\dots,\ptd_k=\gamma^m\cdot\ptd$ be the points in
\[[\ptd,\gamma^m\cdot\ptd]\cap\left(\bigcup_{j=1}^{3g-3+n}\Pi^{-1}(c_j)\right),\]
enumerated so that $\ptd_j\in(\ptd_{j-1},\ptd_{j+1}]$ for all $j=1,\dots,k-1$. Observe that $k=m\cdot\sum_{j=1}^{3g-3+n}i(c,c_j)$.

Choose any $j\in\{0,\dots,k-m\}$. If we can show that $\nu\big(G(\ptd_j,\ptd_{j+m}]\big)\geq \frac{K_\nu(S')}{10}$, then by (1) of Lemma \ref{compute length}, 
\begin{eqnarray*}
i(c,\nu)&=&\nu\big(G(\ptd,\gamma\cdot\ptd]\big)\\
&=&\frac{1}{m}\cdot\nu\big(G(\ptd,\gamma^m\cdot\ptd]\big)\\
&=&\frac{1}{m}\cdot\sum_{j=0}^{\frac{k}{m}-1}\nu\big(G(\ptd_{j\cdot m},\ptd_{(j+1)\cdot m}]\big)\\
&\geq&\frac{1}{m}\frac{k}{m}\frac{K_\nu(S')}{10}.
\end{eqnarray*}
which proves the lemma.

We will now show that $\nu\big(G(\ptd_j,\ptd_{j+m}]\big)\geq \frac{K_\nu(S')}{10}$ for all $j\in\{0,\dots,k-m\}$. If the interval $(\ptd_j,\ptd_{j+m}]$ intersects $\Pi^{-1}(c_1)$ at least thrice, then Lemma \ref{intersection} implies that 
\[\nu\big(G(\ptd_j,\ptd_{j+m}]\big)\geq \frac{K_\nu(S')}{10},\] 
and we are done. (This is necessarily the case if $3g-3+n=1$.) On the other hand, if $(\ptd_j,\ptd_{j+m}]$ intersects $\Pi^{-1}(c_1)$ at most twice, then by the pigeon hole principle, there is some $j_1\in\{j,\dots,j+\frac{2m}{3}\}$ so that $(\ptd_{j_1},\ptd_{j_1+\frac{m}{3}}]$ does not intersect $\Pi^{-1}(c_1)$. In other words, there is a component $S_1$ of $S'\setminus c_1$ so that the interval $(\ptd_{j_1},\ptd_{j_1+\frac{m}{3}}]$ lies in some lift $\Std_1\subset\Std$ of the subsurface $S_1\subset S$. Since $\frac{m}{3}=3^{3g-3+n-1}\geq 1$, it follows that $S_1$ cannot be a pair of pants.

If $(\ptd_{j_1},\ptd_{j_1+\frac{m}{3}}]$ intersects $\Pi^{-1}(c_2)$ at least thrice, then Lemma \ref{intersection} again implies that 
\[\nu\big(G(\ptd_j,\ptd_{j+m}]\big)\geq\nu\big(G(\ptd_{j_1},\ptd_{j_1+\frac{m}{3}}]\big)\geq \frac{K_\nu(S_1)}{10}\geq \frac{K_\nu(S')}{10}.\] 
(This is necessarily the case if $3g-3+n=2$.) Otherwise, $(\ptd_{j_1},\ptd_{j_1+\frac{m}{3}}]$ intersects $\Pi^{-1}(c_2)$ at most twice, so there must be some $j_2\in\{j_1,\dots,j_1+\frac{2m}{9}\}\subset\{j,\dots,j+\frac{8m}{9}\}$ with the property that $(\ptd_{j_2},\ptd_{j_2+\frac{m}{9}}]$ does not intersect $\Pi^{-1}(c_1\cup c_2)$. Hence, there is a component $S_2$ of $S'\setminus (c_1\cup c_2)$ so that $(\ptd_{j_2},\ptd_{j_2+\frac{m}{9}}]$ lies in some lift $\Std_2\subset\Dbbb$ of the subsurface $S_2\subset S$. As before, $S_2$ cannot be a pair of pants because $\frac{m}{9}=3^{3g-3+n-2}\geq 1$.
 
By iterating this procedure, $3g-3+n-1$ times, we will have either already proven that $\nu\big(G(\ptd_j,\ptd_{j+m}]\big)\geq \frac{K_\nu(S')}{10}$, or have some $j_{3g-3+n-1}\in\{j,\dots,j+m-3\}$ and some component $S_{3g-3+n-1}$ of $S'\setminus (c_1\cup\dots\cup c_{3g-3+n-1})$ so that 
\begin{itemize}
\item $S_{3g-3+n-1}$ is not a pair of pants
\item $(\ptd_{j_{3g-3+n-1}},\ptd_{j_{3g-3+n-1}+3}]$ lies in some lift $\Std_{3g-3+n-1}\subset\Dbbb$ of the subsurface $S_{3g-3+n-1}\subset S$.
\end{itemize}

In this case, the unique simple closed geodesic in $S_{3g-3+n-1}$ is $c_{3g-3+n}$, and $(\ptd_{j_{3g-3+n-1}},\ptd_{j_{3g-3+n-1}+3}]$ necessarily intersects $\Pi^{-1}(c_{3g-3+n})$ at 
\[\ptd_{j_{3g-3+n-1}+1}, \,\,\ptd_{j_{3g-3+n-1}+2}\,\,\text{ and }\,\,\ptd_{j_{3g-3+n-1}+3}.\] 
Lemma \ref{intersection} then implies that 
\[\nu\big(G(\ptd_j,\ptd_{j+m}]\big)\geq\nu\big(G(\ptd_{j_{3g-3+n-1}},\ptd_{j_{3g-3+n-1}+3}]\big)\geq \frac{K_\nu(S_{3g-3+n-1})}{10}\geq \frac{K_\nu(S')}{10}.\qedhere\] 
\end{proof}

\subsection{Length lower bounds: winding and intersection with binodal edges}
In this section, fix a $\nu$-minimal pants decomposition $\Pmc_\nu(S')$. Next, we want a lower bound of $i(c,\nu)$ in terms of $b(c)$ and $w_2(c)$. To do so, we need the following two technical lemmas. Informally, Lemma \ref{binodal bound} tells us how much length $c$ has to pick up if it crosses sufficiently many binodal edges. On the other hand, Lemma \ref{winding bound} tells us how much length $c$ has to pick up if it ``winds around" a lot between binodal edges. 

\begin{lem} \label{binodal bound}
Let $P\subset S'$ be a pair of pants given by $\Pmc_\nu(S')$ and let $\widetilde{P}\subset\Std'$ be the universal cover of $P$. Also, let $p,q\in\widetilde{P}$ be points so that $[p,q]$ intersects the geodesics in $\widetilde{\Qmc}$ transversely (if at all). Then
\[\nu\big(G(p,q]\big)\geq \max\Big\{\big|\widetilde{\Bmc}[p,q]\big|-8,0\Big\}\cdot\frac{K_\nu(P)}{16}.\]
(See Definition \ref{binodal definition} for definition of $\widetilde{\Bmc}[p,q]$.)
\end{lem}

\begin{proof}
If $k:=\big|\widetilde{\Bmc}[p,q]\big|=1,\dots,8$, the desired inequality holds, so we will assume for the rest of this proof that $k\geq 9$. Let $p_1,\dots,p_k$ be the points along $[p,q]$ that also lie in the geodesics in $\widetilde{\Bmc}[p,q]$, enumerated so that they lie along $[p,q]$ in that order. Suppose that for all $j=1,\dots,k-8$, we have $\nu\big(G(p_j,p_{j+8}]\big)\geq \frac{1}{2}K_\nu(P)$. Then 
\begin{eqnarray*}
\nu\big(G(p,q]\big)&\geq&\nu\big(G(p_1,p_k]\big)\\
&\geq&\frac{1}{8}\sum_{j=1}^{k-8}\nu\big(G(p_j,p_{j+8}]\big)\\
&\geq&(k-8)\cdot\frac{K_\nu(P)}{16}.
\end{eqnarray*}

It is thus sufficient to show that $\nu\big(G(p_j,p_{j+8}]\big)\geq\frac{1}{2} K_\nu(P)$ for all $j=1,\dots,k-8$. Fix any $j=1,\dots,k-8$. For all $i=0,\dots,8$, let $q_i:=p_{j+i}$ and let $L_i$ be the geodesic in $\widetilde{\Qmc}$ that contains $q_i$. Observe that $L_i$ and $L_{i+1}$ share a common endpoint in $\partial\Gamma$, which is the repelling fixed point of some primitive $\gamma_i\in\Gamma$ so that $[[\gamma_i]]\in\Cmc\Gmc(S)$ is a boundary component of $S'$. Denote this common endpoint by $\gamma_i^-$. We will first prove the following claim: there exist $i_1,i_2\in\{1,\dots,6\}$ so that $i_1\neq i_2$ and 
\[\big(\gamma_{i_1}\cdot (q_0,q_8]\big)\cap(q_0,q_8]\neq\emptyset\neq\big(\gamma_{i_2}\cdot (q_0,q_8]\big)\cap(q_0,q_8].\]
This will be done in the following cases.

\begin{figure}
\includegraphics[scale=0.8]{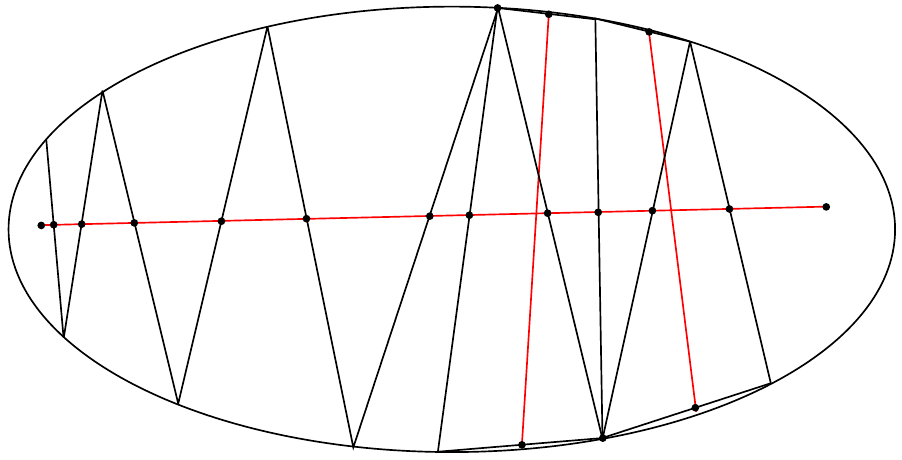}
\put (-584, 89){\makebox[0.7\textwidth][r]{\footnotesize$p$ }}
\put (-271, 96){\makebox[0.7\textwidth][r]{\footnotesize$q$ }}
\put (-567, 94){\makebox[0.7\textwidth][r]{\footnotesize$q_0$ }}
\put (-556, 86){\makebox[0.7\textwidth][r]{\footnotesize$q_1$ }}
\put (-544, 86){\makebox[0.7\textwidth][r]{\footnotesize$q_2$ }}
\put (-503, 87){\makebox[0.7\textwidth][r]{\footnotesize$q_3$ }}
\put (-478, 88){\makebox[0.7\textwidth][r]{\footnotesize$q_4$ }}
\put (-423, 89){\makebox[0.7\textwidth][r]{\footnotesize$q_5$ }}
\put (-377, 99){\makebox[0.7\textwidth][r]{\footnotesize$q_6$ }}
\put (-346, 91){\makebox[0.7\textwidth][r]{\footnotesize$q_7$ }}
\put (-307, 101){\makebox[0.7\textwidth][r]{\footnotesize$q_8$ }}
\put (-575, 70){\makebox[0.7\textwidth][r]{\footnotesize$L_0$ }}
\put (-558, 68){\makebox[0.7\textwidth][r]{\footnotesize$L_1$ }}
\put (-540, 66){\makebox[0.7\textwidth][r]{\footnotesize$L_2$ }}
\put (-507, 64){\makebox[0.7\textwidth][r]{\footnotesize$L_3$ }}
\put (-473, 62){\makebox[0.7\textwidth][r]{\footnotesize$L_4$ }}
\put (-420, 132){\makebox[0.7\textwidth][r]{\footnotesize$L_5$ }}
\put (-395, 130){\makebox[0.7\textwidth][r]{\footnotesize$L_6$ }}
\put (-345, 50){\makebox[0.7\textwidth][r]{\footnotesize$L_7$ }}
\put (-292, 45){\makebox[0.7\textwidth][r]{\footnotesize$L_8$ }}
\put (-400, 180){\makebox[0.7\textwidth][r]{\footnotesize$\gamma_5^-$ }}
\put (-357, 1){\makebox[0.7\textwidth][r]{\footnotesize$\gamma_6^-$ }}
\put (-396, 10){\makebox[0.7\textwidth][r]{\footnotesize$\gamma_5\cdot q_4$ }}
\put (-409, 89){\makebox[0.7\textwidth][r]{\footnotesize$x$ }}
\put (-370, 176){\makebox[0.7\textwidth][r]{\footnotesize$\gamma_5\cdot x$ }}
\put (-359, 99){\makebox[0.7\textwidth][r]{\footnotesize$y$ }}
\put (-331, 22){\makebox[0.7\textwidth][r]{\footnotesize$\gamma_6\cdot y$ }}
\put (-330, 169){\makebox[0.7\textwidth][r]{\footnotesize$\gamma_6\cdot x$ }}
\caption{Case 1 of proof of Lemma \ref{binodal bound}, with $k_1=5$ and $k_2=6$.}
\label{figure8}
\end{figure}
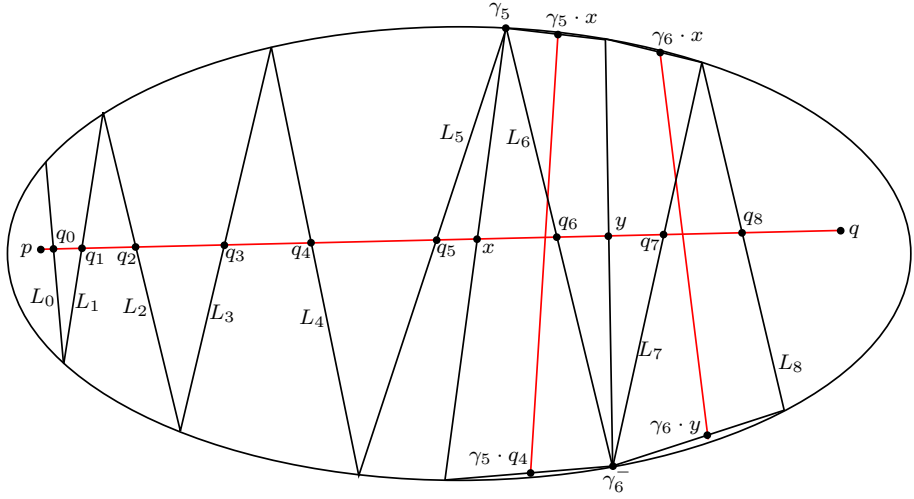

\textbf{Case 1: There is some $k_1,k_2\in\{1,\dots,6\}$ so that $k_1\neq k_2$ and $\suc(L_{k_t})\neq L_{k_t+1}$ for $t=1,2$.} In this case, let $i_t=k_t$. By replacing $\gamma_{i_t}$ with $\gamma_{i_t}^{-1}$ if necessary, we can assume that $\suc^2(L_{i_t})=\gamma_{i_t}\cdot L_{i_t}$. Observe that $\gamma_{i_t}\cdot\suc^{-1}(L_{i_t})$ is an edge in $\widetilde{\Qmc}$ that forms a triangle with $\suc(L_{i_t})$ and $\suc^2(L_{i_t})$. On the other hand, $\gamma_{i_t}\cdot\suc(L_{i_t+1})$ is an edge in $\widetilde{\Qmc}$ whose endpoints in $\partial\Dbbb$ both lie in $(\gamma_{i_t}^-,\gamma_{i_{t+2}}^-)_{\gamma_{i_{t+1}}^-}$ (see Notation \ref{interval notation 2}). Thus, $\big(\gamma_{i_t}\cdot (q_0,q_8]\big)\cap(q_0,q_8]$ is non-empty (see Figure \ref{figure8}).

\begin{figure}
\includegraphics[scale=0.8]{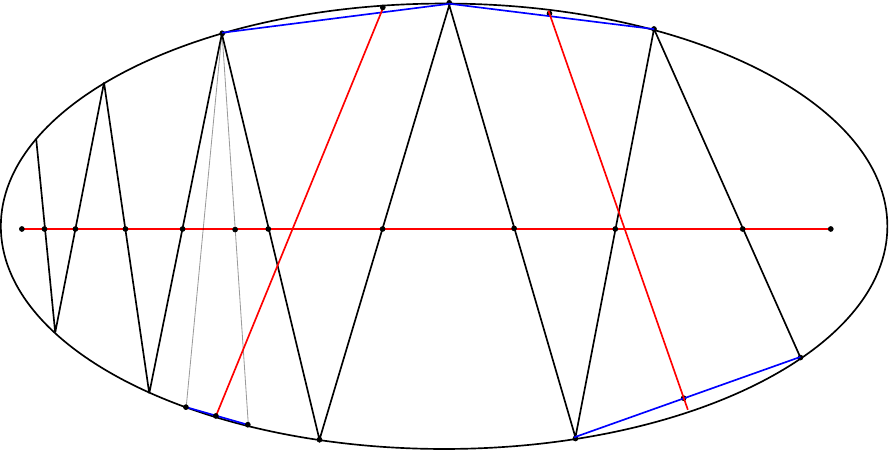}
\put (-583, 87){\makebox[0.7\textwidth][r]{\footnotesize$p$ }}
\put (-266, 87){\makebox[0.7\textwidth][r]{\footnotesize$q$ }}
\put (-501, 87){\makebox[0.7\textwidth][r]{\footnotesize$x$ }}
\put (-566, 88){\makebox[0.7\textwidth][r]{\footnotesize$q_0$ }}
\put (-555, 80){\makebox[0.7\textwidth][r]{\footnotesize$q_1$ }}
\put (-544, 88){\makebox[0.7\textwidth][r]{\footnotesize$q_2$ }}
\put (-521, 88){\makebox[0.7\textwidth][r]{\footnotesize$q_3$ }}
\put (-479, 88){\makebox[0.7\textwidth][r]{\footnotesize$q_4$ }}
\put (-436, 81){\makebox[0.7\textwidth][r]{\footnotesize$q_5$ }}
\put (-393, 81){\makebox[0.7\textwidth][r]{\footnotesize$q_6$ }}
\put (-355, 81){\makebox[0.7\textwidth][r]{\footnotesize$q_7$ }}
\put (-294, 81){\makebox[0.7\textwidth][r]{\footnotesize$q_8$ }}
\put (-573, 68){\makebox[0.7\textwidth][r]{\footnotesize$L_0$ }}
\put (-555, 68){\makebox[0.7\textwidth][r]{\footnotesize$L_1$ }}
\put (-541, 68){\makebox[0.7\textwidth][r]{\footnotesize$L_2$ }}
\put (-525, 66){\makebox[0.7\textwidth][r]{\footnotesize$L_3$ }}
\put (-472, 62){\makebox[0.7\textwidth][r]{\footnotesize$L_4$ }}
\put (-430, 132){\makebox[0.7\textwidth][r]{\footnotesize$L_5$ }}
\put (-396, 130){\makebox[0.7\textwidth][r]{\footnotesize$L_6$ }}
\put (-361, 50){\makebox[0.7\textwidth][r]{\footnotesize$L_7$ }}
\put (-292, 52){\makebox[0.7\textwidth][r]{\footnotesize$L_8$ }}
\put (-505, 163){\makebox[0.7\textwidth][r]{\footnotesize$\gamma_3^-$ }}
\put (-468, -3){\makebox[0.7\textwidth][r]{\footnotesize$\gamma_4^-$ }}
\put (-408, 176){\makebox[0.7\textwidth][r]{\footnotesize$\gamma_5^-$ }}
\put (-360, -2){\makebox[0.7\textwidth][r]{\footnotesize$\gamma_6^-$ }}
\put (-330, 167){\makebox[0.7\textwidth][r]{\footnotesize$\gamma_7^-$ }}
\put (-274, 29){\makebox[0.7\textwidth][r]{\footnotesize$\gamma_8^-$ }}
\put (-370, 173){\makebox[0.7\textwidth][r]{\footnotesize$\gamma_5\cdot q_8$ }}
\put (-330, 21){\makebox[0.7\textwidth][r]{\footnotesize$\gamma_5\cdot x$ }}
\put (-435,175){\makebox[0.7\textwidth][r]{\footnotesize$\gamma_3\cdot q_5$}}
\put (-500,6){\makebox[0.7\textwidth][r]{\footnotesize$\gamma_3\cdot q_1$}}
\caption{Case 2 of proof of Lemma \ref{binodal bound}, with $i_1=3$ and $i_2=5$.}
\label{figure10}
\end{figure}
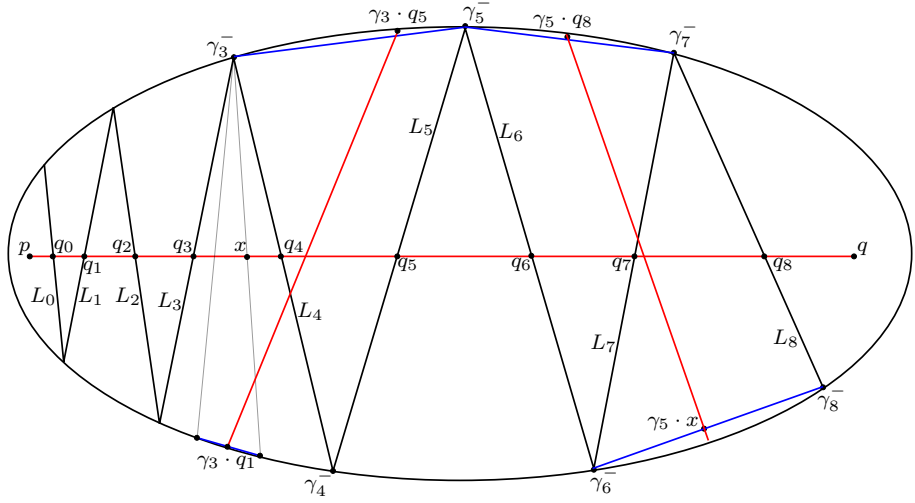

\textbf{Case 2: There is a unique $k\in\{1,\dots,6\}$ so that $\suc(L_k)\neq L_{k+1}$.} In this case, let $i_1=k$, let $i_2=5$ if $i_1\leq 3$ and let $i_2=2$ if $i_1\geq 4$. The same argument as Case 1 will show that $\big(\gamma_{i_1}\cdot (q_0,q_8]\big)\cap(q_0,q_8]$ is non-empty. We will now prove that $\big(\gamma_{i_2}\cdot (q_0,q_8]\big)\cap(q_0,q_8]$ is non-empty when $i_1\leq 3$; the case when $i_1\geq 4$ is similar. By replacing $\gamma_5$ by $\gamma_5^{-1}$ if necessary, we can assume that $\gamma_5\cdot \gamma_4^-=\gamma_7^-$. Observe then that $\gamma_5\cdot \gamma_3^-=\gamma_6^-$, $\gamma_5\cdot L_8\subset [\gamma_5^-,\gamma_7^-]_{\gamma_6^-}$, and $\gamma_5\cdot \suc^{-1}(L_4)=\{\gamma_6^-,\gamma_8^-\}$ (see Figure \ref{figure10}). In particular, $\big(\gamma_{i_2}\cdot (q_0,q_8]\big)\cap(q_0,q_8]$ is non-empty.

\textbf{Case 3: For all $k\in\{1,\dots,6\}$, $\suc(L_k)=L_{k+1}$.} In this case, let $i_1=2$ and let $i_2=5$. The argument given in Case 2 proves that $\big(\gamma_{i_t}\cdot (q_0,q_8]\big)\cap(q_0,q_8]$ is non-empty for $t=1,2$. This concludes the proof of the claim.

Next, we will use the claim to prove the lemma. Assume without loss of generality that $i_1<i_2$. Let $x_1,x_2\in (q_0,q_8]$ be points so that $\gamma_{i_t}\cdot x_t\in(q_0,q_8]$. (They exist because of the claim.) By replacing each $\gamma_{i_t}$ with $\gamma_{i_t}^{-1}$ if necessary, we can assume that $x_t,\gamma_{i_t}\cdot x_t$ lie along $(q_0,q_8]$ in that order. Observe then that $x_1$ has to lie in $(q_{i_1-1},q_{i_1+1})$, $\gamma_{i_1}\cdot x_1$ has to lie in $(q_{i_1},q_{i_1+2})$, $x_2$ has to lie in $(q_{i_2-1},q_{i_2+1})$ and $\gamma_{i_2}\cdot x_2$ has to lie in $(q_{i_2},q_{i_2+2})$. In particular, $x_1,\gamma_{i_2}\cdot x_2$ lie along $(q_0,q_8]$ in that order, and $\gamma_{i_1}\cdot x_1,x_2\in(x_1,\gamma_{i_2}\cdot x_2]$. It is clear that $[[\gamma_{i_2}\cdot\gamma_{i_1}]]\in\Cmc\Gmc(P)$ is non-peripheral. Hence, Lemma \ref{compute length} implies that
\begin{eqnarray*}
2\nu\big(G(q_0,q_8]\big)& \geq & \nu\big(G(x_1,\gamma_{i_2}\cdot x_2]\big)+\nu\big(G(x_2,\gamma_{i_1}\cdot x_1]\big)\\
& = & \nu\big(G(x_1,\gamma_{i_2}\cdot x_2]\big)+\nu\big(G(\gamma_{i_2}\cdot x_2,(\gamma_{i_2}\gamma_{i_1})\cdot x_1]\big)\\
& \geq & \nu\big(G(x_1,(\gamma_{i_2}\gamma_{i_1})\cdot x_1]\big)\\
& \geq & i([[\gamma_{i_2}\cdot\gamma_{i_1}]],\nu)\\
& \geq & K_\nu(P).\hspace{7cm}\qedhere 
\end{eqnarray*}
\end{proof}

\begin{lem}\label{winding bound}
Let $c\in\Pmc_\nu(S')$, let $\gamma_c\in\Gamma'$ so that $[[\gamma_c]]=c$, and let $p,q\in\overline{\Omega}_{\gamma_c}\subset\Dbbb$ so that $[p,q]$ intersects 
\[\bigcup_{j\in\Zbbb}\gamma_c^j\cdot [p_{\gamma_c}^+,p_{\gamma_c}^-]\]
transversely. Then 
\[\nu\big(G[p,q]\big)\geq (w_2[p,q]-1)\cdot\frac{i(c,\nu)}{2}.\]
(See Section \ref{Minimal pants decompositions and related structures} for the definition of $p^\pm_{\gamma_c}$ and Notation \ref{w notation} for the definition of $w_2[p,q]$.)
\end{lem}

\begin{proof}

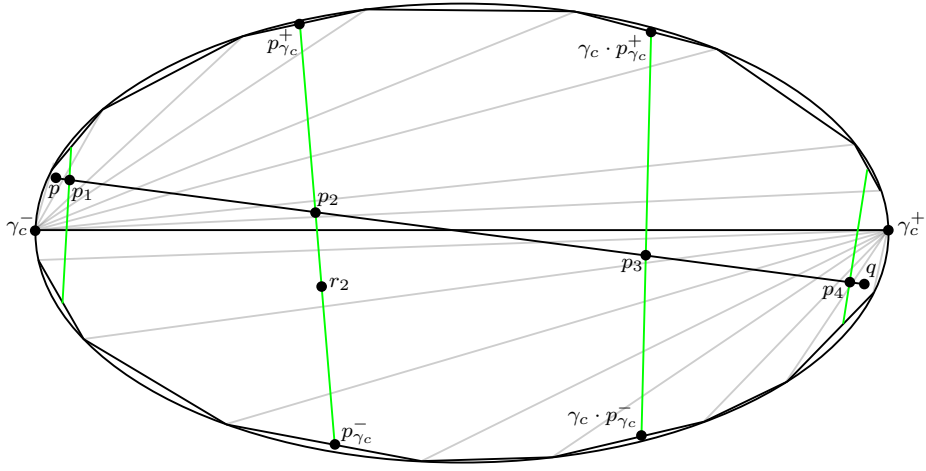
\begin{figure}
\includegraphics[scale=0.9]{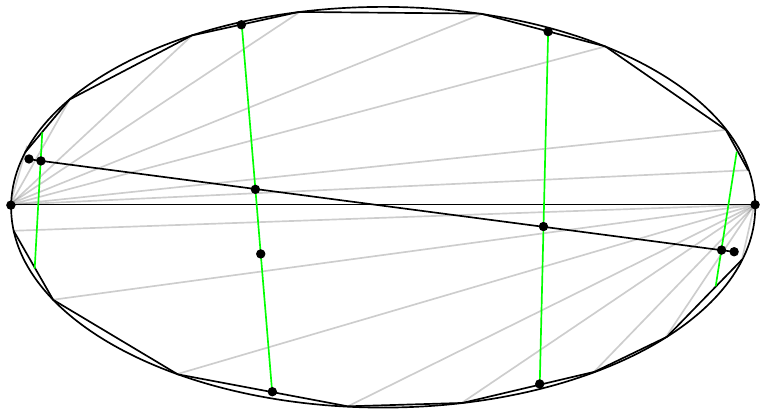}
\put (-566, 105){\makebox[0.7\textwidth][r]{\footnotesize$p$ }}
\put (-257, 75){\makebox[0.7\textwidth][r]{\footnotesize$q$ }}
\put (-553, 104){\makebox[0.7\textwidth][r]{\footnotesize$p_1$ }}
\put (-460, 102){\makebox[0.7\textwidth][r]{\footnotesize$p_2$ }}
\put (-345, 77){\makebox[0.7\textwidth][r]{\footnotesize$p_3$ }}
\put (-269, 67){\makebox[0.7\textwidth][r]{\footnotesize$p_4$ }}
\put (-456, 70){\makebox[0.7\textwidth][r]{\footnotesize$r_2$ }}
\put (-475, 160){\makebox[0.7\textwidth][r]{\footnotesize$p_{\gamma_c}^+$ }}
\put (-447, 13){\makebox[0.7\textwidth][r]{\footnotesize$p_{\gamma_c}^-$ }}
\put (-347, 19){\makebox[0.7\textwidth][r]{\footnotesize$\gamma_c\cdot p_{\gamma_c}^-$ }}
\put (-343, 158){\makebox[0.7\textwidth][r]{\footnotesize$\gamma_c\cdot p_{\gamma_c}^+$ }}
\put (-575, 91){\makebox[0.7\textwidth][r]{\footnotesize$\gamma_c^-$ }}
\put (-238, 91){\makebox[0.7\textwidth][r]{\footnotesize$\gamma_c^+$ }}
\caption{Proof of Lemma \ref{winding bound} when $j=2$.}
\label{figure9}
\end{figure}

Let $k=w_2[p,q]$ and note that if $k=0,1$ there is nothing to prove. Therefore, assume $k\geq 2$ and let $p_1,\dots,p_k$ be the points in
\[[p,q]\cap\bigg(\bigcup_{j\in\Zbbb}\gamma_c^j\cdot [p_{\gamma_c}^+,p_{\gamma_c}^-]\bigg)\]
in that order along $[p,q]$. Fix any $j=1,\dots,k-1$, let $r_j:=\gamma_c^{-1}\cdot p_{j+1}$ and assume without loss of generality that $p_j\in[p_{\gamma_c}^+,p_{\gamma_c}^-]$. Also, assume that $p_{\gamma_c}^+,p_j,r_j,p_{\gamma_c}^-$ lie along $[p_{\gamma_c}^+,p_{\gamma_c}^-]$ in that order; the other case is similar. Then
\begin{eqnarray*}
\nu\big(G[p_j,r_j)\big)&=&\nu\big(G[p_{\gamma_c}^+,p_{\gamma_c}^-]\big)-\nu\big(G[p_{\gamma_c}^+,p_j)\big)-\nu\big(G[r_j,p_{\gamma_c}^-]\big)\\
&\leq&\nu\big(G[p_{\gamma_c}^+,\gamma_c\cdot p_{\gamma_c}^-]\big)-\nu\big(G[p_{\gamma_c}^+,p_j)\big)-\nu\big(G[p_{j+1},\gamma_c\cdot p_{\gamma_c}^-]\big)\\
&\leq&\nu\big(G[p_j,p_{j+1})\big),
\end{eqnarray*}
where the first inequality above is a consequence of the way $p_{\gamma_c}^+$ and $p_{\gamma_c}^-$ are defined.

By (2) of Lemma \ref{compute length}, we have 
\begin{eqnarray*}
i(c,\nu) &\leq&\nu\big(G[p_{j+1},r_j)\big)\\
&\leq&\nu\big(G[p_{j+1},p_j)\big)+\nu\big(G[p_j,r_j)\big)\\
&\leq&\nu\big(G[p_{j+1},p_j)\big)+\nu\big(G[p_j,p_{j+1})\big).
\end{eqnarray*}
Hence,
\begin{eqnarray*}
\nu\big(G[p,q]\big)&\geq&\frac{1}{2}\Big(\nu\big(G[p_1,p_k)\big)+\nu\big(G(p_1,p_k]\big)\Big)\\
&=&\frac{1}{2}\sum_{j=1}^{k-1}\Big(\nu\big(G[p_j,p_{j+1})\big)+\nu\big(G(p_j,p_{j+1}]\big)\Big)\\
&\geq&(k-1)\cdot\frac{i(c,\nu)}{2}. \hspace{6cm} \qedhere
\end{eqnarray*}
\end{proof}

\subsection{Length lower bounds: the combinatorial description}
Combining the previous lemmas in this section, we can obtain the following lower bound for $i(c,\nu)$ in terms of the $\nu$-panted systole length and the $\nu$-systole length.

\begin{thm}\label{length bound theorem}
Let $S'\subset S$ be a connected essential subsurface of genus $g$ with $n$ boundary components, let
\[\overline{K}_\nu(S'):=\frac{K_\nu(S')}{400\cdot 3^{3g-3+n}+96}\,\,\,\,\,\text{and let}\,\,\,\,\,\overline{L}_\nu(S'):=\frac{L_\nu(S')}{400\cdot 3^{3g-3+n}+96}.\] 
Then 
\[i(c,\nu)\geq b(c)\cdot\overline{K}_\nu(S')+w_1(c)\cdot\overline{L}_\nu(S').\]
\end{thm}

\begin{proof}
By Lemma \ref{pants curve bound}, we know that 
\begin{equation}\label{pants curve}
5\cdot 3^{3g-3+n}\cdot i(c,\nu)\geq p(c)\cdot \frac{K_\nu(S')}{2}.
\end{equation}
Let $\gamma\in\Gamma'$ so that $[[\gamma]]=c\in\Cmc\Gmc(S')$, let $\qtd\in\Std'\subset\Dbbb$ so that 
\[\qtd\in\Pi^{-1}\bigg(c\cap\bigg(\bigcup_{j=1}^{3g-3+n}c_j\bigg)\bigg),\]
(recall $\Pmc_\nu(S')=\{c_1,\dots,c_{3g-3+2n}\}$, where $c_{3g-3+n+1},\dots, c_{3g-3+2n}$ are the boundary components of $S'$) and let $\qtd=\qtd_0,\qtd_1,\dots,\qtd_{p(c)}=\gamma\cdot\qtd$ be the points in
\[[\qtd,\gamma\cdot\qtd]\cap\bigg(\bigcup_{j=1}^{3g-3+n}\Pi^{-1}(c_j)\bigg),\]
enumerated so that $\qtd_j\in(\qtd_{j-1},\qtd_{j+1})$ for all $j=1,\dots,k-1$. 

Note that for any pair of pants $P$ given by $\Pmc_\nu(S')$, $K_\nu(P)\geq K_\nu(S')$. Hence, by Lemma \ref{binodal bound}, we have 
\begin{eqnarray}\label{binodal}
i(c,\nu)&=&\sum_{j=0}^{p(c)-1}\nu\big(G(\qtd_j,\qtd_{j+1}]\big)\nonumber\\
&\geq&\sum_{j=0}^{p(c)-1} \max\Big\{\big|\widetilde{\Bmc}[\qtd_j,\qtd_{j+1}]\big|-8,0\Big\}\frac{K_\nu(P)}{16}\\
&\geq&\sum_{j=0}^{p(c)-1}\Big(\big|\widetilde{\Bmc}[\qtd_j,\qtd_{j+1}]\big|-8\Big)\cdot\frac{K_\nu(S')}{16}\nonumber\\
&\geq&b(c)\cdot\frac{K_\nu(S')}{16}-p(c)\cdot\frac{K_\nu(S')}{2}.\nonumber
\end{eqnarray}
Adding the inequalities (\ref{pants curve}) and (\ref{binodal}) then gives
\begin{equation}\label{intermediate}
(5\cdot 3^{3g-3+n}+1)i(c,\nu)\geq b(c)\cdot\frac{K_\nu(S')}{16}.
\end{equation}

Let $k_j:=\big|\widetilde{\Bmc}[\qtd_j,\qtd_{j+1}]\big|$. For each interval $[\qtd_j,\qtd_{j+1}]$, let $\etd_{j,1},\dots,\etd_{j,k_j}$ be the edges in $\widetilde{\Bmc}[\qtd_j,\qtd_{j+1}]$, enumerated so that $\ptd_{j,i}\in(\ptd_{j,i-1},\ptd_{j,i+1})$ for all $i=2,\dots,k_j-1$, where $\ptd_{j,i}:=\etd_{j,i}\cap[\qtd_j,\qtd_{j+1}]$. We previously observed (see discussion after Notation \ref{t notation}) that the interval $\big(\suc^{-1}(\etd_{j,i})\cap[\qtd_j,\qtd_{j+1}],\suc(\etd_{j,i+1})\cap[\qtd_j,\qtd_{j+1}]\big]$ lies in $\Omega_{\gamma(\etd_{j,i},\etd_{j,i+1})}$. Also, it is clear that for every point $r\in[\qtd_j,\qtd_{j+1}]$, there are at most four different values of $i$ so that $r\in\big[\suc^{-1}(\etd_{j,i})\cap[\qtd_j,\qtd_{j+1}],\suc(\etd_{j,i+1})\cap[\qtd_j,\qtd_{j+1}]\big]$. Thus, by Lemma \ref{winding bound},
\begin{eqnarray}\label{winding}
4i(c,\nu)&=&\sum_{j=0}^{p(c)-1}4\nu\big(G(\qtd_j,\qtd_{j+1}]\big)\nonumber\\
&\geq&\sum_{j=0}^{p(c)-1}\sum_{i=1}^{k_j}\nu\big(G\big[\suc^{-1}(\etd_{j,i})\cap[\qtd_j,\qtd_{j+1}],\suc(\etd_{j,i+1})\cap[\qtd_j,\qtd_{j+1}]\big]\big)\nonumber\\
&\geq&\sum_{j=0}^{p(c)-1}\sum_{i=1}^{k_j}\Big(w_2\big[\suc^{-1}(\etd_{j,i})\cap[\qtd_j,\qtd_{j+1}],\suc(\etd_{j,i+1})\cap[\qtd_j,\qtd_{j+1}]\big]-1\Big)\cdot\frac{L_\nu(S')}{2}\nonumber\\
&=&(w_2(c)-b(c))\cdot\frac{L_\nu(S')}{2}\nonumber
\end{eqnarray}

We finish the proof by combining the above inequalities to obtain a positive lower bound for $i(c,\nu)$. To ensure positivity of this lower bound, we need to add the above inequality to a multiple (strictly greater than 16) of inequality (\ref{intermediate}). We choose this multiple to be 20 to improve readability. 
\begin{eqnarray*}
(100\cdot 3^{3g-3+n}+24)i(c,\nu)&\geq& (w_2(c)-b(c))\cdot\frac{L_\nu(S')}{2}+b(c)\cdot \frac{5\cdot K_\nu(S')}{4}\\
&\geq &w_2(c)\cdot\frac{L_\nu(S')}{2}+b(c)\cdot \frac{3\cdot K_\nu(S')}{4}\\
&\geq &\left(\frac{1}{2}w_1(c)-b(c)\right)\cdot\frac{L_\nu(S')}{2}+b(c)\cdot \frac{3\cdot K_\nu(S')}{4}\\
&\geq &w_1(c)\cdot\frac{L_\nu(S')}{4}+b(c)\cdot \frac{K_\nu(S')}{4},
\end{eqnarray*}
where the third inequality follows from Lemma \ref{w1 and w2}. Dividing both sides by $100\cdot 3^{3g-3+n}+24$ yields the required inequality.
\end{proof}

\section{Vanishing of entropy and a systolic inequality}\label{entropy}

\textbf{For the rest of this section, let $S'\subset S$ be a connected essential subsurface of genus $g$ with $n$ boundary components.} If we choose a period minimizing $\nu\in\Cmc(S)$, we can associate to $S'$ a quantity which we call the topological entropy.

\begin{definition}
Let $\nu\in\Cmc(S)$ be a period minimizing geodesic current. The \emph{$\nu$-topological entropy} of $S'$ is
\[h_\nu(S'):=\limsup_{T\to\infty}\frac{1}{T}\log\big|\{c\in\Cmc\Gmc(S'):i(c,\nu)\leq T\}\big|.\]
When $S'=S$, we will use the notation $h_\nu:=h_\nu(S)$.
\end{definition}

\subsection{A systolic inequality}
The goal of this section is to prove Theorem \ref{main theorem}. The constant $C$ that arises from our proof is $400\cdot 3^{3g-3+n}+96$. We will divide the proof of Theorem \ref{main theorem} into three lemmas. The first two lemmas give us the first inequality.

\begin{lem}\label{nicepants}
Let $\nu\in\Cmc(S)$ be a period minimizing geodesic current and $\Pmc_{\nu}(S')$ be a $\nu$-minimal pants decompositions of $S'$. Then there exists a pair of pants $P\subset S'$ and a closed geodesic $e\in\Cmc\Gmc(P)$ so that
\begin{itemize}
\item $e$ has a unique self-intersection point $p$,
\item $i(e,\nu)\leq 4K_\nu(S')$, 
\item the three closed geodesics obtained by performing surgery to $e$ at $p$ are the three boundary components of $P$.
\end{itemize}
\end{lem}

\begin{proof} By Lemma \ref{panted systole}, it is sufficient to construct a primitive, non-simple $\bar{e}\in\Cmc\Gmc(S')$ so that $i(\bar{e},\nu)\leq 4K_\nu(S')$.

Let $d\in\Cmc\Gmc(S')$ be a closed geodesic that is not a multiple of a curve in $\Pmc_\nu(S')$, and so that $i(d,\nu)=K_\nu(S')$. Note that $d$ is primitive. If $d$ is non-simple, set $\bar{e}$ to be $d$ and we are done. If $d$ is simple, then there is some $c\in\Pmc_\nu(S')$ that intersects $d$ transversely, so that $i(c,\nu)\leq i(d,\nu)$. Choose a hyperbolic structure $\Sigma$ on $S$. There exists $\gamma,\eta\in\Gamma'$ so that $[[\gamma]]=c$, $[[\tau]]=d$, and the axes $L_\tau$ and $L_\gamma$ of $\tau$ and $\gamma$ respectively intersect transversely. Let $p\in\Std$ be the intersection point of $L_\tau$ and $L_\gamma$. 

By Lemma \ref{compute length}, we have that 
\begin{align*}
i\big([[\tau\gamma\tau\gamma^{-1}]],\nu\big)&\leq\nu\big(G(\gamma\cdot p,\tau\gamma\tau\gamma^{-1}\cdot (\gamma\cdot p)]\big)\\
&\leq\nu\big(G(\gamma\cdot p,p]\big)+\nu\big(G(p,\tau\cdot p]\big)+\nu\big(G(\tau\cdot p,\tau\gamma\cdot p]\big)\\
&\hspace{1cm}+\nu\big(G(\tau\gamma\cdot p,\tau\gamma\tau\cdot p]\big)\\
&=\nu\big(G(\gamma\cdot p,p]\big)+\nu\big(G(p,\tau\cdot p]\big)+\nu\big(G(p,\gamma\cdot p]\big)+\nu\big(G(p,\tau\cdot p]\big)\\
&= 2i(c,\nu)+2i(d,\nu)\\
&\leq 4i(d,\nu)\\
&=4K_\nu(S').
\end{align*}
It is easy to see that $[[\tau\gamma\tau\gamma^{-1}]]$ is non-simple, so we can set $\bar{e}$ to be $[[\tau\gamma\tau\gamma^{-1}]]$.
\end{proof}

\begin{lem}\label{first inequality}
Let $\nu\in\Cmc(S)$ be a period minimizing geodesic current. Then 
\[h_\nu(S')\geq\frac{\log(2)}{4K_\nu(S')}.\]
\end{lem}

\begin{proof} Let $\Pmc_{\nu}(S')$, $P$, $e$ and $p$ be as in Lemma \ref{nicepants}. Let $\gamma_1,\gamma_2,\gamma_3\in\pi_1(P)$ be primitive elements so that $\gamma_3\cdot\gamma_2\cdot\gamma_1=\id$, and so that the closed geodesics $[[\gamma_1]]$, $[[\gamma_2]]$, $[[\gamma_3]]$ are the boundary components of $P$. Then $e$ has to be either
\[[[\gamma_3^{-1}\cdot\gamma_2]],\,\,\,[[\gamma_2^{-1}\cdot\gamma_1]]\,\,\,\text{ or }\,\,\,[[\gamma_1^{-1}\cdot\gamma_3]].\]

Assume without loss of generality that $e=[[\gamma_2^{-1}\cdot\gamma_1]]$. Then 
\begin{align*}
i(e,\nu)&=\nu\big(G(p,\gamma_2^{-1}\gamma_1\cdot p]\big)\nonumber\\
&=\nu\big(G(\gamma_2\cdot p,\gamma_1\cdot p]\big)\nonumber\\
&=\nu\big(G(\gamma_2\cdot p, p]\big)+\nu\big(G(p,\gamma_1\cdot p]\big)\\
&=\nu\big(G(p,\gamma_2^{-1}\cdot p]\big)+\nu\big(G(p,\gamma_1^{-1}\cdot p]\big),\nonumber
\end{align*}
so $i(e,\nu)\geq \nu\big(G(\gamma_2\cdot p, p]\big),\nu\big(G(p,\gamma_1\cdot p]\big),\nu\big(G(p,\gamma_2^{-1}\cdot p]\big)$ and $\nu\big(G(p,\gamma_1^{-1}\cdot p]\big)$.
Since $\pi_1(P)\subset\Gamma'$ is a free group of rank 2 generated by $\gamma_1$ and $\gamma_2$, no two distinct elements of the form
\[\gamma_1^{\epsilon_1}\gamma_2^{\delta_1}\dots\gamma_1^{\epsilon_t}\gamma_2^{\delta_t}\gamma_1^2\]
are conjugate, where $\epsilon_i,\delta_i=\pm1$. 

By Lemma \ref{compute length}, we have that for any $f=[[\gamma_1^{\epsilon_1}\gamma_2^{\delta_1}\dots\gamma_1^{\epsilon_t}\gamma_2^{\delta_t}\gamma_1^2]]\in\Cmc\Gmc(S')$, 
\begin{align*}
i(f,\nu)&\leq \nu\big(G(p,\gamma_1^{\epsilon_1}\gamma_2^{\delta_1}\dots\gamma_1^{\epsilon_t}\gamma_2^{\delta_t}\gamma_1^2\cdot p]\big)\\
&\leq \nu\big(G(p,\gamma_1^{\epsilon_1}\cdot p]\big)+\nu\big(G(\gamma_1^{\epsilon_1}\cdot p,\gamma_1^{\epsilon_1}\gamma_2^{\delta_1}\cdot p]\big)+\dots\\
&\hspace{.4cm}+\nu\big(G(\gamma_1^{\epsilon_1}\gamma_2^{\delta_1}\dots\gamma_1^{\epsilon_t}\gamma_2^{\delta_t}\cdot p,\gamma_1^{\epsilon_1}\gamma_2^{\delta_1}\dots\gamma_1^{\epsilon_t}\gamma_2^{\delta_t}\gamma_1^2\cdot p]\big)\\
&=\nu\big(G(p,\gamma_1^{\epsilon_1}\cdot p]\big)+\nu\big(G(p,\gamma_2^{\delta_1}\cdot p]\big)+\dots+\nu\big(G(p,\gamma_2^{\delta_t}\cdot p]\big)+2\nu\big(G(p,\gamma_1\cdot p]\big)\\
&\leq (2t+2)i(e,\nu)\\
&\leq  8(t+1) K_\nu(S').
\end{align*}
This means that 
\begin{align*}
&\hspace{.45cm}\big|\big\{[\gamma]\in[\Gamma'\setminus\{\id\}]:i\big([[\gamma]],\nu\big)\leq T\big\}\big|\\
&\geq\left|\left\{[\gamma_1^{\epsilon_1}\gamma_2^{\delta_1}\dots\gamma_1^{\epsilon_t}\gamma_2^{\delta_t}\gamma_1^2]\in[\pi_1(P)\setminus\{\id\}]:\epsilon_i,\delta_i\in\{-1,1\} \text{ and }t\leq \frac{T}{8K_\nu(S')}-1\right\}\right|\\
&\geq 4^{\left\lfloor\frac{T}{8K_\nu(S')}\right\rfloor-1},
\end{align*}
so $h_\nu(S')\geq \frac{\log(2)}{4K_\nu(S')}$.
\end{proof}

Now, we finish the proof of Theorem \ref{main theorem} by proving the second inequality.

\begin{lem}\label{second inequality}
There is a constant $C\in\Rbbb^+$ which depends only on the topology of $S'$, so that for any period minimizing $\nu\in\Cmc(S)$, we have 
\[h_\nu(S')K_\nu(S')\leq C\cdot\left(\log(4)+1+\log\left(1+\frac{1}{x_0}\right)\right),\]
where $x_0$ is the unique positive solution to the equation $(1+x)^{\left\lceil\frac{K_\nu(S')}{L_\nu(S')}-1\right\rceil}x=1$.
\end{lem}

\begin{proof}
Simplify notation by denoting $\overline{K}_\nu(S')$ and $\overline{L}_\nu(S')$ defined in the statement of Theorem \ref{length bound theorem} simply by $K$ and $L$ respectively. Choose a $\nu$-minimal pants decomposition $\Pmc_\nu(S')$. By Theorem \ref{length bound theorem} and Proposition \ref{combinatorial prop}, we have
\begin{eqnarray*}
\Big|\big \{c\in\Cmc\Gmc(S'):i(c,\nu)\leq T\big\}\Big|&\leq& \Big|\big \{c\in\Cmc\Gmc(S'):b(c)\cdot K+w_1(c)\cdot L\leq T\big\}\Big|\\
&\leq&\frac{1}{2}\Big|\big \{\sigma\in\Psi':B(\sigma)\cdot K+W_1(\sigma)\cdot L\leq T\big\}\Big|\\
&=&\frac{1}{2}\sum_{i=1}^{\left\lfloor\frac{T}{K}\right\rfloor}\left|\left \{\sigma\in\Psi':B(\sigma)=i,W_1(\sigma)\leq \left\lfloor\frac{T-Ki}{L}\right\rfloor\right\}\right|.
\end{eqnarray*}
(See Definition \ref{admissible sequence} for the definition of $\Psi'$.)

If $\sigma=\{(u_i,v_i,w_i,T_i,t_i)\}_{i=1}^m$, let $\sigma'$ be the cyclic sequence 
\[\sigma':=\{(u_i,v_i,w_i,T_i)\}_{i=1}^m.\] 
For any $e\in \Qmc$, let $e',e''\in\Qmc$ be the geodesics so that $\{e,e',e''\}=\Qmc_j$ for some $j$. If $v_i=e$, then there are four possibilities for $(u_i,v_i,w_i,T_i)$, namely 
\[(e',e,e'',S), (e',e,e'',Z),(e'',e,e',S), (e'',e,e',Z).\] 
Since $|\{e\in\Qmc:e\subset S'\}|=6g-6+3n$, we see from the definition of $\Psi'$ that 
\[\Big|\big \{\sigma':\sigma\in\Psi',B(\sigma)=i\big\}\Big|\leq\frac{(24g-24+12n)\cdot 4^{i-1}}{i}.\]
Hence, 
\[\left|\left \{\sigma:B(\sigma)=i,W_1(\sigma)\leq \left\lfloor\frac{T-Ki}{L}\right\rfloor\right\}\right|\leq\frac{(24g-24+12n)\cdot 4^{i-1}}{i}\cdot{\left\lfloor\frac{T-Ki}{L}\right\rfloor+i\choose i},\]
which implies that for $T\gg K$,
\begin{eqnarray*}
\Big|\big \{c\in\Cmc\Gmc(S'):i(c,\nu)\leq T\big\}\Big|&\leq&\frac{1}{2}\sum_{i=1}^{\left\lfloor\frac{T}{K}\right\rfloor}\frac{(24g-24+12n)\cdot 4^{i-1}}{i}\cdot{\left\lfloor\frac{T-Ki}{L}\right\rfloor+i\choose i}\\
&\leq&\frac{1}{2}(24g-24+12n)\cdot 4^{\left\lfloor\frac{T}{K}\right\rfloor-1}\cdot{\left\lfloor\frac{T-KQ}{L}\right\rfloor+Q\choose Q},
\end{eqnarray*}
where $Q=Q(T,K,L)\in\{0,\dots,\left\lfloor\frac{T}{K}\right\rfloor\}$ is the integer so that for all $i\in\{0,\dots,\left\lfloor\frac{T}{K}\right\rfloor\}$,
\[{\left\lfloor\frac{T-KQ}{L}\right\rfloor+Q\choose Q}\geq {\left\lfloor\frac{T-Ki}{L}\right\rfloor+i\choose i}.\]

As a consequence, we have
\begin{eqnarray*}
h_\nu(S')\cdot K&\leq& \log(4)+\limsup_{T\to\infty}\frac{K}{T}\log{\left\lfloor\frac{T-KQ}{L}\right\rfloor+Q\choose Q}\\
&\leq&\log(4)+1+\log\left(1+\frac{1}{x_0}\right),
\end{eqnarray*}
where the last inequality is a computation that we do in Appendix \ref{main computation} (see Proposition \ref{combination}). Since $K=\frac{K_\nu(S')}{400\cdot 3^{3g-3+n}+96}$, we have proven the lemma.
\end{proof}

\subsection{Corollaries of the systolic inequality}
Theorem \ref{main theorem} has several interesting corollaries, which we will now explain. The first is a slight simplification of the inequality in Theorem \ref{main theorem} from which we can deduce all our other corollaries.

\begin{cor}\label{useful corollary}
There is a constant $C\in\Rbbb^+$ which depends only on the topology of $S'$, so that for any period minimizing $\nu\in\Cmc(S)$, we have 
\[\frac{1}{4}\log(2)\leq h_\nu(S')K_\nu(S')\leq C\cdot\left(\log(4)+1+\log\left(1+\frac{\sqrt{5}+1}{2}\cdot\frac{K_\nu(S')}{L_\nu(S')}\right)\right).\]
\end{cor}

\begin{proof}
Let $a:=\left\lceil\frac{K_\nu(S')}{L_\nu(S')}-1\right\rceil$, and consider the function \mbox{$f_a:[0,\infty)\to\Rbbb$} defined by $f_a(x)=(1+x)^a\cdot x$. Observe that $f_a$ is increasing, $f_a(0)=0$, and $\lim_{x\to\infty} f_a(x)=\infty$. Also, let $x_0=x_0(a)$ be the unique point in $[0,1)$ so that $f_a(x_0)=1$. It is sufficient to show that for all $a\geq 0$,
\[\frac{1}{x_0(a)}\leq\frac{\sqrt{5}+1}{2}\cdot\frac{K_\nu(S')}{L_\nu(S')}.\]

First, consider the case when $a=0$. Then $\frac{K_\nu(S')}{L_\nu(S')}=1$ and $x_0(0)=1$. We see immediately that in this case, the required inequality holds.

Next, consider the case when $a\geq 1$. The equation $f_a(x_0)=1$ can be rearranged as
\[a\cdot x_0=-\frac{x_0\log(x_0)}{\log(1+x_0)}.\]
Since the function $g:(0,1)\to\Rbbb$ given by $g(x)=-\frac{x\log(x)}{\log(1+x)}$ is positive and strictly decreasing, we see that $a\cdot x_0(a)$ is minimized over all $a\geq 1$ when $x_0(a)$ is maximized. From the definition of $f_a$, it is clear that $x_0(a)$ is strictly decreasing with $a$, so $a\cdot x_0(a)$ is minimized over all $a\geq 1$ when $a=1$. It is easy to compute that $x_0(1)=\frac{\sqrt{5}-1}{2}$, so $a\cdot x_0(a)\geq \frac{\sqrt{5}-1}{2}$ for all $a\geq 1$. Hence,
\[\frac{1}{x_0(a)}\leq \frac{\sqrt{5}+1}{2}\cdot\left\lceil\frac{K_\nu(S')}{L_\nu(S')}-1\right\rceil\leq\frac{\sqrt{5}+1}{2}\cdot\frac{K_\nu(S')}{L_\nu(S')}.\qedhere\]
\end{proof}

Using Corollary \ref{useful corollary}, we have the following universal upper bound on the systole length renormalized by the entropy. 

\begin{cor}\label{non-closed}
There is a constant $C\in\Rbbb^+$ which depends only on the topology of $S'$, so that for any period minimizing $\nu\in\Cmc(S)$, we have 
\[h_\nu(S')L_\nu(S')\leq C.\]
\end{cor}

\begin{proof}
By Corollary \ref{useful corollary}, we see that there is a constant $C'$ depending only on the topology of $S'$, so that for any period minimizing $\nu\in\Cmc(S)$,
\begin{align*}
h_\nu(S')L_\nu(S')&\leq C'\cdot\left((\log(4)+1)\cdot\frac{L_\nu(S')}{K_\nu(S')}+\frac{L_\nu(S')}{K_\nu(S')}\cdot\log\left(1+\frac{\sqrt{5}+1}{2}\cdot\frac{K_\nu(S')}{L_\nu(S')}\right)\right)\\
&\leq C'\cdot\left(\log(4)+1+\frac{\sqrt{5}+1}{2}\right)=:C, 
\end{align*}
where the last inequality holds because $x\log(1+\frac{k}{x})\leq k$ for all $x>0$ and $k>0$.
\end{proof}

Corollary \ref{non-closed} together with Theorem \ref{positively ratioed to geodesic currents} proves Corollary \ref{easy inequality}. Also, given any negatively curved Riemannian metric $m'$ on $S'$ with geodesic boundary, one can always find a negatively curved Riemannian metric $m$ on $S$ whose restriction to $S'$ is $m'$. It is then known (see Proposition 3 of Otal \cite{Ota2}) that the Lebesgue-Liouville current $\nu_m$ of $m$ has the property that $i(\nu_m,c)=\ell_m(c)$, where $\ell_m:\Cmc\Gmc(S)\to\Rbbb$ is the length function induced by $m$. This fact combined with Corollary \ref{non-closed} then gives us Corollary \ref{Riemannian}.

Corollary \ref{useful corollary} also gives us a criterion that determines when the topological entropy of  a sequence of geodesic currents in the ``$\epsilon$-thick" part of $\Cmc(S)$ converges to $0$. Before we state the corollary, we first define what the ``$\epsilon$-thick" part of $\Cmc(S)$ is.

\begin{definition}
Let $\Cmc(S)^{\min}\subset\Cmc(S)$ be the set of period minimizing geodesic currents and let $\epsilon>0$. Define 
\[\Cmc(S')^{\min}_\epsilon:=\{\nu\in \Cmc(S)^{\min}:L_\nu(S')\geq\epsilon\}\]
and $\Mmc(S')_\epsilon:= \Cmc(S')^{\min}_\epsilon/MCG(S')$.
\end{definition}

Observe that if $\mu,\nu\in\Cmc(S')^{\min}_\epsilon$ lie in the same equivalence class in $\Mmc(S')_\epsilon$, then $h_\mu(S')=h_\nu(S')$. Thus, we can think of $h_{-}(S')$ as a function from $\Mmc(S')_\epsilon$ to $\Rbbb$.

\begin{cor}\label{entropy degeneration corollary}
Let $\epsilon$ be any positive number and let $\{[\nu_k]\}_{k=1}^\infty$ be a sequence in $\Mmc(S')_\epsilon$. Then $\lim_{k\to\infty}h_{\nu_k}(S')=0$ if and only if the following condition holds: for each $k$, there is a (possibly empty) collection $\Dmc_k$ of pairwise non-intersecting simple closed geodesics in $S'$ so that 
\begin{itemize}
\item $\displaystyle\sup_{k}\max\{i(c,\nu_k):c\in\Dmc_k\}<\infty$, and 
\item $\displaystyle\lim_{k\to\infty}\min\{i(c,\nu_k):c\in\Cmc\Gmc(S'\setminus\Dmc_k) \text{ is non-peripheral}\}=\infty.$
\end{itemize}
\end{cor}

\begin{proof} Let us choose an appropriate collection $\Dmc_k$ of closed geodesics. For each $\nu_k$, let $\Pmc_{\nu_k}(S')=\{c_{1,k},\dots,c_{3g-3+2n,k}\}$ be a minimal pants decomposition, where $c_{1,k},\dots,c_{3g-3+n,k}$ are non-peripheral and enumerated so that $i(c_{j,k},\nu_k)\leq i(c_{j+1,k},\nu_k)$. Let $j_0\in\{0,\dots,3g-3+n\}$ be the number so that
\begin{itemize}
 \item $\limsup_{k\to\infty}i(c_{j,k},\nu_k)<\infty$ for all $j\leq j_0$,
 \item $\limsup_{k\to\infty}i(c_{j,k},\nu_k)=\infty$ for all $j\in\{j_0+1,\dots,3g-3+n\}$. 
 \end{itemize}
 (We use the convention $j_0=0$ if $\limsup_{k\to\infty}i(c_{1,k},\nu_k)=\infty$ and $j_0=3g-3+n$ if $\limsup_{k\to\infty}i(c_{3g-3+n,k},\nu_k)<\infty$.) Let
 \[\Dmc_k:=\left\{\begin{array}{ll}
 \{c_{1,k}\dots,c_{j_0,k}\}&\text{if }j_0>0,\\
 \emptyset&\text{if }j_0=0,
 \end{array}\right.\]
 
First, we show that if the condition does not hold, then $\limsup_{k\to\infty}h_{\nu_k}(S')>0$. Since the condition does not hold, there is a constant $C$ so that for each $k$, there is a component $S_k''$ of $S'\setminus\Dmc_k$ and a non-peripheral primitive closed geodesic $d_k\in\Cmc\Gmc(S_k'')$ satisfying $i(d_k,\nu_k)\leq C$. Notice that when $j_0=3g-3+n$, $S_k''$ is a pair of pants which implies that $d_k$ is primitive and non-simple. Likewise, if $j_0<3g-3+n$, for sufficiently large $k$, $d_k$ is primitive and non-simple because $\lim_{k\to\infty}i(c_{j,k},\nu_k)=\infty$ for all $j>j_0$. In either case, $K_{\nu_k}(S''_k)\leq i(d_k,\nu_k)\leq C$ for all $k$. Thus, $h_{\nu_k}(S''_k)\geq\frac{\log(2)}{4C}$ by Corollary \ref{useful corollary}. Since $h_{\nu_k}(S')\geq h_{\nu_k}(S''_k)$, we see that $\limsup_{k\to\infty}h_{\nu_k}(S')>0$. 

Next, we suppose that the condition holds and we prove that $\lim_{k\to\infty}h_{\nu_k}(S')=0$. Observe that for $k$ big enough, the curves in $\mathcal D_k$ are part of a minimal pants decomposition $\Pmc_{\nu_k}(S')$. Let $d_k\in\Cmc\Gmc(S')$ be a closed geodesic that is not a multiple of an element in $\Pmc_{\nu_k}(S')$ so that $i(d_k,\nu_k)=K_{\nu_k}(S')$. By Corollary \ref{useful corollary}, it is sufficient to show that $\lim_{k\to\infty}i(d_k,\nu_k)=\infty$. Choose a hyperbolic structure $\Sigma$ on $S$, set $A:=\sup_{k}\max\{i(c,\nu_k):c\in\Dmc_k\}$, and set $B_k:=\min\{i(c,\nu_k):c\in\Cmc\Gmc(S'\setminus\Dmc_k) \text{ is non-peripheral}\}$. Since the condition holds, $0<A<\infty$ and $\lim_{k\to\infty}B_k=\infty$.

If $d_k\in\Cmc\Gmc(S'\setminus\Dmc_k)$, then it is non-peripheral, so
\[
i(d_k,\nu_k)\geq B_k.
\] 
On the other hand, if $d_k$ does not lie in $\Cmc\Gmc(S'\setminus\Dmc_k)$, let $S''_k$ be a component of $S'\setminus\Dmc_k$ that intersects $d_k$. Let $\gamma_k\in\Gamma'$ be the group element so that $[[\gamma_k]]=d_k$ and let $L_{\gamma_k}$ be the axis of $\gamma_k$ in $\Std'$. Choose distinct points $p_k,q_k\in L_{\gamma_k}$ so that $\Pi(p_k),\Pi(q_k)\in\partial S_k''$ and $\Pi(r)\notin\partial S_k''$ for all $r\in[p_k,q_k)$. It is clear that 
\[
\nu\big(G(p_k,q_k]\big),\nu\big(G(q_k,p_k]\big)\leq i(d_k,\nu_k).
\]

Let $\eta_k,\tau_k\in\pi_1(S_k'')$ be the group elements so that $[[\eta_k]], [[\tau_k]]\in\Dmc_k$, and $p_k$ and $q_k$ lie in the axes of $\eta_k$ and $\tau_k$ respectively. Observe that $i\big([[\eta_k]],\nu_k\big),i\big([[\tau_k]],\nu_k\big)\leq A$. Hence, by Lemma \ref{compute length},
\begin{align*}
i\big([[\tau_k\eta_k]],\nu_k\big)&\leq \nu_k\big(G(\eta_k^{-1}\cdot p_k,\tau_k\eta_k\cdot(\eta_k^{-1}\cdot p_k)]\big)\\
&\leq \nu_k\big(G(\eta_k^{-1}\cdot p_k,p_k]\big)+\nu_k\big(G(p_k,q_k]\big)+\nu_k\big(G(q_k,\tau_k\cdot q_k]\big)\\
&\hspace{2cm}+\nu_k\big(G(\tau_k\cdot q_k,\tau_k\cdot p_k]\big)\\
&= i\big([[\eta_k]],\nu_k\big)+i\big([[\tau_k]],\nu_k\big)+\nu_k\big(G(p_k,q_k]\big)+\nu_k\big(G[p_k,q_k)\big)\\
&\leq 2A+2i(d_k,\nu_k).
\end{align*}

Similarly, $i\big([[\tau_k^{-1}\eta_k]],\nu_k\big)\leq 2A+2i(d_k,\nu_k)$. Since $[[\tau_k\eta_k]],[[\tau_k^{-1}\eta_k]]\in\Cmc\Gmc(S_k'')$ cannot both be peripheral, either $i\big([[\tau_k\eta_k]],\nu_k\big)\geq B_k$ or $i\big([[\tau_k^{-1}\eta_k]],\nu_k\big)\geq B_k$. Thus, 
\[i(d_k,\nu_k)\geq \frac{1}{2}B_k-A.\]
This implies that $\lim_{k\to\infty}i(d_k,\nu_k)\geq\frac{1}{2}\lim_{k\to\infty}B_k-A=\infty$. 
\end{proof}

Consider the case when $S'=S$. In the above theorem, each $\Dmc_k$ can be completed to a pants decomposition of $S$. Since there are only finitely many mapping class group orbits of pants decompositions of $S$, we can apply Corollary \ref{entropy degeneration corollary} and Theorem \ref{positively ratioed to geodesic currents} to deduce Corollary \ref{mapping class}.

\appendix

\section{From positive cross ratios to geodesic currents}\label{From cross ratios to geodesic currents}

In this appendix, we give a proof Theorem \ref{cross ratio intersection} (which was previously observed by Hamenst\"adt \cite{Ham2}) for the convenience of the reader. Recall that an \emph{algebra} $\Amc$ on a set $X$ is a family of subsets of $X$ so that
\begin{enumerate}
	\item for all $A$ in $\mathcal A$, the complement $A^c$ of $A$ is in $\mathcal A$,
	\item for all $A_1,A_2\in\mathcal A$, the union $A_1\cup A_2\in\mathcal A$.
\end{enumerate}
Also, a \emph{premeasure} $\nu$ on $\Amc$ is a function
\[
\nu\colon\Amc\to[0,\infty]
\]
such that 
\begin{enumerate}
	\item $\nu(\emptyset)=0$;
	\item if $\{A_i\}_{i\in\mathbb N}$ is a countable family of pairwise disjoint sets in $\Amc$ whose union lies in $\Amc$, then
\[
\nu\left(\bigcup_{i\in\mathbb N} A_i\right)=\sum_{i\in\mathbb N}\nu(A_i).
\]
\end{enumerate}

Choose an orientation on $\partial\Gamma$. For the rest of this appendix, we will assume that all intervals in $\partial\Gamma$ are of the form $[x,y)_z$ (see Notation \ref{interval notation 2}) unless otherwise stated, where $x,y,z$ lie in $\partial\Gamma$ in the order specified by the orientation. We will start by defining a particular algebra in $\Gmc(\Std)$.

\begin{definition}\label{def:sets}
Let $\mathcal L=\{\{I_1,J_1\},\{I_2,J_2\},\dots \{I_n,J_n\}\}$ be a finite list of pairs of proper subintervals of $\partial \Gamma$. 
\begin{itemize}
	\item $\mathcal L$ is an \emph{admissible list} if for all $k=1,\dots, n$, either $I_k=J_k$ or $I_k$ and $J_k$ are disjoint. 
	\item For any admissible list $\mathcal L$, let $\Gmc_\mathcal L$ denote the set of geodesics $\{a,b\}\in \Gmc(\Std)$ such that there exists a $k\in\{1,\dots,n\}$ with $I_k$ and $J_k$ each containing one endpoint of $\{a,b\}$.
\end{itemize}
\end{definition}

For the rest of this appendix, let $\Amc:=\{\Gmc_\Lmc:\Lmc\text{ is an admissible list}\}$. Note that every $\Gmc_\Lmc\in\Amc$ can be written as a finite disjoint union 
\[\Gmc_\Lmc=\bigcup_{k=1}^n\Gmc_{\{I_k,J_k\}}.\]

\begin{lem} $\Amc$ is an algebra.
\end{lem}

\begin{proof} For all $\Gmc_{\mathcal L},\Gmc_{\mathcal L'}\in\Amc$, note that $\Gmc_{\mathcal L}\cup\Gmc_{\mathcal L'}=\Gmc_{\mathcal L\cup \mathcal L'}\in\Amc$. To prove closure under complements, it is sufficient to show that $\Gmc_{\{I,J\}}^c\in\Amc$ and $\Gmc_{\{I_1,I_2\}}\cap\Gmc_{\{J_1,J_2\}}\in\Amc$ because of De Morgan's laws.

First, we will show that $\Gmc_{\{I,J\}}^c\in\Amc$. If $I$ and $J$ have disjoint interiors, let $x,y$ be the endpoints of $I$ and $z,w$ be the endpoints of $J$. Also, let $K:=[y,z)_x$ and $L:=[w,x)_y$, and observe that $\partial\Gamma=K\cup L\cup I\cup J$ is a disjoint union. It is then easy to see that $(\Gmc_{\{I,J\}})^c=\Gmc_{\Lmc}$, where
\begin{align*}
\Lmc&:=\{\{I,I\},\{I,K\},\{I,L\},\{K,K\},\{K,J\},\{K,L\},\{J,J\},\{J,L\},\{L,L\}\}.
\end{align*}
On the other hand, if $I=J$, let $K:=\partial \Gamma-I$. Then $(\Gmc_{\{I,J\}})^c=\Gmc_{\Lmc}$, where $\Lmc:=\{\{I,K\},\{K,K\}\}$. 

Next, we will show that $\Gmc_{\{I_1,I_2\}}\cap\Gmc_{\{J_1,J_2\}}\in\Amc$. Let $K_{i,j,1},\dots,K_{i,j,t_{i,j}}$ be intervals so that $I_i\cap J_j=\bigcup_{k=1}^{t_{i,j}}K_{i,j,k}$ is a disjoint union. (Note that $t_{i,j}$ is either $0$, $1$ or $2$.) Then observe that $\Gmc_{\{I_1,I_2\}}\cap\mathcal \Gmc_{\{J_1,J_2\}}=\Gmc_{\Lmc}$, where
\begin{align*}
\Lmc&:=\big\{\{K_{1,1,1},K_{2,2,1}\},\{K_{1,1,t_{1,1}},K_{2,2,1}\},\{K_{1,1,1},K_{2,2,t_{2,2}}\},\{K_{1,1,t_{1,1}},K_{2,2,t_{2,2}}\},\\
&\hspace{1.3cm}\{K_{1,2,1},K_{2,1,1}\},\{K_{1,2,t_{1,2}},K_{2,1,1}\},\{K_{1,2,1},K_{2,1,t_{2,1}}\},\{K_{1,2,t_{1,2}},K_{2,1,t_{2,1}}\}\big\}.
\end{align*}
\end{proof}

Next, we will use the positive cross ratio $B$ to define a premeasure on $\Amc$. Let $\nu_B\colon \Amc\to [0,\infty]$ be a function defined as follows:
\begin{enumerate}
\item $\nu_B(\emptyset)=0$.
\item $\nu_B(\Gmc_{\{I,J\}})=\infty$ when $I=J$ or $I$ and $J$ share a common endpoint.
\item If the intervals $I$ and $J$ are non-empty and have disjoint closures, let $x$ and $y$ be the endpoints of $I$ and $z$ and $w$ the endpoints of $J$, so that $x$, $y$, $z$, $w$ lie in this cyclic order along $\partial\Gamma$. Then define 
\[\nu_B(\Gmc_{\{I,J\}}):=B(x,y,z,w).\]
\item If $\Gmc_\Lmc=\bigcup_{k=1}^n \Gmc_{\{I_k,J_k\}}$ is a disjoint union, define
\[\nu_B(\Gmc_\Lmc):=\sum_{1\leq k\leq n} \nu_B(\Gmc_{\{I_k,J_k\}}).\]
\end{enumerate}

Since $B(x,y,z,w)=B(z,w,x,y)$, $\nu_B(\Gmc_{\{I,J\}})$ is well-defined. Also, if 
\[\Gmc_\Lmc=\bigcup_{k=1}^n \Gmc_{\{I_k,J_k\}}=\bigcup_{k=1}^m \Gmc_{\{I'_k,J'_k\}}\] 
are two ways to write $\Gmc_\Lmc$ as disjoint unions, then by taking intersections, we can write $\Gmc_\Lmc$ as the disjoint union 
\[\Gmc_\Lmc=\bigcup_{k=1}^l \Gmc_{\{I''_k,J''_k\}}\] 
so that for all $k=1,\dots,l$, $\Gmc_{\{I''_k,J''_k\}}$ is a connected component of $\Gmc_{\{I_s,J_s\}}\cap\Gmc_{\{I'_t,J'_t\}}$ for some $s\in\{1,\dots,n\}$ and $t\in\{1,\dots,m\}$.
By the additive property of the cross ratio, if $I_1, I_2,I,J$ are subintervals of $\partial\Gamma$ so that $I_1\cup I_2=I$, then $\nu_B(\Gmc_{\{I_1,J\}})+\nu_B(\Gmc_{\{I_2,J\}})=\nu_B(\Gmc_{\{I,J\}}).$ This implies that 
\[\sum_{k=1}^n\nu_B(\Gmc_{\{I_k,J_k\}})=\sum_{k=1}^l\nu_B(\Gmc_{\{I''_k,J''_k\}})=\sum_{k=1}^m\nu_B(\Gmc_{\{I'_k,J'_k\}}),\]
so $\nu_B$ is well-defined.

It is also clear from the definition that $\nu_B$ is finitely additive, i.e. if $\Gmc_{\Lmc_1},\dots,\Gmc_{\Lmc_n}$ are pairwise disjoint, then
\[\nu_B\left(\bigcup_{i=1}^n\Gmc_{\Lmc_i}\right)=\sum_{i=1}^n\nu_B\left(\Gmc_{\Lmc_i}\right).\]
Furthermore, the positivity of $B$, ensures that $\nu_B$ takes values in $[0,\infty]$. 

\begin{prop} For any positive cross ratio $B$, $\nu_B$ is a premeasure on $\Amc$.
\end{prop}

\begin{proof} Set $\nu:=\nu_B$ to simplify the notation. By definition, $\nu(\emptyset)=0$, so we need only to prove countable additivity. Let $\{\Gmc_{\Lmc_k}\}_{k=1}^\infty$ be a family of disjoint sets in $\Amc$ with $\bigcup_{k=1}^\infty \Gmc_{\Lmc_k}=\Gmc_\Lmc$ for some admissible list $\Lmc$. Up to repartitioning and renumbering, we can assume the following:
\begin{itemize}
\item For all $k$, $\Lmc_k=\{I_k,J_k\}$ for some intervals $I_k,J_k$, 
\item There is some admissible list $\{\{I'_1,J'_1\},\dots,\{I'_t,J'_t\}\}$ so that $\Gmc_\Lmc=\bigcup_{s=1}^t\Gmc_{\{I'_s,J'_s\}}$ is a disjoint union.
\end{itemize}
By finite additivity, we have that for all $n\in\Zbbb^+$,
\begin{align*}
\nu(\Gmc_{\Lmc})&=\nu\left(\bigcup_{k=1}^n \Gmc_{\Lmc_k}\right)+\nu\left(\Gmc_{\Lmc}\setminus \bigcup_{k=1}^n \Gmc_{\Lmc_k}\right)\\
&\geq\nu\left(\bigcup_{k=1}^n\Gmc_{\Lmc_k}\right)=\sum_{k=1}^n\nu(\Gmc_{\Lmc_k}).
\end{align*}
Thus, $\nu(\Gmc_{\Lmc})\geq \sum_{k=1}^\infty\nu(\Gmc_{\Lmc_k})$. 

To finish the proof, we need to show that $\nu(\Gmc_{\Lmc})\leq \sum_{k=1}^\infty\nu(\Gmc_{\Lmc_k})$. First consider the case where $\nu(\Gmc_{\Lmc})<\infty$, then $\nu(\Gmc_{\{I_s',J_s'\}}),\nu(\Gmc_{\Lmc_k})<\infty$ for all $s=1,\dots,t$, $k\in\Zbbb^+$. Since $B$ is continuous, for any $\varepsilon>0$ and any $s=1,\dots,t$, we can find compact subintervals $I_s''\subset I_s'$ and $J_s''\subset J_s'$ such that
\begin{equation}\label{eqn:finite_cont}
\nu(\Gmc_{\{I_s',J_s'\}})-\nu(\Gmc_{\{I_s'',J_s''\}})<\frac{\varepsilon}{t}.
\end{equation}
Similarly, for any $k\in \mathbb N$ we can find open intervals $I^*_k\supset I_k$  and $J^*_k\supset J_k$ such that
\begin{equation}\label{eqn:infinite_cont}
\nu(\Gmc_{\{I^*_k,J^*_k\}})-\nu(\Gmc_{\Lmc_k})<\frac{\varepsilon}{2^k}.
\end{equation}

Observe that the open sets $\{\Gmc_{\{I^*_k,J^*_k\}}\}_{k=1}^\infty$ is an open cover of the compact set $\bigcup_{s=1}^t\Gmc_{\{I_s'',J_s''\}}$, so it has a finite subcover $\{\Gmc_{\{I^*_k,J^*_k\}}\}_{k=1}^N$. By using inequalities (\ref{eqn:finite_cont}) and (\ref{eqn:infinite_cont}), we have
\begin{align*}
\nu(\Gmc_{\Lmc})&<\varepsilon + \nu\left(\bigcup_{s=1}^t\Gmc_{\{I_s'',J_s''\}}\right)\leq\varepsilon +\sum_{k=1}^N\nu(\Gmc_{\{I^*_k,J^*_k\}})\\
&\leq\varepsilon +\sum_{k=1}^N\left(\nu(\Gmc_{\Lmc_k})+\frac{\varepsilon}{2^k}\right)\leq 2\varepsilon +\sum_{k=1}^\infty\nu(\Gmc_{\Lmc_k}).
\end{align*}
Since $\varepsilon$ was arbitrary, this proves that $\nu(\Gmc_{\Lmc})\leq \sum_{k=1}^\infty\nu(\Gmc_{\Lmc_k})$ when $\nu(\Gmc_{\Lmc})<\infty$. 

Next, consider the case where $\nu(\Gmc_{\Lmc})=\infty$. This means that $\nu(\Gmc_{\{I_{s_0}',J_{s_0}'\}})=\infty$ for some $s_0=1,\dots,t$, so either $I_{s_0}'$ and $J_{s_0}'$ are disjoint and share an endpoint, or $I_{s_0}'=J_{s_0}'$. In either case, we can find subintervals $\bar{I}\subset I_{s_0}'$ and $\bar{J}\subset J_{s_0}'$ that are disjoint and share an endpoint. Let $p$ be the common endpoint of $\bar{I}$ and $\bar{J}$, let $p_{\bar{J}}$ be the endpoint of $\bar{J}$ that is not $p$, and for any $q\in\bar{J}$ let $\bar{J}_q$ be the subinterval of $\bar{J}$ with endpoints $q$ and $p_{\bar{J}}$. Since 
\[\lim_{q\to p_{\bar{J}}}\nu(\Gmc_{\{\bar{I},\bar{J}_q\}})=0,\,\,\,\lim_{q\to p}\nu(\Gmc_{\{\bar{I},\bar{J}_q\}})=\infty,\] 
we know that for every $t\in\Rbbb$, there exists $q_t$ in the interior of $\bar{J}$ such that $\nu(\Gmc_{\{\bar{I},\bar{J}_{q_t}\}})=t$. Using the previous case, we know that for all $q\in\bar{J}$,
\[
\nu(\Gmc_{\{\bar{I},\bar{J}_q\}})\leq \sum_{k=1}^\infty \nu(\Gmc_{\Lmc_k}).
\]
Thus, $\sum_{k=1}^\infty \nu(\Gmc_{\Lmc_k})=\infty=\nu(\Gmc_{\Lmc})$.
\end{proof}

We will now recall a standard procedure to obtain a unique measure extending a premeasure. See for example Chapter 1 of Folland \cite{Fol1} for more details. Given a premeasure $\nu$ on $\mathcal A$ and $E\subset \Gmc(\tilde{S})$, define the \emph{outer measure}
\[
\nu^*(E)=\inf\left\{\sum_{k=1}^\infty\nu(A_k)\colon A_k\in\Amc,\ E\subset\bigcup_{k=1}^\infty A_k \right\}.
\]
A premeasure $\nu\colon \Amc\to [0,\infty]$ is \emph{$\sigma$-finite} if $X$ can be written as a union of countably many sets with finite outer measure. The following theorem (Theorem 1.14 of \cite{Fol1}) relates $\sigma$-finite premeasures and measures.

\begin{thm}\label{thm:extension_measure}
Let $\mathcal M$ be the $\sigma$-algebra generated by $\mathcal A$. Then, $\mu:=\nu^*_{\vert_\mathcal M}$ is a measure on $\mathcal M$ and the restriction of $\mu$ to $\Amc$ is $\nu$. If $\nu$ is $\sigma$-finite, then $\mu$ is unique.
\end{thm}

\begin{lem}\label{lem:sigmafinite} If $B$ is a positive cross ratio, then the premeasure $\nu_B$ is $\sigma$-finite.
\end{lem}

\begin{proof} Let $\Delta$ be a countable dense subset of $\partial \Gamma$. For any $m\in\Nbbb$ and $p\in\Delta$, let $I_m^p$ be the open interval in $\partial\Gamma$ centered at $p$ of width $1/m$. Define 
\[
\Delta^2_m:=\{(p,q)\in\Delta^2\colon I_m^p \mbox{ and }I_m^q \mbox{ have disjoint closures and } p \mbox{ preceeds }q\}.
\] 
Clearly
\begin{align}\label{ctblcovering}
\Gmc(\Std)=\bigcup_{m=1}^\infty\bigcup_{(p,q)\in\Delta_m^2}\Gmc_{\{I_m^p,I_m^q\}}.
\end{align}
Observe that the right hand side of Equation \ref{ctblcovering} is a countable union of sets, each with finite measure. 
\end{proof}

\begin{proof}[Proof of Theorem \ref{cross ratio intersection}]
Since the premeasure $\nu_B$ is $\sigma$-finite, Theorem \ref{thm:extension_measure} ensures that there is a unique measure $\mu_B$ on the $\sigma$-algebra $\Mmc$ generated by $\Amc$. The $\Gamma$-invariance of $B$ ensures that $\mu_B$ is also $\Gamma$-invariant. It is easy to see that the topology on $\Gmc(\Std)$ lies in $\Mmc$. Thus, $\mu_B$ is a geodesic current. 

Next, we show that for all $c\in\Cmc\Gmc(S)$, $i(c,\mu_B)=\ell_B(c)$. By (1) of Lemma \ref{compute length}, we know that 
\begin{align*}
i(c,\mu_B)&=\mu_B(\Gmc_{\{[\gamma^+,\gamma^-)_z,[z,\gamma\cdot z)_{\gamma^-}\}})\\
&=\nu_B(\Gmc_{\{[\gamma^+,\gamma^-)_z,[z,\gamma\cdot z)_{\gamma^-}\}})\\
&=B(\gamma^-,\gamma^+,\gamma\cdot z,z)\\
&=\ell_B(c).\qedhere
\end{align*}
\end{proof}
\section{Computation for proof of Theorem \ref{main theorem}}\label{main computation}

The goal of this appendix is to prove the following statement, which we need to finish the proof of Theorem \ref{main theorem}. 

\begin{prop}\label{combination}
Let $T,K,L$ be positive numbers so that $K\geq L$. Also, let $Q=Q(T,K,L)\in\{0,\dots,\left\lfloor\frac{T}{K}\right\rfloor\}$ be an integer so that for all $i\in\{0,\dots,\left\lfloor\frac{T}{K}\right\rfloor\}$, we have
\[{\left\lfloor\frac{T-KQ}{L}\right\rfloor+Q\choose Q}\geq {\left\lfloor\frac{T-Ki}{L}\right\rfloor+i\choose i}.\] Then
\[\limsup_{T\to\infty}\frac{K}{T}\log{\left\lfloor\frac{T-KQ}{L}\right\rfloor+Q\choose Q}\leq 1+\log\left(1+\frac{1}{x_0}\right),\]
where $x_0$ is the unique positive solution to the equation $(1+x)^{\left\lceil\frac{K}{L}-1\right\rceil}x=1$.
\end{prop}

The proof of Proposition \ref{combination} is a refinement of the argument given in Appendix B of \cite{Zha1}. However, we will make this appendix self-contained for the convenience of the reader. The main tool in this computation is an old result known as Stirling's formula, which we state here.

\begin{thm}[Stirling's Formula]
$\displaystyle\lim_{n\to\infty}\frac{n!}{\left(\frac{n}{e}\right)^n\sqrt{2\pi n}}=1$.
\end{thm}

In order to use Stirling's formula, we need to prove the following lemma. For the rest of this appendix, let 
\[F=F(T,K,L):=\left\lfloor\frac{T-KQ}{L}\right\rfloor.\] 

\begin{lem}\label{F and Q}
Let $K\geq L>0$ be fixed numbers, $a:=\left\lceil\frac{K}{L}-1\right\rceil$ and $b:=\left\lfloor\frac{K}{L}-1\right\rfloor$. Then the following hold:
\begin{enumerate}[(1)]
\item $\displaystyle\lim_{T\to\infty}F=\infty$.
\item $\displaystyle\lim_{T\to\infty}Q=\infty$.
\item $\displaystyle 1\leq\left(1+\liminf_{T\to\infty}\frac{Q}{F}\right)^a\cdot \liminf_{T\to\infty}\frac{Q}{F}$. 
\item $\displaystyle 1\geq\left(1+\limsup_{T\to\infty}\frac{Q}{F}\right)^b\cdot \limsup_{T\to\infty}\frac{Q}{F}$. 
\end{enumerate}
\end{lem}

\begin{proof}
From the definition of $Q$, we see that
\begin{align}\label{inequality 1}
1&\leq{\lfloor\frac{T-QK}{L}\rfloor+Q\choose Q}\Bigg/{\lfloor\frac{T-(Q+1)K}{L}\rfloor+Q+1\choose Q+1}\nonumber\\
&=\frac{(\lfloor\frac{T-QK}{L}\rfloor+Q)(\lfloor\frac{T-QK}{L}\rfloor+Q-1)\dots(\lfloor\frac{T-(Q+1)K}{L}\rfloor+Q+2)(Q+1)}{\lfloor\frac{T-QK}{L}\rfloor(\lfloor\frac{T-QK}{L}\rfloor-1)\cdots(\lfloor\frac{T-(Q+1)K}{L}\rfloor+1)},
\end{align}
which can be rearranged to be
\begin{align}\label{Qbound1}
\frac{Q+1}{F}&\geq\frac{(\lfloor\frac{T-QK}{L}\rfloor-1)(\lfloor\frac{T-QK}{L}\rfloor-2)\cdots(\lfloor\frac{T-(Q+1)K}{L}\rfloor+1)}{(\lfloor\frac{T-QK}{L}\rfloor+Q)(\lfloor\frac{T-QK}{L}\rfloor+Q-1)\dots(\lfloor\frac{T-(Q+1)K}{L}\rfloor+Q+2)}.
\end{align}

Similarly, the definition of $Q$ also tells us that
\begin{align}\label{inequality 2}
1&\geq{\lfloor\frac{T-(Q-1)K}{L}\rfloor+Q-1\choose Q-1}\Bigg/{\lfloor\frac{T-QK}{L}\rfloor+Q\choose Q}\nonumber\\
&=\frac{(\lfloor\frac{T-(Q-1)K}{L}\rfloor+Q-1)(\lfloor\frac{T-(Q-1)K}{L}\rfloor+Q-2)\dots(\lfloor\frac{T-QK}{L}\rfloor+Q+1)Q}{\lfloor\frac{T-(Q-1)K}{L}\rfloor(\lfloor\frac{T-(Q-1)K}{L}\rfloor-1)\cdots(\lfloor\frac{T-QK}{L}\rfloor+1)}
\end{align}
which implies 
\begin{align}\label{Qbound2}
\frac{Q}{F+1}&\leq\frac{(\lfloor\frac{T-(Q-1)K}{L}\rfloor)\cdots(\lfloor\frac{T-QK}{L}\rfloor+2)}{(\lfloor\frac{T-(Q-1)K}{L}\rfloor+Q-1)\dots(\lfloor\frac{T-QK}{L}\rfloor+Q+1)}\leq 1.
\end{align}

Proof of (1). Suppose for contradiction that $\displaystyle\liminf_{T\to\infty}F<\infty$. By the definition of $F$, we see that $\displaystyle\limsup_{T\to\infty}Q=\infty$. Thus,
\begin{equation*}
\limsup_{T\to\infty}\frac{Q}{F+1}=\infty,
\end{equation*}
which contradicts (\ref{Qbound2}).  

Proof of (2). Suppose again for contradiction that $\displaystyle\liminf_{T\to\infty}Q<\infty$. By taking an appropriate subsequence, we can assume that $\displaystyle\lim_{T\to\infty}Q<\infty$. Hence, $\displaystyle\lim_{T\to\infty} F=\infty$, so we have
\begin{equation*}
\lim_{T\to\infty}\frac{F}{Q+1}=\infty.
\end{equation*}
On the other hand, if $\displaystyle\lim_{T\to\infty}Q<\infty$ and $\displaystyle\lim_{T\to\infty} F=\infty$, then the right hand side of the inequality (\ref{Qbound1}) converges to $1$ as $T\to\infty$, which implies that 
\begin{equation*}
\lim_{T\to\infty}\frac{Q+1}{F}\geq 1.
\end{equation*}
This is a contradiction. 

Proof of (3). Since (1) and (2) hold, taking the limit infimum of (\ref{inequality 1}) as $T\to\infty$ gives
\begin{align*}
1\leq& \liminf_{T\to\infty}\left(1+\frac{Q}{\lfloor\frac{T-QK}{L}\rfloor}\right)\left(1+\frac{Q}{\lfloor\frac{T-QK}{L}\rfloor-1}\right)\dots\\
&\hspace{2cm}\left(1+\frac{Q}{\lfloor\frac{T-(Q+1)K}{L}\rfloor+2}\right)\cdot\frac{Q+1}{\lfloor\frac{T-(Q+1)K}{L}\rfloor+1}\\
\leq &\left(1+\liminf_{T\to\infty}\frac{Q}{F}\right)^a\cdot \liminf_{T\to\infty}\frac{Q}{F}.
\end{align*}

Proof of (4). Similarly, by taking limit supremum of (\ref{inequality 2}) as $T\to\infty$, we get
\begin{align*}
1\geq& \limsup_{T\to\infty}\left(1+\frac{Q-1}{\lfloor\frac{T-(Q-1)K}{L}\rfloor}\right)\left(1+\frac{Q-1}{\lfloor\frac{T-(Q-1)K}{L}\rfloor-1}\right)\dots\\
&\hspace{2cm}\left(1+\frac{Q-1}{\lfloor\frac{T-QK}{L}\rfloor+2}\right)\cdot\frac{Q}{\lfloor\frac{T-QK}{L}\rfloor+1}\\
\geq &\left(1+\limsup_{T\to\infty}\frac{Q}{F}\right)^b\cdot \limsup_{T\to\infty}\frac{Q}{F}.\qedhere 
\end{align*}
\end{proof}

By (3) and (4) of Lemma \ref{F and Q}, we see that $\displaystyle\limsup_{T\to\infty}\frac{F}{Q}$ is a positive real number, which we will denote by $D$ in the sequel. We now use (3) of Lemma \ref{F and Q} to find an inequality relating $D$ to $a:=\left\lceil\frac{K}{L}-1\right\rceil$.

\begin{lem}\label{D}
For any positive numbers $K\geq L>0$, let $a:=\left\lceil\frac{K}{L}-1\right\rceil$. Then 
\[D\leq\frac{1}{x_0},\]
where $x_0$ is the unique positive solution to the equation $(1+x)^ax=1$.
\end{lem}

\begin{proof}
Consider the function \mbox{$f_a:[0,\infty)\to\Rbbb$} defined by $f_a(x)=(1+x)^a\cdot x$, and observe that $f_a$ is increasing, $f_a(0)=0$, and $\lim_{x\to\infty} f(x)=\infty$. By (3) of Lemma \ref{F and Q}, we know that $f_a\left(\frac{1}{D}\right)\geq 1$, so $\frac{1}{D}\geq f_a^{-1}(1)=x_0$.
\end{proof}

With Lemma \ref{F and Q} and Lemma \ref{D}, we are now ready to prove Proposition \ref{combination}.

\begin{proof}[Proof of Proposition \ref{combination}]
Since (1) and (2) of Lemma \ref{F and Q} hold, we can apply Stirling's formula to obtain
\[\lim_{T\to\infty}\left({F+Q\choose Q}\cdot\sqrt{\frac{2\pi Q F}{F+Q}}\cdot\left(\frac{Q}{F+Q}\right)^Q\cdot\left(\frac{F}{F+Q}\right)^F\right)=1.\]
Taking the logarithm and multiplying by $\frac{K}{T}$ then gives an expression that can be rearranged to yield
\begin{align}\label{Stirling application}
\limsup_{T\to\infty}\frac{K}{T}\log{F+Q\choose Q}=&\limsup_{T\to\infty}\frac{K}{2T}\log\left(\frac{1}{Q}+\frac{1}{F}\right)+\limsup_{T\to\infty}\frac{KQ}{T}\log\left(1+\frac{F}{Q}\right)\nonumber\\
&+\limsup_{T\to\infty}\frac{KQ}{T}\cdot\frac{F}{Q}\log\left(1+\frac{Q}{F}\right)+\limsup_{T\to\infty}\frac{K}{2T}\log\left(\frac{1}{2\pi}\right)\\
=&\limsup_{T\to\infty}\frac{KQ}{T}\log\left(1+\frac{F}{Q}\right)+\limsup_{T\to\infty}\frac{KQ}{T}\cdot\frac{F}{Q}\log\left(1+\frac{Q}{F}\right).\nonumber
\end{align}

By the definition of $F$, we have
\[\limsup_{T\to\infty}\frac{Q}{F}=\frac{L}{K}\limsup_{T\to\infty}\frac{1}{\frac{T}{KQ}-1},\]
which implies that 
\[\limsup_{T\to\infty}\frac{KQ}{T}=\frac{K\cdot\displaystyle\limsup_{T\to\infty}\frac{Q}{F}}{L+K\cdot\displaystyle\limsup_{T\to\infty}\frac{Q}{F}}\leq 1.\]
Applying this to the inequality (\ref{Stirling application}) then gives
\begin{align*}
\limsup_{T\to\infty}\frac{K}{T}\log{F+Q\choose Q}&\leq\limsup_{T\to\infty}\log\left(1+\frac{F}{Q}\right)+\limsup_{T\to\infty}\frac{F}{Q}\log\left(1+\frac{Q}{F}\right)\\
&\leq\log(1+D)+1\\
&\leq\log\left(1+\frac{1}{x_0}\right)+1,
\end{align*}
where the second inequality is a consequence of the fact that $x\log(1+\frac{1}{x})\leq 1$ for all $x>0$, and the final inequality is Lemma \ref{D}.
\end{proof}


\begin{thebibliography}{99}

\bibitem{Bon1} F. Bonahon, {\em Bouts des vari\'et\'es hyperboliques de dimension 2}, Ann. Math. 124 (1986), 71-158.

\bibitem{Bon2} F. Bonahon, {\em The geometry of Teichm\"uller space via geodesic currents}, Invent. Math. 92 (1988), 139-162

\bibitem{BILW} M. Burger, A. Iozzi, F. Labourie, A. Wienhard {\em Maximal representations of surface groups: Symplectic Anosov structures}, Pure Appl. Math. Q. 1 (2005), Special Issue: In memory of Armand Borel, 543-590.

\bibitem{BIW} M. Burger, A. Iozzi, A. Wienhard {\em Surface group representations with maximal Toledo invariant}, Ann. of Math. 172 (2010), 517-566.

\bibitem{BurPoz1} M. Burger, M. Pozzetti {\em Maximal representations, non Archimedean Siegel spaces, and buildings}, arXiv:1509.01184.

\bibitem{CO1} J. L. Clerc, B. {\O}rsted {\em The Gromov norm of the Kaehler class and the Maslov index}, Asian J. Math. 7 (2003), no. 2, 269-295. 


\bibitem{DT1} A. Dominic, D. Toledo {\em The Gromov norm of the K\"ahler class of symmetric domains}, Math. Ann. 276 (1987), no. 3, 425-432.

\bibitem{Ebe1} P. Eberlein {\em Geometry of nonpositively curved manifolds}, Chicago Lectures in Math., Univ. of Chicago Press (1996).

\bibitem{Fol1} G.B. Folland {\em Real analysis}, Pure and Applied Mathematics (New York) (1999).


\bibitem{GGKW} F. Gu\'eritaud, O. Guichard, F. Kassel, A. Wienhard {\em Anosov representations and proper actions}, to appear in {\em Geometry and Topology}, Preprint arXiv:1502.03811.

\bibitem{Gui1} O. Guichard {\em Composantes de Hitchin et repr\'esentations hyperconvexes de groupes de surface}, J. Diff. Geom. 80 (2008), 391-431.

\bibitem{GW} O. Guichard, A. Wienhard {\em Anosov representations: domains of discontinuity and applications}, Invent. Math. 190 (2012), 357-438.

\bibitem{Gol1} W. M. Goldman {\em Topological components of spaces of representations}, Invent. Math. 93 (1988), no. 3, 557-607.

\bibitem{Ham1} U. Hamenst\"adt, {\em Cocycles, Hausdorff measures and cross ratios}, Ergod. Th. \& Dynam. Sys. 17 (1997), 1061-1081.

\bibitem{Ham2} U. Hamenst\"adt, {\em Cocycles, symplectic structures and intersection}, Geom. Funct. Anal. (1999), Volume 9, 90 - 140.

\bibitem{Hel1} S. Helgason {\em Differential Geometry, Lie groups and Symmetric Spaces}, Graduate Studies in Mathematics, Volume 34, American Mathematical Society (1978).

\bibitem{Hit1} N. J. Hitchin {\em Lie groups and Teichm\"uller space}, Topology 31 (1992), no. 3, 449-473.

\bibitem{Hum1} J. Humphreys {\em Introduction to Lie Algebras and Representation Theory}, Graduate Texts in Mathematics, Volume 9, Springer (1972).

\bibitem{KLP1} M. Kapovich, B. Leeb, J. Porti {\em Morse actions of discrete groups on symmetric space}, Preprint arXiv:1403.7671.

\bibitem{KLP2} M. Kapovich, B. Leeb, J. Porti {\em A Morse Lemma for quasigeodesics in symmetric spaces and euclidean buildings}, Preprint arXiv:1411.4176.

\bibitem{Lab1} F. Labourie, {\em Anosov flows, surface groups and curves in projective space}, Invent. Math. 165 (2006), 51-114.

\bibitem{Lab2} F. Labourie. {\em Cross ratios, surface groups, $PSL(n,\mathbb R)$ and diffeomorphisms of the circle}, Publ. Math. Inst. Hautes \'Etudes Sci. No.106 (2007), 139-213.

\bibitem{Lab3} F. Labourie {\em Cross Ratios, Anosov representations and the Energy Functional on Teichm\"uller space}, Ann. Sci. \'Ecole Norm. Sup. 41 (2008), Issue 3, (439-471).

\bibitem{Led2} F. Ledrappier {\em Structure au bord des vari\'et\'es \`a courbure n\'egative},  S\'eminaire de th\'eorie spectrale et g\'eom\'etrie de Grenoble, 71, (1994-1995), 97-122

\bibitem{Mil1} J. Milnor {\em On the existence of a connection with curvature zero}, Comm. Math. Helv. 32 (1958), 215-223.

\bibitem{Nie1} X. Nie {\em Entropy degeneration of convex projective surfaces}, Conformal Geometry and Dynamics. 19 (2015), 318-322.

\bibitem{Nie2} X. Nie {\em On the Hilbert geometry of simplicial Tits sets}, Annales de l'Institut Fourier 65 (2014), 1005-1030.

\bibitem{Ota1} J-P. Otal {\em Sur la g\'eom\'etrie symplectique de l'espace des g\'eod\'esiques d'une vari\'et\'e \`a courbure n\'egative}, Rev. Math. Iberoam. 8 (1992), 441-456.

\bibitem{Ota2} J-P. Otal {\em Le spectre marqu\'e des longueurs des surfaces \`a courbure n\'egative}, Ann. Math. 131 (1990), 151-162.

\bibitem{Qui1} J.F. Quint {\em Mesures de Patterson-Sullivan en rang sup\'erieur}, Geom. Funct. Anal. 12 (2002), 776-809.

\bibitem{Sam1} A. Sambarino {\em Quantitative properties of convex representations}, Comm. Math. Helv. 89 (2014), 443-488.

\bibitem{Sam2} A. Sambarino {\em Hyperconvex representations and exponential growth}, Ergod. Th. \& Dynam. Sys. 34 (2014), no. 3, 986-1010.

\bibitem{Sab1} S. Sabourau {\em Entropy and systoles on surfaces}, Ergod. Th. \& Dynam. Sys. 26 (2006), no. 5, 1653-1669.

\bibitem{Tit1} J. Tits {\em Repr\'esentations lin\'eaires irr\'eductibles d'un group r\'eductif sur un corp quelconque}, J. Reine Angew. Math. 247 (1971), 196-220.

\bibitem{Tur1} V. Turaev {\em A cocycle for the symplectic first Chern class and the Maslov index}, Funct. Anal. Appl. 18 (1984), 35-39.

\bibitem{Zha1} T. Zhang, {\em The degeneration of convex $\Rbbb\Pbbb^2$ structures on surfaces}, Proc. London Math. Soc. (2015), Volume 111, 967-1012.

\bibitem{Zha2} T. Zhang, {\em Degeneration of Hitchin representations along internal sequences}, Geom. Funct. Anal. (2015), Volume 25, 1588-1645.

\end{thebibliography}
\end{document}